 \documentclass[noinfoline,11pt]{imsart}
\usepackage{graphicx}
\usepackage{comment}
\usepackage{url} 

\usepackage{hyperref}
\newcommand{\blind}{0}

\usepackage{amsthm}
\usepackage{mathrsfs}
\usepackage{euscript}
\usepackage{mathtools}
\usepackage{amsmath}
\usepackage{amssymb}
\usepackage{amsfonts}
\usepackage[shortlabels]{enumitem}
\usepackage{xcolor}
\usepackage{bm}
\usepackage{subfig}
\usepackage{graphicx}
\usepackage{stmaryrd}

\usepackage[left=24mm, right=24mm, top=30mm, bottom=30mm]{geometry}

\newtheorem{theorem}{Theorem}

\newtheorem{example}[theorem]{Example}

\newtheorem{corollary}[theorem]{Corollary}

\newtheorem{definition}[theorem]{Definition}

\newtheorem{lemma}[theorem]{Lemma}

\newtheorem{proposition}[theorem]{Proposition}
\newtheorem{remark}[theorem]{Remark}

\newcommand{\rangleL}{{L^2(\mathbb{S}^2)}}
\newcommand{\rangleLL}{{L^2(\mathbb{S}^2 \times \mathbb{S}^2)}}
\newcommand{\rangleH}{{\mathcal{H}_p}}
\newcommand{\rangleHH}{{\mathbb{H}_p}} 
\newcommand{\Hp}{{\mathcal{H}_p}}
\newcommand{\Hq}{{\mathcal{H}_q}}
\newcommand{\Hprod}{{\mathbb{H}_p}}
\newcommand{\Hqprod}{{\mathbb{H}_q}} 
\newcommand{\pen}{{\eta}}
\newcommand{\green}{\Psi}
\newcommand{\Y}{Y}

\DeclareMathOperator*{\argmin}{arg\,min}
\newcommand{\lag}{{h}}
\newcommand{\W}{W}
\newcommand{\pio}{{\Pi_1(p,q)}}
\newcommand{\pit}{{\Pi_2(p,q)}}
\newcommand{\piotemp}{{\Pi^s_1(p,q)}}
\newcommand{\pittemp}{{\Pi^s_2(p,q)}}
\newcommand{\num}{{N_\lag}}
\newcommand{\sumjk}{{\mathop{ \sum_{j=1}^r \sum_{k=1}^r}_{j\ne k \text{ if } \lag = 0}}}
\newcommand{\R}{\mathbb{R}}

\newcommand{\Sph}{\mathbb{S}}
\newcommand{\CL}{\EuScript{C}}
\newcommand{\CHq}{\mathscr{C}}
\newcommand{\pu}{p_U}
\newcommand{\U}{{\mathcal{U}}}
\newcommand{\X}{{\mathcal{X}}}
\newcommand{\remainder}{\mathbin{\%}}

\usepackage{xcolor}
\definecolor{matthieu}{rgb}{0.8, 0.5, 0.2}
\definecolor{julien}{RGB}{0,127,119}
\definecolor{victor}{rgb}{0.75,0, 0}
\definecolor{alessia}{rgb}{0.96, 0.0, 0.63}

\begin{document}

\begin{frontmatter}

\title{Functional Estimation of Anisotropic Covariance and Autocovariance Operators on the Sphere}

\runtitle{Covariance and Autocovariance Operators on the Sphere}

\begin{aug}
\author{\fnms{Alessia} \snm{Caponera}\ead[label=e1]{alessia.caponera@epfl.ch}},   
\author{\fnms{Julien} \snm{Fageot}\ead[label=e2]{julien.fageot@epfl.ch}}, 
\author{\fnms{Matthieu}  \snm{Simeoni}\ead[label=e3]{matthieu.simeoni@epfl.ch}}
\and
\author{\fnms{Victor M.} \snm{Panaretos}\ead[label=e4]{victor.panaretos@epfl.ch}}
\thankstext{t1}{Alessia Caponera and Victor Panaretos acknowledge support from SNSF Grant 200020\_207367. Julien Fageot acknowledges support from SNSF Grant P400P2\_194364. Matthieu Simeoni acknowledges support from SNSF Grants CRSII5\_193826 and 200021\_181978/2 .}

\runauthor{Caponera, Fageot, Simeoni \& Panaretos}

\affiliation{\'Ecole Polytechnique F\'ed\'erale de Lausanne}

\address{\'Ecole Polytechnique F\'ed\'erale de Lausanne\\ \printead{e1,e2,e3,e4} }

\end{aug}

\begin{abstract} 
We propose nonparametric estimators for the second-order central moments of possibly anisotropic spherical random fields, within a functional data analysis context. We consider a measurement framework where each random field among an identically distributed collection of spherical random fields is sampled at a few random directions, possibly subject to measurement error. The collection of random fields could be i.i.d. or serially dependent. Though similar setups have already been explored for random functions defined on the unit interval, the nonparametric estimators proposed in the literature often rely on local polynomials, which do not readily extend to the (product) spherical setting. We therefore formulate our estimation procedure as a variational problem involving a generalized Tikhonov regularization term. The latter favours smooth covariance/autocovariance functions, where the smoothness is specified by means of suitable Sobolev-like pseudo-differential operators. Using the machinery of reproducing kernel Hilbert spaces, we establish representer theorems that fully characterize the form of our estimators. We determine their uniform rates of convergence as the number of random fields diverges, both for the \emph{dense} (increasing number of spatial samples) and \emph{sparse} (bounded number of spatial samples) regimes. We moreover demonstrate the computational feasibility and practical merits of our estimation procedure in a simulation setting, assuming a fixed number of samples per random field. Our numerical estimation procedure leverages the sparsity and second-order Kronecker structure of our setup to reduce the computational and memory requirements by approximately three orders of magnitude compared to a naive implementation would require. \end{abstract}

\bigskip
\begin{keyword}[class=AMS]
\kwd[Primary ]{62G08}
\kwd[; secondary ]{62M}
\end{keyword}

\begin{keyword}
\kwd{functional data analysis}
\kwd{measurement error}
\kwd{representer theorem}
\kwd{sparse sampling}
\kwd{spherical random field}
\end{keyword}

\end{frontmatter}

\setcounter{tocdepth}{1}
\tableofcontents

\newpage


\section{Introduction}
Functional Data Analysis (FDA) comprises a wide class of statistical methods for the analysis of collections of data modelled as functions. In a classical FDA setting (see ~\cite{RamSIl, Hsing}), we typically assume that we have access to a collection of \emph{fully observed} realizations of continuous-domain random processes, say $X_1(\cdot),\dots,X_n(\cdot)$, which can be modelled as random elements of some separable Hilbert space. Depending on the context, these realizations can be either independent and identically distributed ($\text{i.i.d.}$ functional data) or exhibit some form of serial dependence (functional time series).
In many practical setups however, this ``fully observed" context is not suitable. For example, environmental scientists monitor temperature, salinity and currents in the Earth’s oceans from spatial samples of such indicators collected by  drifting floats, such as the ones of the Argo fleet \cite{argo2000argo}. While originating from a latent continuous-domain process, the data in this context are not observed in their fully functional form, but rather come in the form of \emph{noisy} and \emph{sparse} spatial samples  \emph{irregularly} distributed on the surface of the Earth (see Figure \ref{fig:argo}). 
This suggests to consider the following measurement scheme
\begin{equation}
    \label{eq:firstregressionproblem}
\W_{ij} = X_i(U_{ij}) + \epsilon_{ij}, \qquad  i =1,\dots,n, \, j =1,\dots,r_i,
\end{equation}
where $U_{ij}$ are the sampling locations, $\epsilon_{ij}$ are noise disturbances. 
In this paper, we focus  specifically on latent processes $X_i$'s that constitute random fields over the \emph{$2D$ sphere} $\mathbb{S}^2$ and  address the functional data analysis problem of estimating their first and second-order moment structure. This problem will be studied for $X_i$'s that are independent replicates of some second-order process $X=\{X(u), \, u \in \Sph^2\}$ as well as for $X_i$'s that constitute a finite stretch of some stationary sequence $\X=\{X_t(\cdot), \, t \in \mathbb{Z}\}$. In either case, we will \emph{not} assume the process to be either strongly or weakly isotropic, hence our reference to an anisotropic setting.

\begin{figure}[h!]
    \centering
    \includegraphics[width=0.55\linewidth]{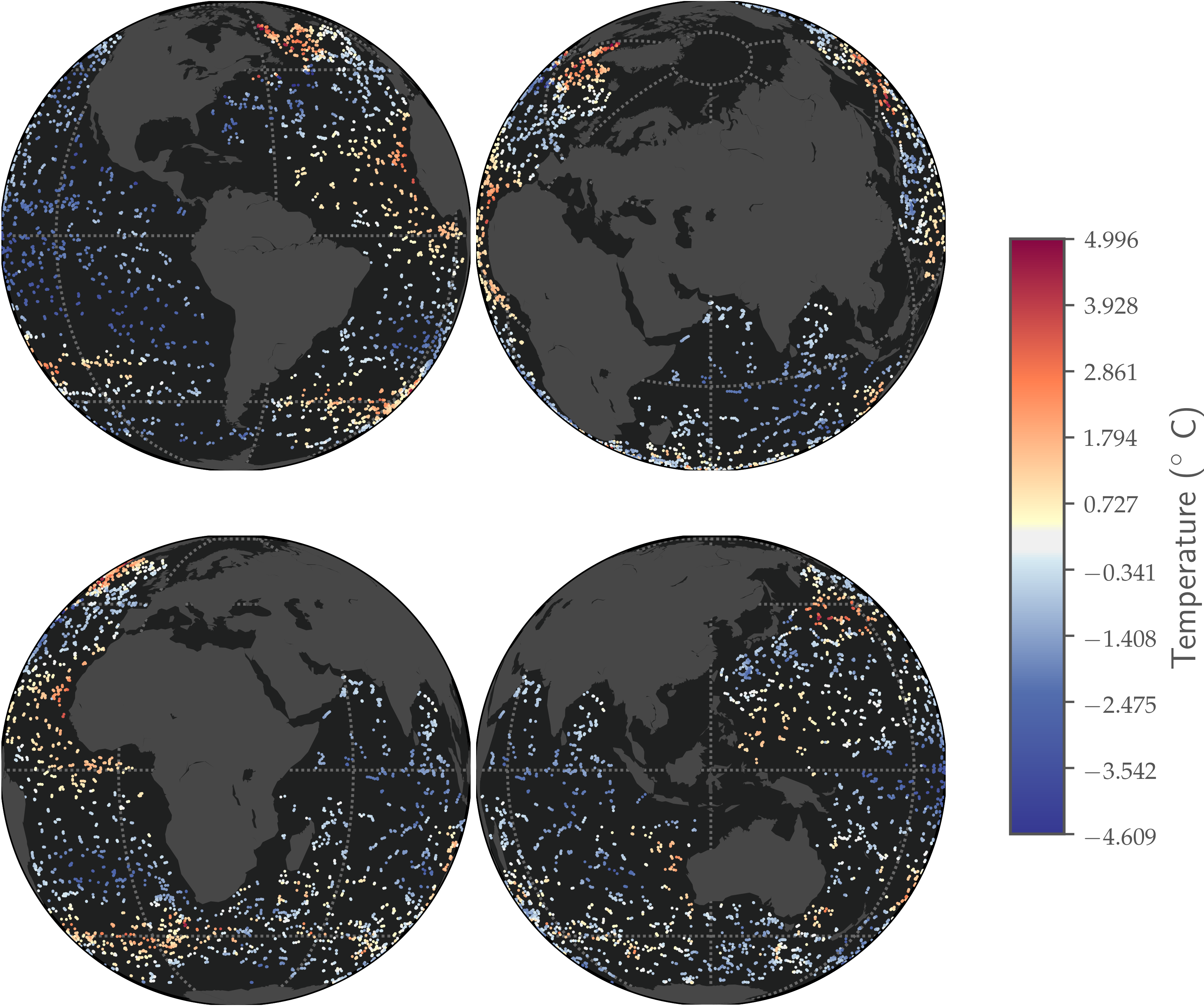}
    \caption{Sea surface temperature anomalies recorded by Argo floats in January 2011. Float locations are marked by dots coloured according to the recorded anomaly (red=warmer temperatures, blue=colder temperatures).}
    \label{fig:argo}
\end{figure}

\paragraph{Related Work in FDA.}
The model presented in Equation \eqref{eq:firstregressionproblem}, in the case of  $\text{i.i.d.}$ functional data $\{X_i\}$ on the domain $[0,1]$, has been treated in several works, as for instance in \cite{yao:2005b} with the aim of building a bridge between FDA and longitudinal studies, and also in \cite{hall2006,li2010,zhang2016sparse,caiYuan1,caiYuan2}.
In this setting, it is not possible to recover \emph{a priori} the entire signals by a pre-smoothing step, and hence the statistical estimators and procedures cannot be based on the intrinsically infinite dimensional inputs and techniques as in the classical FDA setting. However, having access to $n$ (partial) replications still makes inference possible, even if the number of points per replicate is small. The core point is that all the information coming from the measurements can be shared to estimate the main features of the underlying process, such as the mean and the covariance kernel. This idea of \emph{borrowing information} leads naturally to the implementation of smoothing techniques.
In the key paper \cite{yao:2005b}, the authors propose to estimate the covariance kernel
of the latent functional process by a local polynomial smoother, which is well defined in a 2-dimensional planar domain, and for such estimator they obtain rates of convergence. In the subsequent work \cite{rubin2020}, similar results for the lag-$h$ autocovariances and spectral density kernels have been established in the serially dependent setting. 
However, when the domain is the product sphere, standard local polynomial smoothing techniques do not straightforwardly extend. 
Indeed, the covariance estimation task consists in estimating a function on a 4-dimensional curved surface (\emph{i.e.}, $\mathbb{S}^2 \times \mathbb{S}^2$) and ideally the selected smoothing procedure would naturally incorporate the manifold geometry of the product sphere.

There exists a vast literature on (usually time-dependent) sequences of spherical random fields aimed at the characterization and parametric estimation of covariance/autocovariance functions (see \cite{bergporcu,gneiting,JJ:15,jun,Porcu18,porcu,white2019, simeoni2021siml} and in particular \cite{PFN:20} for a comprehensive review). The great majority of such works have focused on \emph{isotropic} (in space) and \emph{stationary} (in time) processes. However, in particular for the analysis of climate data, there is an attempt to move from the assumption of isotropy to that of \emph{axial symmetry} (\emph{i.e.}, heterogeneous spatial dependence across latitudes), but always in a purely parametric framework (see, for instance, \cite{castruccio2013,axial} and the references therein).

Very recently, the problem of investigating (isotropic and stationary) serially-dependent spherical random fields and their second-order structure has been tackled from a FDA perspective by \cite{cm19, cdv19, SPHARMA} when fully observed functional time series are available.

Our aim is to bridge the gap between the idealised setting of these previous works and the many practical experiments which intrinsically feature irregular and/or sparse measurements of non-isotropic random fields on the sphere, such as the Argo data from Figure \ref{fig:argo}.
Hence, we focus on both the $\text{i.i.d.}$ and serially-dependent setting, specifically for functional data measured as defined in Equation \eqref{eq:firstregressionproblem}, without assuming any form of isotropy in space. We provide a methodology which intrinsically incorporates the structure of the (product) spherical domain and also allows to estimate the mean, the covariance, and the autocovariance functions in a completely nonparametric manner. This in particular has many advantages, such as the possibility of performing spatial kriging without any parametric restrictions.

In the following, we summarize the proposed strategy and the main contributions of our work.

    \paragraph{Variational Methods for Mean and Covariance Estimation.} 
    Following the general approach of \cite{caiYuan1,caiYuan2}, our strategy for estimating the mean, covariance, and autocovariances of 
    the spherical random fields $X_i$ 
    is to define the respective estimators as solutions of regularized quadratic variational problems of the form
    \begin{equation} \label{eq:generaloptimpb}
        \min_{f\in\mathcal{H}} \mathcal{L} (\boldsymbol{y}, \Phi(f))  + \lambda \|f \|^2_{\mathcal{H}}. 
    \end{equation}
    In this setting, the loss functional $\mathcal{L} (\boldsymbol{y}, \Phi(f))$ constraints the measurement vector $\Phi(f)$ to be close to the data $\boldsymbol{y}$ -- in our case, the data are the noisy sampled values $W_{ij}$ in \eqref{eq:firstregressionproblem} for the mean estimation, or pairwise products $W_{ij}W_{i'j'}$ between these quantities for the covariance estimation. 
    The choice of the Hilbert space $\mathcal{H}$ and therefore the regularization $\|f \|_{\mathcal{H}}$ determines the smoothness of the estimators as functions over $\Sph^2$ or $\Sph^2 \times \Sph^2$. The parameter $\lambda > 0$ balances the role of the data-fidelity and the regularization. 
    
    Optimization problems of the form~\eqref{eq:generaloptimpb} have a rich history far beyond covariance estimation. The underlying Hilbert space $\mathcal{H}$ can be made of vectors or functions. They are known as Tikhonov regularization in inverse problems~\cite{Tikhonov1963solution,ito2014inverse,Unser2016representer} and ridge-regression in statistics~\cite{Hoerl1962ridge,saleh2019theory}.
    Quadratic regularized optimized problems over Hilbert \textit{function spaces} are well-known to be connected to splines since the pioneering works of Wahba in the 1970's~\cite{kimeldorf1971some}; see also~\cite{wahba1990spline}. 
    The specification of the form of the solutions of problems such as~\eqref{eq:generaloptimpb} is known as a representer theorem~\cite{scholkopf2001generalized,scholkopf2018learning}.
    Quadratic representer theorems for sequences~\cite{Unser2016representer}, continuous functions~\cite{gupta2018continuous} and periodic functions~\cite{badoual2018periodic} have been proposed. 
    The case of spherical functions is covered~\cite{simeoni2020functional, michel2012lectures}.
    For the mean and covariance estimation of random fields over $[0,1]$, we refer to~\cite{caiYuan1,caiYuan2}. To the best of our knowledge, this setting has never been proposed for \textit{spherical} random fields. 
    It is worth noting that, even if we restricted our attention to the sampling regression problem~\eqref{eq:firstregressionproblem}, we can easily generalize the approach to generalized linear measurements, as detailed for instance in~\cite{badoual2018periodic,simeoni2020functional,michel2012lectures}. Estimation of spherical random fields from generalized measurements is typically relevant in radio interferometric or acoustic imaging \cite{simeoni2021siml, simeoni2019graph,simeoni2019deepwave}.

    \paragraph{Contributions.}
    The specificity of our work is to apply quadratic regularized optimization problems to mean, covariance, and autocovariance estimation of spatially \textit{spherical} -- hence non Euclidean -- and \textit{serially-dependent} random fields. To the best of our knowledge, these ideas have not been applied in these contexts. {Our method appears, in fact, to be the first and only thus far to tackle non-parametric and non-isotropic covariance estimation with sparsely sampled functional data on the sphere.} 
    Our main contributions are as follows. 

\begin{itemize}

    
    \item \textit{Representer theorems for the estimators:} 
    We introduce a family of regularization-based estimators for the estimation of the mean, covariance, and autocovariance functions of spatially spherical and temporally stationary Gaussian random fields. 
    The regularizations for the first and second moments are given by $f \mapsto \| \mathscr{D} f \|_{\rangleL}^2$ and $g \mapsto \| (\mathscr{D} \otimes \mathscr{D}) g \|_{\rangleLL}^2$, where $\mathscr{D}$ is a smoothness-inducing pseudo-differential operator.
    For such regularizations, we obtain the general form of the mean and covariance estimators, expressed in terms of the regression data and the Green's functions of the regularizing pseudo-differential operators. 
    
    \item \textit{Asymptotic Theory:}
    Our main result is to provide an asymptotic analysis of the performance of the mean, covariance, and autocovariance estimators obtained via the representer theorems. 
    The performance is expressed in terms of the smoothness properties of both the $X_i$ and their second-order structure.
    By distinguishing between the regularity of the random fields and that of their second-order structure, we provide a refined analysis relative to existing results on $[0,1]^2$, which for important classes of spherical random fields results in improved rates. 
    
    \item \textit{Serially Dependent Setting:} We extend our methods and theory to also cover the case of serially correlated random fields on $\mathbb{S}^2$, a.k.a. \emph{spherical functional time series}. To the best of our knowledge, in the context of 
    functional time series, variational methods for estimating the mean and, especially, the sequence of lag $h$ autocovariance kernels have not been considered before. 

   \item \textit{Computation:} We develop a practically feasible implementation of our estimation methodology, circumventing the computational  and memory challenges raised by the high dimensionality of the estimation task. More specifically, leveraging our theoretical results as well as sparse/second-order structure, we develop a \emph{computationally tractable} and \emph{memory-thrifty} implementation. The latter allows us to reduce the computational/memory requirements by two to three orders of magnitude approximately. We propose moreover a $K$-fold cross-validation procedure, both for evaluating the finite-sample performance of our estimator but also to optimally select the regularization parameter $\lambda$ in \eqref{eq:generaloptimpb}. {The simulation code is openly available and fully-reproducible, and is released in the form of a standalone Python 3 Jupyter notebook hosted on GitHub: \url{https://github.com/matthieumeo/sphericov/blob/main/companion_nb.ipynb}.} Instructions for installing the required dependencies and running the notebook are available at this link: \url{https://github.com/matthieumeo/sphericov#readme}.
   
\end{itemize}

\paragraph{Outline.} The rest of the paper is organized as follows. In Section~\ref{sec:background}, we introduce some preliminary concepts on spherical and tensorial Sobolev spaces and related pseudo-differential operators. In Section~\ref{sec:model}, we formally present the model and the proposed estimators for the mean and covariance functions in the $\text{i.i.d.}$ setting. For such estimators we provide representer theorems in Section~\ref{sec:representer} and the asymptotic analysis in Section~\ref{sec:asymptoticanalysis}. Our procedure is then implemented in Section~\ref{sec:simulations} for a simulated experiment. In Section~\ref{sec:temporal} the case of serially dependent spherical random fields is examined and rates of convergence are provided for the lag-$h$ autocovariance kernels.
The proofs of all formal statements are collected in Section~\ref{sec:proofsformal}.

\section{Mathematical Background}\label{sec:background}

We denote by $L^2(\Sph^2)$ the space of real-valued square-integrable functions over the 2D sphere. 
It is well-known that any function $f \in L^2(\Sph^2)$ can be expanded as 
\begin{equation}
    f = \sum_{\ell =0}^{\infty} \sum_{  m = - \ell}^{\ell} f_{\ell,m} Y_{\ell,m},
\end{equation}
where $\{Y_{\ell,m}\}_{\ell  \in \mathbb{N}, - \ell \leq m \leq \ell}$ is a standard orthonormal basis of real spherical harmonics and the spherical Fourier coefficients $f_{\ell,m } = \langle f , Y_{\ell,m} \rangle_\rangleL \in \R$ of $f$ are such that
$\|   f \|^2_\rangleL := \sum_{\ell =0}^{\infty} \sum_{  m = - \ell}^{\ell} f_{\ell,m}^2 < \infty$ -- see for instance \cite[Chapter 3]{MP:11}.
We also consider functions over the domain $\Sph^2 \times \Sph^2$, which is required to deal with the second-order moments of spherical random fields. The family $\{ Y_{\ell,m} \otimes Y_{\ell',m'}\}_{\ell, \ell'  \in \mathbb{N}, - \ell \leq m \leq \ell, - \ell' \leq m' \leq \ell'}$ is then an orthonormal basis of $L^2(\Sph^2\times\Sph^2) = L^2(\Sph^2) \otimes L^2(\Sph^2)$, where $( Y_{\ell,m} \otimes Y_{\ell',m'}) (u,v) =  Y_{\ell,m}(u)  Y_{\ell',m'} (v)$ for any $(u,v)\in \Sph^2 \times \Sph^2$. 

    \subsection{Sobolev Spaces on $\Sph^2$ and $\Sph^2\times \Sph^2$}
    \label{sec:sobolevspaces}
    We characterize the smoothness properties of the random fields, means, and covariance functions in terms of Sobolev spaces, defined in Definition~\ref{def:sobospace}. 
    
\begin{definition} \label{def:sobospace}
Let $p \geq 0$. The spherical Sobolev space of order $p$ on $\Sph^2$ is defined as
\begin{equation}
    \Hp =  \Hp (\Sph^2) = 
    \left\{ f \in L^2(\Sph^2) , \ \|f\|^2_{\Hp} := \sum_{\ell =0}^{\infty} (1 + \ell(\ell + 1))^p  \sum_{m = - \ell}^{\ell}  \langle f , Y_{\ell,m} \rangle^2 < \infty \right\}.
\end{equation}
The tensorial Sobolev space of order $p$ on $\Sph^2 \times \Sph^2$ is
\begin{equation}
    \Hprod =  \Hprod (\Sph^2 \times \Sph^2) = \Hp \otimes \Hp \subset L^2 (\Sph^2 \times \Sph^2).
\end{equation}
\end{definition}
The spaces  $\Hp$ and $\Hprod$ are
Hilbert spaces for their respective Hilbert norms 
$ \| \cdot \|_\rangleH$ and $\| \cdot \|_\rangleHH$ inherited from the inner products
    \begin{align}
        \langle f_1, f_2 \rangle_\rangleH 
        &:=
        \sum_{\ell =0}^{\infty} (1 + \ell(\ell + 1))^p  \sum_{m = - \ell}^{\ell}  \langle f_1 , Y_{\ell,m}\rangle_\rangleL\langle f_2 , Y_{\ell,m} \rangle_\rangleL,  \label{Hp}  \\
        \langle g_1, g_2 \rangle_\rangleHH
        &:=
        \sum_{\ell, \ell'= 0}^\infty 
        (1 + \ell(\ell + 1))^p(1 + \ell'(\ell' + 1))^p \nonumber \\
        & \quad \quad \quad \times \sum_{m = - \ell}^{\ell}  \sum_{m' = - \ell'}^{\ell'}   \langle g_1 , Y_{\ell,m}\otimes Y_{\ell',m'}\rangle_\rangleL  \langle g_2 , Y_{\ell,m}\otimes Y_{\ell',m'} \rangle_\rangleL \label{Hprod}
    \end{align}
    for $f_1,f_2 \in \Hp$ and $g_1,g_2 \in \Hprod$.
    Note that the norm $\| \cdot \|_\rangleHH$ is such that $\|f_1 \otimes f_2\|_\rangleHH =\|f_1\|_\rangleH\|f_2\|_\rangleH $ for any $f_1,f_2 \in \Hp$.

     The parameter $p$ in Definition~\ref{def:sobospace} quantifies the regularity: the higher $p$, the smoother the functions  $f \in \Hp$ and  $g \in \Hprod$. 
     We also define the spaces $\mathcal{S}(\Sph^2)$ and $\mathcal{S}(\Sph^2 \times \Sph^2) = \mathcal{S}(\Sph^2) \otimes \mathcal{S}(\Sph^2)$ of infinitely smooth functions on $\Sph^2$ and $\Sph^2\times \Sph^2$, respectively. We then have that 
\begin{equation}
    \mathcal{S}(\Sph^2) = \bigcap_{p \geq 0} \Hp  \qquad \text{and} \qquad  \mathcal{S}(\Sph^2 \times \Sph^2) = \bigcap_{p \geq 0} \Hprod.
\end{equation}
     The topological duals of $\mathcal{S}'(\Sph^2)$ and $\mathcal{S}'(\Sph^2 \times \Sph^2)$ are respectively the space of distributions $\mathcal{S}(\Sph^2)$ and $\mathcal{S}(\Sph^2 \times \Sph^2)$ and are such that 
  \begin{equation}
    \mathcal{S}'(\Sph^2) = \bigcup_{p \geq 0} \Hp'  \qquad \text{and} \qquad  \mathcal{S}' (\Sph^2 \times \Sph^2) = \bigcup_{p \geq 0} \Hprod'.
\end{equation}  
     
     Among Sobolev spaces, we are specifically interested in the ones on which the evaluation functionals $f \mapsto f(x)$ are continuous for any $x$ in the domain ($\Sph^2$ or $\Sph^2\times \Sph^2$). In other terms, we shall only consider Sobolev spaces $\Hp$ and $\Hprod$ that are Reproducing Kernel Hilbert Spaces (RKHS)~\cite{aronszajn1950theory}. We characterize this property in Proposition~\ref{prop:RKHSprop}, whose proof is provided in Section~\ref{sec:proofs1}.
 
    \begin{proposition}
    \label{prop:RKHSprop}
    Let $p \geq 0$. Then, we have the  equivalences
    \begin{equation}
        \Hp \text{ is a RKHS} \quad \Longleftrightarrow \quad  \Hprod \text{ is a RKHS} \quad \Longleftrightarrow  \quad p > 1   .
    \end{equation}
    \end{proposition}
    
    The RKHS properties allow us to provide uniform control functions on $\Hp$ and $\Hprod$, as will be used for instance in Lemma~\ref{lemma::sup} in Section~\ref{sec:proofs3}. This implies in particular the uniform convergence of Fourier series; \emph{e.g.}, for any $f \in \Hp$,
$$
\sup_{u \in \mathbb{S}^2} \left |f(u) - \sum_{\ell=0}^L \sum_{m=-\ell}^\ell \langle f, Y_{\ell,m} \rangle_\rangleL Y_{\ell,m}(u)\right| \to 0, \qquad L \to \infty.
$$

    \subsection{Spherical and Tensorial Admissible Operators}
    \label{sec:sphoperators}

    The Tikhonov regularization used in our estimation strategy relies on linear operators that will impact the smoothness of the estimators. We now introduce the class of linear operators that are relevant for our work.
    
\begin{definition}[Admissible Operators]\label{sph-pseudo-diff}
An \emph{admissible operator} will be any linear operator of the form
\begin{equation*}
\mathscr{D} :  \begin{dcases}
\mathcal{S}(\mathbb{S}^2) \to \mathcal{S}(\mathbb{S}^2)\\
f \mapsto \mathscr{D} f = \sum_{\ell=0}^\infty D_\ell \sum_{m=-\ell}^\ell f_{\ell,m} \Y_{\ell,m}
\end{dcases}
\end{equation*}
where $\{f_{\ell,m} = \langle f, Y_{\ell,m} \rangle_\rangleL\}_{\ell \in \mathbb{N},  -\ell \le m\le \ell}$ are the spherical Fourier coefficients of $f$ and $\{D_\ell\}_{\ell \in \mathbb{N}}$ is a sequence of non-zero real numbers such that
\begin{equation} \label{eq:conditionDn}
C_1 (1+\ell)^p \le |D_\ell| \le C_2 (1+\ell)^p,
\end{equation}
for some real number $p\ge0$,  some positive constants $C_1, C_2 > 0$ and every $\ell \in \mathbb{N}$.
We call the $D_\ell$, $\ell \in \mathbb{N}$, the \emph{spherical Fourier coefficients} and $p$ the \emph{spectral growth order} of $\mathscr{D}$. 
\end{definition}


The condition~\eqref{eq:conditionDn} implies that $\mathscr{D}$ is invertible from $\mathcal{S}(\mathbb{S}^2)$ to itself. 
The inverse of $\mathscr{D}$ is given by $\mathscr{D}^{-1} f = \sum_{\ell=0}^\infty \frac{1}{D_\ell} \sum_{m=-\ell}^\ell f_{\ell,m} \Y_{\ell,m}$.
Note that Definition \ref{sph-pseudo-diff} excludes some classical pseudo-differential operators such as the Laplacian operator $\mathscr{D} = \Delta_{\Sph^2}$ itself, which is not invertible. Indeed, we have in this case that $D_\ell = - \ell(\ell+1)$, hence $D_0 = 0$ and the condition \eqref{eq:conditionDn} is not fulfilled. It is however possible to generalize Definition \ref{sph-pseudo-diff} above to include the Laplacian and more generally spherical pseudo-differential operators with finite-dimensional null space -- see for example \cite[Definition 4]{simeoni2021functional}. 



\begin{example} \label{ex:sobolevop}
We consider the family of \textit{Sobolev operators}, \emph{i.e.},  operators of the form $\mathscr{D}:= (\operatorname{Id} - \Delta_{\mathbb{S}^2} )^{p/2}$ for some $p \geq 0$.
The spherical Fourier coefficients of $(\operatorname{Id} - \Delta_{\mathbb{S}^2} )^{p/2}$ are given by 
$$
D_\ell = (1+ \ell(\ell+1))^{p /2 }, \qquad \ell \in \mathbb{N}.
$$
We easily see that  $(\operatorname{Id} - \Delta_{\mathbb{S}^2} )^{p/2}$ satisfies the conditions of Definition~\ref{sph-pseudo-diff} with spectral growth order $p$. 
\end{example}


We can use the admissible operators in Definition~\ref{sph-pseudo-diff} to construct operators acting on functions over $\Sph^2 \times \Sph^2$ as follows. Let $\mathscr{D}$ be an admissible operator. Then, $\mathscr{D} \otimes \mathscr{D}$ is a linear operator from $\mathcal{S}(\Sph^2\times\Sph^2)$ to itself characterized by the relations  
$$
(\mathscr{D} \otimes \mathscr{D}) (f_1 \otimes f_2) = (\mathscr{D}  f_1 ) \otimes (\mathscr{D}  f_2)
$$
for any $f_1,f_2 \in \mathcal{S}(\Sph^2)$. The invertibility of $\mathscr{D}$ then implies the invertibility of $\mathscr{D} \otimes \mathscr{D}$ from $\mathcal{S}(\Sph^2\times\Sph^2)$ to itself. 

The Sobolev operator $\mathscr{D} = (\operatorname{Id} - \Delta_{\mathbb{S}^2} )^{p/2}$ of order $p \geq 0$ is such that
\begin{equation}
     \|f\|_\rangleH = \| \mathscr{D} f \|_\rangleL \quad \text{and} \quad 
     \|g\|_\rangleHH = \| ( \mathscr{D} \otimes \mathscr{D}) g \|_\rangleLL
\end{equation}
for any $f \in \Hp$ and $g \in \Hprod$.
More generally, any admissible operator $\mathscr{D}$ specifies a continuous bijection $\mathscr{D} : \Hp \rightarrow L^2(\Sph^2)$ (see Proposition~\ref{prop:Disgood} in Section~\ref{sec:proofs1} for a formal statement). This implies that 
$f \mapsto \| \mathscr{D} f \|_\rangleL$ specifies a norm on $\Hp$ which is equivalent to $\|\cdot \|_\rangleH$. Similarly, $g \mapsto \|( \mathscr{D} \otimes \mathscr{D} ) g  \|_\rangleL$ specifies a norm on $\Hprod$ which is equivalent to $\|\cdot \|_\rangleHH$. 

\begin{remark}
For any admissible operator $\mathscr{D}$, the family $\left\{ Y_{\ell,m} / D_\ell \right\}_{\ell  \in \mathbb{N}, - \ell \leq m \leq \ell}$ is orthonormal in $(\Hp, \| \mathscr{D} \cdot \|_{L^2(\Sph^2)})$ This simply follows from the fact that, as easily seen from Definition~\ref{sph-pseudo-diff}, $\mathscr{D} Y_{\ell,m} = D_\ell Y_{\ell,m}$ and therefore,
\begin{equation} \label{eq:useful1}
    \langle \mathscr{D}Y_{\ell,m}, \mathscr{D}Y_{\ell',m'} \rangle_\rangleL
    =
    D_{\ell} D_{\ell'}^{*} \langle Y_{\ell,m} , Y_{\ell',m'} \rangle_\rangleL  
    =
    D_\ell^2 \delta_\ell^{\ell'} \delta_m^{m'}.
\end{equation}
In particular, $\left\{ Y_{\ell,m} / (1 + \ell (\ell+1))^{p/2} \right\}_{\ell  \in \mathbb{N}, - \ell \leq m \leq \ell}$ is an orthonormal basis of $(\Hp, \| \cdot \|_{\Hp})$. 
The relation~\eqref{eq:useful1} moreover implies, for any $f \in \Hp$, we have the useful relation
\begin{equation}
    \label{eq:useful2}
\langle \mathscr{D}f, \mathscr{D}Y_{\ell,m} \rangle_\rangleL = D_\ell^2 \langle f, Y_{\ell,m} \rangle_\rangleL.
\end{equation}
\end{remark}

We now introduce the Green's functions and the zonal Green kernel of admissible operators in Definition~\ref{def:green}. 

\begin{definition}\label{def:green}
Let $\mathscr{D}$ be an admissible operator. 
For any $u \in \mathbb{S}^2$, we denote by $\green^{\mathscr{D}}_{u} = \mathscr{D}^{-1} \delta_u$. Then, we call  $\green^{\mathscr{D}}_{u}$ the \emph{Green's function} of the operator $\mathscr{D}$ at position $u$. 
Moreover, there exists a function $\psi_{\mathscr{D}} : [-1,1] \rightarrow \R$ such that, for any $u,v \in \Sph^2$,
\begin{equation} \label{eq:zonalgreen}
    \green_u^{\mathscr{D}} (v) = \psi_{\mathscr{D}} (\langle u, v\rangle), 
\end{equation}
where $\langle u, v\rangle$ is the usual inner product between points in $\Sph^2 \subset \R^3$. The function $\psi_{\mathscr{D}}$ is called the \emph{zonal Green's kernel} of $\mathscr{D}$.
\end{definition}

It is known that the Green's functions of admissible operators are continuous as soon as the spectral growth order $p$ satisfies $p > 1$~\cite[Proposition 4]{simeoni2021functional}. 
The existence of the zonal Green's kernel is for instance proved in ~\cite[Proposition 3]{simeoni2021functional}; it can be expressed in terms of the spherical Fourier coefficients $\{D_\ell\}$ and the $2$-dimensional ultraspherical polynomials (see \cite[Eq. (16)]{simeoni2021functional}). 
The existence of the zonal Green's kernel has important practical consequences. It means in particular that the Green's functions at any positions can be easily computed from $\psi_{\mathscr{D}}$, as will be exploited in Section~\ref{sec:simulations}. 


\section{Model and Estimation Methodology} \label{sec:model}

In this section, we present the theoretical setting in which we develop our methodology.
From a purely functional data analysis perspective, we consider a second-order stochastic process $X=\{X(u), \, u \in \mathbb{S}^2\}$ that is a random element of $\Hq$, for some $q >1$, with
$$
\mathbb{E}[X (u)] = \mu(u), \qquad  \mathbb{E}[X(u) X(v)] = R(u,v),
$$
and covariance function
$$
C(u,v) = R(u,v) - \mu(u)\mu(v),
$$
for $u,v \in \mathbb{S}^2$. Recall from Proposition \ref{prop:RKHSprop} that $\Hq$ with $q>1$ is a RKHS, hence the process $X$ inherits all the properties of a RKHS valued process, see \cite[Section 7.5]{Hsing}. Moreover, since $q>1$, the realizations of $X$ are almost surely continuous on $\Sph^2$.


Now suppose we have $X_1,\dots,X_n$ independent replicates of $X$ and that for the $i$-th replicate we make measurements at $r_i$ random locations on $\mathbb{S}^2$. Formally, we consider the following regression problem
$$
\W_{ij} = X_i(U_{ij}) + \epsilon_{ij}, \qquad  i =1,\dots,n, \, j =1,\dots,r_i,
$$
where the $U_{ij}$'s 
are independently drawn from a common distribution on $\mathbb{S}^2$, and the $\epsilon_{ij}$'s are independent and identically distributed measurement errors of mean 0 and variance $0<\sigma^2<\infty$. 
Furthermore, the $X_i$'s, the measurement locations, and the measurement errors are assumed to be mutually independent.
Then,
\begin{equation}\label{eq::meanconditional}
\mathbb{E}[\W_{ij}|U_{ij}=u_{ij}] = \mu(u_{ij}),
\end{equation}
and
\begin{equation}
    \label{eq::secondmomentconditional}
    \mathbb{E}[\W_{ij}\W_{lk}|U_{ij}=u_{ij}, U_{lk}=u_{lk}] = \begin{cases}
    \mu(u_{ij}) \mu(u_{lk}) \quad \text{if } i \ne l\\
    R(u_{ij}, u_{ik}) + \sigma^2\delta_j^k \qquad  \text{if }   i=l
    \end{cases} .
\end{equation}
Moreover,
$$
\operatorname{Cov}[\W_{ij},\W_{lk}|U_{ij}=u_{ij}, U_{lk}=u_{lk}] = \begin{cases}
0 \quad \text{if } i \ne l\\
C(u_{ij}, u_{ik}) + \sigma^2\delta_j^k \qquad  \text{if }   i=l
\end{cases} .
$$
Below, expectations will be sometimes computed conditionally on/with respect to the whole set of $U_{ij}$'s, which will be denoted by $\U$.

It appears evident from Equations \eqref{eq::meanconditional} and \eqref{eq::secondmomentconditional} that the measurements themselves and the \emph{off-diagonal} products can be seen as unbiased estimators for the mean and second-order moment functions, respectively, computed at fixed locations.
In particular, we can recover the covariance by performing smoothing on the pooled measurements (to first recover the mean) and then on the pooled product observations (to recover the second-order moment function). Note that it is important in this second step to discard the “diagonal" elements (not easy to visualize on $\Sph^2 \times \Sph^2$ as opposed to $[0,1] \times [0,1]$) so that potential biases arising from the noise are neglected (as in \cite{yao:2005b}).
We then use the previously introduce machinery to build smoothing techniques that takes into account the geometry of $\Sph^2$. 

Specifically, for a given $p\ge q$ and $\pen, \lambda > 0 $, we can define the following estimators for the mean function
\begin{equation}\label{eq::mu-est}
    \mu_\lambda := \argmin_{g \in \Hp} \frac{4\pi}{n} \sum_{i=1}^n \frac{1}{r_i} \sum_{j=1}^{r_i} (\W_{ij}  - g(U_{ij}))^2 + \lambda \| \mathscr{D} g\|^2_\rangleL,
\end{equation}
and for the second-order moment function 
\begin{equation}\label{eq::R-est}
    R_\pen := \argmin_{g \in \Hprod}  \frac{(4\pi)^2}{n} \sum_{i=1}^n \frac{1}{r_i(r_i-1)} \sum_{1\le j \ne k \le r_i } (\W_{ij}\W_{ik}  - g(U_{ij}, U_{ik}))^2 + \pen \| (\mathscr{D} \otimes \mathscr{D})  g\|^2_\rangleLL,
\end{equation}
where $\mathscr{D}$ is any admissible operator with spectral growth order $p$ (see Definition \ref{sph-pseudo-diff} in Section \ref{sec:sphoperators}).
An estimate of the complete covariance kernel $C(u,v)=R(u,v) - \mu(u)\mu(v)$ is then given by
\begin{equation}\label{eq::cov-est}
C_{\pen,\lambda}(u,v) = R_\pen(u,v)  - \mu_\lambda(u) \mu_\lambda(v).
\end{equation}
\begin{remark}
In order to simplify the notation, from now on, we will not differentiate between the two penalty parameters $\pen, \lambda$. In particular, for the covariance estimator, we will consider $\pen = \lambda$ and write directly $C_\pen$. This will not affect in any way our asymptotic results.
\end{remark}

We remark that our estimators are genuinely nonparametric and anisotropies in space are allowed. Moreover, the class of admissible operators is very large and, hence, this implementation gives the possibility to be flexible with respect to particular practical problems (non-asymptotic regimes).

\section{Representer Theorems for Mean and Covariance Estimation}\label{sec:representer}

In this section, we specify the form of the solutions of the optimization problems which are the cornerstones of our estimation strategy. 
More precisely, representer theorems for the mean estimator \eqref{eq::mu-est} and the second-order moment estimator \eqref{eq::R-est} are stated in Sections \ref{sec:RTmean} and \ref{sec:RTsecond} respectively.
In both cases, the form of the solution is deduced from general principles for optimization problems over Hilbert spaces, that are presented in Section~\ref{sec:proofs2}.



    \subsection{Representer Theorem for Mean Estimation}  \label{sec:RTmean}

     As we have seen in \eqref{eq::mu-est} and \eqref{eq::R-est}, 
     our  strategy for the mean and covariance estimations relies on the minimization of quadratic cost functionals with two components: (i) a data-fidelity term which constraint the solution to be consistent with the sampled observations and (ii) a regularization term which enforces some smoothness condition via the admissible operator $\mathscr{D}$.
     Our goal in this section and the next one is to reveal the form of the solution of the optimization problem together with their main properties. We start with the mean estimation in Theorem \ref{theo:meanesti}. 
     
    \begin{theorem}
    \label{theo:meanesti}
   Let $\mathscr{D}$ be an admissible operator with spectral growth $p > 1$. 
    Let $n \geq 1$, $r_i \geq 1$ for $i = 1 ,\ldots n$. We consider weights $w_{ij} \in \R$ and pairwise distinct positions $u_{ij} \in \Sph^2$ for $i=1,\ldots, n$ and $j =1 ,\ldots, r_i$. We fix $\eta > 0$. 
    Then, the optimization problem 
    \begin{equation}\label{eq::mu-estbis}
    \min_{f \in \Hp} \frac{4\pi}{n} \sum_{i=1}^n \frac{1}{r_i} \sum_{j=1}^{r_i} (w_{ij}  - f(u_{ij}))^2 + \pen \| \mathscr{D} f\|^2_\rangleL
    \end{equation}
    has a unique solution $\mu_{\eta} \in \Hp$ which is given by
    \begin{equation} \label{eq:RTformnu}
        \mu_{\eta} =  \sum_{i=1}^n \frac{1}{\sqrt{r_i}} 
        \sum_{j=1}^{r_i} \alpha_{ij} \psi_{\mathscr{D}^* \mathscr{D}} (\langle \cdot , u_{ij} \rangle)
    \end{equation}
    for some  $(\alpha_{ij})_{1 \leq i \leq n , \ 1 \leq j \leq r_i}$, where $\psi_{\mathscr{D}^* \mathscr{D}}$  is the zonal Green's kernel of $\mathscr{D}^* \mathscr{D}$ (see Definition~\ref{def:green}). 
    
    Moreover, the coefficients $\alpha_{ij}$ are computed as follows. Let $L = \sum_{i=1}^n r_i$ be the total number of measurements. We set $\boldsymbol{\alpha} \in \R^{L}$ the vectorized version of  $(\alpha_{ij})_{1 \leq i \leq n , \ 1 \leq j \leq r_i}$, $\boldsymbol{y} \in \R^L$ the vectorized version of the normalized observations $\left(\frac{w_{ij}}{\sqrt{r_i}}\right)_{1 \leq i \leq n , \ 1 \leq j \leq r_i}$, 
    and $\boldsymbol{\mathrm{G}} \in \R^{L\times L}$ the matrix whose entries are given by 
    \begin{equation} \label{eq:G12}
        G_{\ell_1 \ell_2} 
        = \frac{\psi_{\mathscr{D}^* \mathscr{D}} (\langle u_{i_1j_1}, u_{i_2j_2} \rangle)}{\sqrt{r_{i_1} r_{i_2}}}.
    \end{equation}
        where the index $\ell_1$ (resp. $\ell_2$) corresponds to the couple $(i_1,j_1)$ (resp. $(i_2,j_2)$) in the vectorization.
    Then, we have that
    \begin{equation}\label{eq:alphas}
        \boldsymbol{\alpha} = \left( \boldsymbol{\mathrm{G}} + \frac{\eta n}{4 \pi}  \boldsymbol{\mathrm{I}}_L \right)^{-1} \boldsymbol{y}
    \end{equation}
    where $\boldsymbol{\mathrm{I}}_L \in \R^{L \times L}$ is the identity matrix.
    \end{theorem}
   
    As far as the mean is concerned, the representer theorem can be deduced from known results in the literature. In particular, the general form \eqref{eq:RTformnu} of the solution $\mu_{\eta}$ in Theorem~\ref{theo:meanesti} can be seen as a particular case of~\cite[Theorem 5.3]{simeoni2020functional}. 
    The specification of the weights $\alpha_{ij}$ is then a finite-dimensional quadratic optimization problem (see for instance~\cite[Section V.A]{gupta2018continuous} and~\cite[Proposition 4]{badoual2018periodic} for similar results on the discretization of quadratic optimization problems over function spaces).
    Theorem~\ref{theo:meanesti} is also a special case of the general Representer Theorem over Hilbert spaces that we recall in Section~\ref{sec:proofs2} (see Theorem~\ref{theo:RTHilbert}). This specification is similar -- and simpler -- to the one we detail in Theorem~\ref{theo:Resti} for the second-order estimator, hence we do not detail it. 

    \begin{remark}
    \label{rmk:smoothing_spline}
    The estimator $\mu_{\eta}$ of \eqref{eq::mu-estbis} depends linearly on $\boldsymbol{\alpha}$ via \eqref{eq:RTformnu}, which depends itself linearly on the observations $\left({w_{ij}}\right)_{1 \leq i \leq n , \ 1 \leq j \leq r_i}$. Hence, $\mu_{\eta}$ is a linear estimator of the mean $\mu$. 
       We moreover observe that  $\mu_\eta$ is a $(\mathscr{D}^*\mathscr{D})$-spline in the sense that~\cite[Definition 7]{simeoni2020functional}
    \begin{equation} \label{eq:splinemu}
        (\mathscr{D}^*\mathscr{D}) \{\mu_\eta\} = \sum_{i=1}^n \frac{1}{\sqrt{r_i}} \sum_{j=1}^{r_i} \alpha_{ij} \delta_{u_{ij}} 
    \end{equation}
    is a sum of Dirac impulses. 
    The estimator $\mu_\eta$ is therefore the optimal $(\mathscr{D}^*\mathscr{D})$-spline with knots at the sampling locations. This is the adaptation in our context of a well-known fact for Tikhonov-type regularization~\cite{simeoni2020functional,kimeldorf1971some,berlinet2011reproducing,badoual2018periodic}. 
    
    This reveals that the choice of the admissible operator $\mathscr{D}$ is crucial. First, the spectral growth order determines the smoothness of the estimated mean, which is in $\mathcal{H}_{p}$. Since $p > 1$, we deduce that $\mu_{\eta}$ is continuous over the sphere $\Sph^2$.
        Second, the shape of the zonal Green's kernel $\psi_{\mathscr{D}^* \mathscr{D}}$ determines the general form of the reconstruction. Distinct admissible operators with identical spectral growth order will have identical asymptotic performances (see Section~\ref{sec:asymptoticanalysis}) but can lead to distinct practical performances in the non-asymptotic regime, what will be exploited in Section~\ref{sec:simulations}. 
    \end{remark}
    
    

    \subsection{Representer Theorem for Second-Order Estimation} \label{sec:RTsecond}
    
    We now present a representer theorem which gives the form of the estimator for the second-order moment function as the solution of \eqref{eq::R-est}. The proof of Theorem~\ref{theo:Resti} is provided in Section~\ref{sec:proofs2}.
    
    \begin{theorem}
        \label{theo:Resti}
     Let $\mathscr{D}$ be an admissible operator with spectral growth $p > 1$. 
    Let $n \geq 1$, $r_i \geq 1$ for $i = 1 ,\ldots n$. We consider weights $w_{ij} \in \R$ and pairwise distinct positions $u_{ij} \in \Sph^2$ for $i=1,\ldots, n$ and $j =1 ,\ldots, r_i$. We fix $\eta > 0$. 
    Then, the optimization problem 
    \begin{equation}\label{eq::R-estbis}
    \min_{g \in \Hprod}  \frac{(4\pi)^2}{n} \sum_{i=1}^n \frac{1}{r_i(r_i-1)} \sum_{1\le j \ne k \le r_i } (w_{ij}w_{ik}  - g(u_{ij}, u_{ik}))^2 + \pen \| (\mathscr{D} \otimes \mathscr{D})  g\|^2_\rangleLL
    \end{equation}    
    has a unique solution $R_\eta \in \Hprod$  which is given by
    \begin{equation} \label{eq:RTformR}
        R_{\eta} =  \sum_{i=1}^n \frac{1}{\sqrt{r_i(r_i-1)}} \sum_{1 \leq j \neq k \leq r_i} \beta_{ijk} \psi_{\mathscr{D}^* \mathscr{D}} ( \langle \cdot ,u_{ij}\rangle)  \otimes \psi_{\mathscr{D}^* \mathscr{D}} ( \langle \cdot ,u_{ik}\rangle)
    \end{equation}
    for some  $(\beta_{ijk})_{1 \leq i \leq n, \ 1\leq j \neq k \leq r_i}$, where $\psi_{\mathscr{D}^* \mathscr{D}}$ is the zonal Green's kernel of $\mathscr{D}^* \mathscr{D}$. 
    
   The coefficients $\beta_{ijk}$ are computed as follows. Let $L = \sum_{i=1}^n r_i (r_i -1)$. We set $\boldsymbol{\beta} \in \R^{L}$ the vectorized version of  $(\beta_{ijk})_{1 \leq i \leq n, \ 1\leq j \neq k \leq r_i}$, $\boldsymbol{z} \in \R^L$ the vectorized version of the normalized observations $\left(\frac{w_{ij} w_{ik}}{\sqrt{r_i(r_i-1)}}\right)_{1 \leq i \leq n, \ 1\leq j \neq k \leq r_i}$, 
    and $\boldsymbol{\mathrm{H}} \in \R^{L\times L}$ the matrix whose entries are given by 
    \begin{equation} \label{eq:H12}
        H_{\ell_1 \ell_2} = 
        \frac{ \psi_{\mathscr{D}^* \mathscr{D}} ( \langle u_{i_1 j_1} , u_{i_2 j_2} \rangle) \times 
        \psi_{\mathscr{D}^* \mathscr{D}} ( \langle u_{i_1 k_1} , u_{i_2 k_2} \rangle) }{\sqrt{r_{i_1}(r_{i_1}-1) r_{i_2}(r_{i_2}-1)}}
    \end{equation}
    where the index $\ell_1$ (resp. $\ell_2$) corresponds to the triplet $(i_1,j_1,k_1)$ (resp. $(i_2,j_2,k_2)$) in the vectorization.
    Then, we have that
    \begin{equation}\label{eq:betas}
        \boldsymbol{\beta} = \left( \boldsymbol{\mathrm{H}} + \frac{\eta n}{(4 \pi)^2}  \boldsymbol{\mathrm{I}}_L \right)^{-1} \boldsymbol{z}.
    \end{equation}
    Moreover, the estimator $R_{\eta}$ is symmetric in the sense that $R_\eta (u,v) = R_\eta(v,u)$ for any $u,v \in \Sph^2$. 
    \end{theorem}
    
    The practical vectorization of the optimization problem \eqref{eq::R-estbis} for the specification of the $\beta_{ijk}$, which is only implicit in Theorem~\ref{theo:Resti}, is detailed in Section~\ref{sec:simulations}; see in particular~\eqref{eq:vectorization_index_map}. 
    
    \begin{remark}
    
    The estimator $R_\eta$ is symmetric. This is due to the use of a tensorial admissible operator $\mathscr{D}\otimes \mathscr{D}$ in the regularization and is consistent with the fact that the second-order moment $R$ is symmetric by construction.
    
    Moreover, $R_\eta$ is linear with respect to the pairwise products $w_{ij}w_{ik}$, and it is a $(\mathscr{D}^*\mathscr{D})\otimes (\mathscr{D}^*\mathscr{D})$-spline, in the sense that
        \begin{equation} \label{eq:splinesecondmoment}
        (\mathscr{D}^*\mathscr{D})\otimes (\mathscr{D}^*\mathscr{D}) \{R_\eta\} = \sum_{i=1}^n \frac{1}{\sqrt{r_i(r_i-1)}} \sum_{1\leq j \neq k \leq r_i}  \beta_{ijk} \delta_{(u_{ij},u_{ik})}, 
    \end{equation}
    with knots determined by the sampling locations $u_{ij}$. As an element of $\Hprod$ for $p > 1$, the estimated second-order moment function $R_\eta$ is also continuous on $\Sph^2\times\Sph^2$.
    \end{remark}

    Finally, we estimate the covariance function $C$ given by $C(u,v)=R(u,v) - \mu(u)\mu(v) = (R - \mu \otimes \mu)(u,v)$ using the estimators $R_\eta$ and $\mu_\eta$ via
    \begin{equation} \label{eq:Cetaesti}
        C_\eta = R_\eta - \mu_\eta \otimes \mu_\eta. 
    \end{equation}
    The estimator $C_\eta$ is continuous and symmetric on $\Sph^2\times \Sph^2$, is a $(\mathscr{D}^*\mathscr{D})\otimes (\mathscr{D}^*\mathscr{D})$-spline in $\Hprod$ (due to \eqref{eq:splinemu} and \eqref{eq:splinesecondmoment}), and is linear with respect to the pairwise products $w_{ij}w_{ik}$.

\begin{remark}
We remark that our nonparametric estimators $R_\eta$ and $C_\eta$ are \emph{not} necessarily positive semi-definite in general. This is a standard occurrence in functional data analysis, not specific to our estimator, but common to all smoothing-based nonparametric methods (e.g. the PACE estimator \cite{yao:2005b}, see also \cite{li2010}, the RKHS estimator \cite{caiYuan2}, or their extensions to time series \cite{rubin2020}). In practice, however, this has negligible effects: asymptotically positive semi-definiteness is recovered, as is proved in Section \ref{sec:asymptoticanalysis}, whereas in finite-sample setups it can be recovered via a projection (which does not affect the rate) on the semi-definite cone, as is explained in Section \ref{sec:simulations}. Indeed the asymptotic theory guarantees that negative eigenvalues will typically only be encountered at the tail end of the spectrum.
\end{remark}

\section{Asymptotic Theory} \label{sec:asymptoticanalysis}

This section contains the main results of our paper, that is, uniform rates of convergence for the mean and covariance estimators. In the following, we define the class of probability measures for which our rates are achieved.

\begin{definition}\label{def:pq-mean}
Consider $p>2$, $1<q\le p$. Let $\pio$ be the collection of probability measures for $\Hq$-valued processes such that for any $X$ with probability law $\mathbb{P}_X \in \pio$
$$\mathbb{E}\|X\|^2_{\Hq} \le M,\qquad
\|\mu\|^2_{\rangleH} \le K,
$$
for some constants $M,K > 0$.
\end{definition}

\begin{definition}\label{def:pq-cov}
Consider $p>2$, $1<q\le p$. Let $\pit$ be the collection of probability measures for $\Hq$-valued processes, such that for any $X$ with probability law $\mathbb{P}_X \in \pit$
$$\mathbb{E}\|X\|^4_{\Hq}\le L, \qquad
\|R\|^2_{\rangleHH} \le K_1, \qquad \|\mu\|^2_{\rangleH} \le K_2,
$$
for some constants $L, K_1, K_2 > 0$.
\end{definition}

\noindent The realizations of the process $X$ are presumed to lie in $\Hq$, for some $1 < q \le p$. This entails that they are allowed to be “rougher" than the mean and covariance functions; indeed, $q$ can be strictly smaller than $p$.
As a result, the class of processes considered in \cite{caiYuan1} is a special case of that considered in Definition \ref{def:pq-mean} for the mean estimation, \emph{i.e.}, when $p=q$. Note indeed that, if $p = q$,
$$
\|\mu\|^2_\rangleH \le \mathbb{E} \|X\|^2_\rangleH \le M,
$$
and the covariance kernel (as well as the second-order moment) belongs to the direct product Hilbert space $\Hprod = \Hp \otimes \Hp$, since
$$
\|C \|_\rangleHH ^2 \le \left (\mathbb{E}\|X - \mu \|_\rangleH ^2\right)^2 \le  \left (\mathbb{E}\|X \|_\rangleH ^2\right)^2 < \infty.
$$
Moreover, if $p = q$ in Definition \ref{def:pq-cov}, we have that
$$
\|C\|^2_\rangleHH \le \mathbb{E} \|X\|^4_\rangleH \le L,
$$
which yields to the class of processes considered in \cite{caiYuan2} for  covariance estimation.
These considerations will be discussed in more detail in Section \ref{sec:examples}.

\begin{remark}
The requirements in Definitions \ref{def:pq-mean} and \ref{def:pq-cov} are smoothness conditions, which form the basis for any nonparametric procedure. Their specific form amounts to classical Sobolev ellipsoid conditions (see, e.g., \cite[Chapter 7]{wasserman2006all}).
\end{remark}

Our rates will be expressed in terms of the number of replicates $n$ and the \emph{average} number of measurement locations
$$
r:= \left( \frac1n \sum_{i=1}^n \frac{1}{r_i}\right)^{-1},
$$
defined as the harmonic mean of the $r_i$'s.
The results can thus be interpreted in both the \emph{dense} and \emph{sparse} sampling regimes.
In a dense design, $r=r(n)$ is required to diverge with $n$ and it is larger than some order of $n$; on the other hand, in a sparse design, the sampling frequency $r$ is bounded and can be arbitrary small (as small as two). 


\subsection{Mean Estimation}\label{sec::mean}

This section is devoted to the estimation of the mean $\mu(\cdot)$. In the following theorem we provide a (uniform) rate of convergence for the estimator given in Equation \eqref{eq::mu-est}, under a suitable condition on the decay of the penalty parameter $\pen$. The proof is provided in Section~\ref{sec:proofs3}.

\begin{theorem}\label{th::mu} Assume that the $U_{ij}, i=1,\dots,n, j=i,\dots,r_i$, are independent copies of $U\sim \operatorname{Unif}(\Sph^2)$. Let $p>2$ and $1<q\le p$, and consider the estimation problem in Equation \eqref{eq::mu-est} for an admissible operator $\mathscr{D}$ of spectral growth order $p$. 
If $\pen \asymp (n r)^{-p/(p+1)}$, then
\begin{equation*}
\lim_{D\to \infty} \limsup_{n\to\infty} \sup_{\mathbb{P}_X \in \pio} \mathbb{P}(\| \mu_\pen  - \mu \|^2_\rangleL > D((n r)^{-p/(p+1)} + n^{-1}) ) = 0.
\end{equation*}
\end{theorem}

\begin{corollary}
Let $X$ be such that $\mathbb{P}_X \in \pio$. Under the same assumptions of Theorem \ref{th::mu},
$$   
\| \mu_\pen  - \mu \|^2_\rangleL = O_\mathbb{P}\left ((n r)^{-p/(p+1)} + n^{-1}\right).
$$
\end{corollary}

Theorem \ref{th::mu} tells us that the estimator $\mu_\pen$ \emph{achieves} the rate $(n r)^{-p/(p+1)} + n^{-1}$ uniformly over the class $\pio$ and hence such rate is called \emph{achievable} for that class (see \cite{stone1980}). 

It is important to remark that we can observe a  phase transition phenomenon with a boundary at $r = n^{1/p}$, which allows to discriminate between sparse and dense sampling regimes. 
Indeed, when the sampling frequency $r$ is small, that is, $r = O(n^{1/p})$, we have
$$   
\| \mu_\pen  - \mu \|^2_\rangleL = O_\mathbb{P}\left ((n r)^{-p/(p+1)}\right),
$$
which is equivalent to the rate of a smoothing spline estimator in nonparametric regression based on $nr$ independent observations. In other words, the convergence rate is not affected by the
spatial dependence. 
In the case of high sampling frequency with $r\gg n^{1/p}$ , we have a parametric rate
$$   
\| \mu_\pen  - \mu \|^2_\rangleL = O_\mathbb{P}\left (n^{-1}\right),
$$
that does not depend on $r$, as in a classical FDA approach where one can observe $X_1,\dots,X_n$ continuously on $\mathbb{S}^2$.

All the previous definitions and results apply also to hypersphere $\mathbb{S}^d$, $d>2$. In particular, \emph{hyperspherical harmonics} (see for instance \cite[Chapter 3]{simeoni2020functional}) can be used to prove the rate, that is $(n r)^{-2p/(2p+d)} + n^{-1}$, for $p > d$ and $d/2 < q \le p$.

\begin{remark}[On the distribution of the measurement locations]\label{rmk::loc-distr}
Theorem \ref{th::mu} is stated and proved for independent measurament locations $U_{ij}, i=1,\dots,n, j=1,\dots,r_i$, \emph{uniformly distributed} over the sphere. It will be clear in the proofs below that, for this specific case, one can work directly with the $L^2$- and $\Hp$- norms introduced in Section \ref{sec:background}, \emph{i.e.}, $\|\cdot\|_\rangleL$ and $\|\mathscr{D}\cdot\|_\rangleL$, and the corresponding \emph{intermediate} spaces (see \cite[Chapter 4]{odenreddy} for further details on intermediate norms and spaces). The set of spherical harmonics $\{Y_{\ell,m}\}$ also plays a crucial role, as it forms an orthonormal basis for $L^2(\Sph^2)$ and an orthogonal basis for $\Hp$. However, the results can be generalized to any other (common) probability density $\pu(\cdot)$ supported on $\Sph^2$ and bounded away from 0 and infinity, that is, $$0< \inf_{u \in \Sph^2} \pu(u) \le \sup_{u \in \Sph^2} \pu(u) < \infty.$$ Therefore, we can define two new inner products
\begin{align}
&(f,g) \mapsto 4\pi \int_{\mathbb{S}^2} f(u)g(u) \pu(u) du, \label{eq::inner1} \\
&(f,g) \mapsto 4 \pi \int_{\mathbb{S}^2} f(u)g(u) \pu(u) du + \langle \mathscr{D}f, \mathscr{D}g \rangle_\rangleL, \label{eq::inner2}
\end{align}
and their corresponding norms which are respectively equivalent to $\|\cdot\|_\rangleL$ and $\|\mathscr{D}\cdot\|_\rangleL$. 
Such equivalence allows to establish results parallel to those in the uniform case. Indeed, it is possible to show that there exists a set of functions $\{\psi_j\}$ that forms an orthonormal basis for $L^2(\mathbb{S}^2)$ endowed with \eqref{eq::inner1} and an orthogonal basis for $\Hp$ endowed with \eqref{eq::inner2}. We are also able to define intermediate norms and associated intermediate spaces. The reader is referred in particular to \cite[Section 2.4]{linthesis} and the references therein. It is clear the the set $\{\psi_j\}$ plays the same role as the set of spherical harmonics and we want to stress that (apart from orthonormality/orthogonality) no other specific properties of spherical harmonics have been used to prove our results. Hence, the generalization can be obtained by following exactly the same steps; see also \cite{Lin2000} for an application.
\end{remark}


\subsection{Covariance Estimation}

In this section, we present our main result, which concerns the estimation of the covariance function $C(\cdot, \cdot)$. The following theorem gives a (uniform) rate of convergence for the estimator given in Equation \eqref{eq::cov-est}, under a suitable condition on the decay of the penalty parameter $\pen$. The proof is provided in Section~\ref{sec:proofs3}.

\begin{theorem}\label{th::R}
Assume that 
$\mathbb{E}[\epsilon_{11}^4]<\infty$ and the $U_{ij}, i=1,\dots,n, j=i,\dots,r_i$, are independent copies of $U\sim \operatorname{Unif}(\Sph^2)$. Let $p>2$ and $1<q\le p$, and consider the estimation problem in Equation \eqref{eq::cov-est} for an admissible operator $\mathscr{D}$ of spectral growth order $p$. If $\pen \asymp (n r/\log n )^{-p/(p+1)}$, then
\begin{equation*}
\lim_{D\to \infty}  \limsup_{n\to\infty} \sup_{\mathbb{P}_X \in \pit}  \mathbb{P}\left (\| C_\pen  - C \|^2_\rangleLL > D\left (\left( \frac{\log n}{nr}\right)^{p/(p+1)} + n^{-1}\right ) \right ) = 0.
\end{equation*}
\end{theorem}

\begin{corollary} \label{coro:asymptoticCetabigO}
Let $X$ be such that $\mathbb{P}_X \in \pit$. Under the same assumptions of Theorem \ref{th::R},
\begin{equation*}
    \| C_\pen  - C\|^2_\rangleLL = O_\mathbb{P}\left ( \left( \frac{\log n}{n r}\right)^{p/(p+1)} + n^{-1}\right).
\end{equation*}
\end{corollary}

\begin{remark}
A careful inspection of the proof reveals that any estimator $\hat{\mu}(\cdot)$ of the mean function that satisfies
\begin{equation*}
\lim_{D\to \infty}  \limsup_{n\to\infty} \sup_{\mathbb{P}_X \in \pit}  \mathbb{P}\left (\| \hat{\mu}  - \mu \|^2_\rangleLL > D\left (\left( \frac{\log n}{nr}\right)^{p/(p+1)} + n^{-1}\right ) \right ) = 0
\end{equation*}
can be used. The estimator $\mu_\pen(\cdot)$ proposed in Section \ref{sec::mean} satisfies this condition both when $\pen \asymp (n r)^{-p/(p+1)}$ (see Theorem \ref{th::mu}) or when $\pen \asymp (n r/\log n)^{-p/(p+1)}$ (see Proof of Theorem \ref{th::R}).
\end{remark}

 In Theorem \ref{th::R} we show that the estimator $C_\pen$ achieves a uniform rate $ (\log n /nr)^{p/(p+1)} + n^{-1}$ on the class $\pit$.
 As was the case for the mean, we have a phase transition, this time at $r=n^{1/p} \log n$. Indeed, 
 when the functions are densely sampled, \emph{i.e.}, $r \gg n^{1/p} \log n$, the sampling frequency has no impact and the rate is
$$
 \| C_\pen  - C\|^2_\rangleL = O_\mathbb{P}\left (n^{-1}\right),
$$
which suggests that, with a sufficient number of measurement locations per surface, the covariance function can be estimated as well as if the $X_1,\dots,X_n$ can be observed on the whole domain $\Sph^2$. On the other hand, when the functions are sparsely sampled, 
the rate is jointly determined by $n$ and $r$, namely,
$$
 \| C_\pen  - C\|^2_\rangleL = O_\mathbb{P}\left ( \left( \frac{\log n}{n r}\right)^{p/(p+1)} \right).
$$
We can observe that, the estimator is based on a total of $nr(r-1)$ paired observations. However, the rate is expressed in terms of just $nr$. This can be explained by the fact that we are actually observing $nr$ $W_{ij}$'s and hence there is a significant redundancy among the off-diagonal products.
 


As was the case for the mean, the result can be extended to the hypersphere $\mathbb{S}^d$, $d>2$, with rate $(\log n /(n r))^{2p/(2p+d)} + n^{-1}$, for $p > d$ and $d/2 < q \le p$.

\begin{remark}
A similar reasoning to that expressed in Remark \ref{rmk::loc-distr} can be applied to the covariance estimator. Indeed, if the measurement locations are independently drawn from a common distribution with probability density $\pu(\cdot)$ supported on $\Sph^2$ and bounded away from 0 and infinity,
we can define the inner products
\begin{align}
&(f,g) \mapsto (4\pi)^2 \int_{\mathbb{S}^2} \int_{\mathbb{S}^2} f(u,v)g(u,v) \pu(u)\pu(v) dudv,\label{eq::inner11} \\
&(f,g) \mapsto (4 \pi)^2 \int_{\mathbb{S}^2} \int_{\mathbb{S}^2} f(u,v)g(u,v) \pu(u)\pu(v) dudv + \langle (\mathscr{D}\otimes \mathscr{D})f,  (\mathscr{D}\otimes \mathscr{D})g \rangle_\rangleLL, \label{eq::inner22}
\end{align}
and their corresponding norms which are respectively equivalent to $\|\cdot\|_\rangleLL$ and $\| (\mathscr{D}\otimes \mathscr{D})\cdot\|_\rangleLL$. Accordingly, the set $\{\psi_j \otimes \psi_{j'}\}$ forms an orthonormal basis for $L^2(\mathbb{S}^2\times \mathbb{S}^2)$ endowed with \eqref{eq::inner11} and an orthogonal basis for $\Hprod$ endowed with \eqref{eq::inner22}. All the results can be then extended.
\end{remark}

\subsection{Examples} \label{sec:examples}

In this section, we exemplify the previous results for two classes of spherical Gaussian random fields. This will reveal the interest of considering the random classes $\Pi_2(p,q)$ with distinct values for $p$ and $q$ for the asymptotic analysis of the performance of our estimation strategy.

For $\beta > 1$, we say that the spherical random field $X$ is in the class $\mathcal{C}_\beta$ if 
\begin{equation}
    X = \mathscr{D}^{-1} W
\end{equation}
where $\mathscr{D}$ is an admissible operator with spectral growth order $\beta$ and $W$ is a Gaussian white noise such that $\mathbb{E} [\langle W , f \rangle_{\rangleL} \langle W , g\rangle_{\rangleL} ] = \sigma^2 \langle f,  g\rangle_{\rangleL}$ for any $f,g \in L^2(\Sph^2)$. 
In the next proposition, we quantify the rate of convergence of our estimation strategy for spherical random fields in $\mathcal{C}_\beta$.
The proof of Proposition~\ref{prop:filternoise} is given in Section~\ref{sec:proofs3}.

\begin{proposition}
\label{prop:filternoise}
Let $\beta > 5/2$. Then, we have that $\mathcal{C}_\beta \subset \pit$ for any  $1 \leq q \leq p$ such that $p < \beta - 1/2 $ and $q < \beta - 1$.
Moreover, if $X \in \mathcal{C}_\beta$ 
is a spherical random field with covariance $C$, then 
for any $\epsilon > 0$, there exists an estimator $C_\pen$ given by \eqref{eq:Cetaesti} such that 
\begin{equation} \label{eq:firstclass}
    \| C_\pen  - C\|^2_\rangleLL = O_\mathbb{P}\left ( \left( \frac{\log n}{n r}\right)^{\frac{\beta-1/2}{\beta+1/2} - \epsilon} + n^{-1}\right).
\end{equation}
\end{proposition}


The second class of spherical random fields that we consider is as follows. For $\beta > 1$, we say that $X \in \mathcal{B}_\beta$ if 
\begin{equation}
    X = \mathscr{D}^{-1} \left\{ \sum_{k=1}^Q \xi_k \delta_{u_k} \right\} = \sum_{k=1}^Q \xi_k \psi_{\mathscr{D}}( \langle \cdot , u_k \rangle) 
\end{equation}
with $\mathscr{D}$ an admissible operator of order $\beta$ and with zonal Green's kernel $\psi_\mathscr{D}$, $Q \geq 1$, $u_1,\dots,u_Q \in \Sph^2$ distinct, and $\boldsymbol{\xi}= ( \xi_1, \ldots , \xi_Q) \sim \mathcal{N}(0,\sigma^2 \mathrm{Id})$ an $\text{i.i.d.}$ Gaussian vector. Note that $X$ is simply a random $\mathscr{D}$-spline with Gaussian weights. In particular, $X$ is located in the finite-dimensional space of splines with $Q$ knots at locations $u_1,\dots,u_Q$. 
The asymptotic performance for the estimation of random fields in $\mathcal{B}_\beta$ is quantified in Proposition~\ref{prop:filterdiracstream}, whose proof is in Section~\ref{sec:proofs3}. 

\begin{proposition}
\label{prop:filterdiracstream}
Let $\beta > 2$. Then, we have that $\mathcal{B}_\beta \subset \pit$ for any  $1 \leq q \leq p < \beta - 1$. 
Moreover, if $X \in \mathcal{B}_\beta$ 
is a spherical random field with covariance $C$, then 
for any $\epsilon > 0$, there exists an estimator $C_\pen$ given by \eqref{eq:Cetaesti} such that 
\begin{equation} \label{eq:secondclass} 
    \| C_\pen  - C\|^2_\rangleLL = O_\mathbb{P}\left ( \left( \frac{\log n}{n r}\right)^{\frac{\beta-1}{\beta} - \epsilon} + n^{-1}\right).
\end{equation}
\end{proposition}


\begin{remark}
The spherical random fields in $\mathcal{C}_\beta$ and $\mathcal{B}_\beta$ are both in the Sobolev spaces $\Hp$ for any $p < \beta - 1$. 
However, as seen in the proofs of Propositions~\ref{prop:filternoise} and \ref{prop:filterdiracstream}, their covariances have distinct Sobolev regularities, measured in the tensorial spaces $\Hprod$. We see in Theorem~\ref{th::R} that the Sobolev regularity of the covariance is crucial for the convergence rate of our estimation strategy. For this reason, the distinction between the regularity of the random field (parameter $q$ in $X \in \pit$) and its covariance (parameter $p$ in $X \in \pit$) is crucial to obtain the best possible convergence rate.

In particular, Propositions~\ref{prop:filternoise} and \ref{prop:filterdiracstream} reveal that we can approach the critical bounds $\frac{\beta-1/2}{\beta+1/2}$ for $X \in \mathcal{C}_\beta$ and $\frac{\beta-1}{\beta}$ for $X \in \mathcal{B}_\beta$. 
If we achieve the state-of-the-art results of~\cite{caiYuan2} for the latter, we are able to improve existing rates for the former class $\mathcal{C}_\beta$. 
\end{remark}


\section{Simulations and Computational Aspects}
\label{sec:simulations}
In this section, we investigate the practical feasibility of our estimation methodology. We use  simulated sparse samples of an anisotropic Gaussian spherical field as a test case study. Since this appears to be the first and only numerical procedure for nonparametric estimation of anisotropic spherical covariance operators from sparse functional data, we do not offer a comparison. Rather, our main focus is on demonstrating the practical feasibility of our estimation methodology. Since the estimation procedure described in Section \ref{sec:model} decouples the estimation of the first- and second-order moments, we assume without loss of generality a \emph{zero-mean} Gaussian random field and focus exclusively on the (computationally more challenging) second-order moment estimate \eqref{eq::R-est}, for which we propose, leveraging Theorem \ref{theo:Resti} and  sparse/second-order structure, a \emph{computationally tractable} and \emph{memory-thrifty} implementation. 
For details regarding the implementation of the mean estimate \eqref{eq::mu-est} we refer to the literature on \emph{smoothing splines}, of which \eqref{eq::mu-est} is a specific instance (see Remark \ref{rmk:smoothing_spline}). Smoothing spline estimates are standard practice in non-parametric regression \cite{green1993nonparametric}, and their use in the spherical setting was already proposed in the literature -- see for example \cite[Section 6.4.2]{michel2012lectures} and the references therein as well as \cite[Chapter 9]{simeoni2020functional} where smoothing splines are considered  for the estimation of global temperature anomaly maps from sparse recordings collected by the Argo fleet~\cite{argo2000argo,kuusela2018locally}. 

The simulation code is openly available and fully-reproducible, and is released in the form of a standalone Python 3 Jupyter notebook hosted  on hosted on GitHub: \url{https://github.com/matthieumeo/sphericov/blob/main/companion_nb.ipynb}. Instructions for installing the required dependencies and running the notebook are available at this link: \url{https://github.com/matthieumeo/sphericov#readme}.

\paragraph{Simulation Setup.}
Our simulation setup is as follows. Consider a spherical point set $\mathcal{V}_0=\{v_1, \ldots, v_Q\}\subset \Sph^2$ and let $S^\varepsilon_\nu: \mathbb{R}_+\to \mathbb{R}_+$ denote the \emph{Matérn function} with smoothness parameter $\nu>0$ and scale parameter $\varepsilon>0$ \cite[Eq. 4.16]{williams2006gaussian}.
Define moreover the \emph{spherical Matérn function} $\psi^\varepsilon_\nu: [-1,1]\to \mathbb{R}_+$ as \cite[Section 5.3]{simeoni2021functional} 
\begin{equation}\psi^\varepsilon_\nu(t)=S_\nu^\varepsilon(\sqrt{2-2t}), \qquad \forall t\in[-1,1].\label{eq:spherical_matern}\end{equation}
It is possible to show that $\psi^\varepsilon_\nu$ in \eqref{eq:spherical_matern} is the zonal Green's kernel of a certain {admissible spherical pseudo-differential operator}\footnote{In the sense of Definition \ref{sph-pseudo-diff}.} $\mathscr{D}_\nu^\varepsilon$ with spectral growth $p=2(\nu+1)$ \cite[Section 5.3.1]{simeoni2021functional}. Following the nomenclature of \cite[Section 5.3.1]{simeoni2021functional}, we refer to $\mathscr{D}_\nu^\varepsilon$ as a \emph{Matérn operator}.

We consider an underlying spherical random field $X=\{X(u), \, u \in \mathbb{S}^2\}$ taking the form of a \emph{$\mathscr{D}_\nu^\varepsilon$-sparse random field}~\cite{unser2012introduction}: 
\begin{equation}\mathscr{D}_\nu^\varepsilon X=\sum_{q=1}^Q \xi_q \delta_{v_q}\quad\Leftrightarrow \quad X(u)=\sum_{q=1}^Q \xi_q \psi_\nu^\varepsilon(\langle u, v_q \rangle)\quad \forall u \in\mathbb{S}^2,\label{eq:sparse_random_filed}\end{equation}
where $\bm{\xi}=[\xi_1,\ldots, \xi_Q]\sim \mathcal{N}_Q(\mathbf{0},\mathbf{R})$ for some covariance matrix $\mathbf{R}\in \mathbb{R}^{Q\times Q}$ -- see also Section \ref{sec:examples}. Roughly speaking, the random field \eqref{eq:sparse_random_filed} is the primitive with respect to the pseudo-differential operator $\mathscr{D}_\nu^\varepsilon$ of a random discrete measure composed of finitely many Dirac measures. It is easy to see that $X$ is a \emph{Gaussian random field}, with mean zero and second-order moment given by: 
\begin{equation}
    R(u,v)=\sum_{p,q=1}^Q R_{pq} \psi_\nu^\varepsilon(\langle u, v_p \rangle)\psi_\nu^\varepsilon(\langle v, v_q \rangle), \qquad (u,v)\in\mathbb{S}^2\times\mathbb{S}^2.
    \label{R_parametric_form}
\end{equation}
Note that since $X$ has zero mean, the bivariate function $R$ coincides in this case with the covariance function of the field. 
As described in Section \ref{sec:model}, our input data consists in simulated realizations of noisy random samples of $\text{i.i.d.}$ replicates $X_1, \ldots, X_n$ of the random field $X$: 
$$
W_{ij} = X_i(U_{ij}) + \epsilon_{ij}, \qquad  i =1,\dots,n, \, j =1,\dots,r_i,
$$
where the random sampling locations are distributed uniformly over the sphere  $U_{ij}\stackrel{\text{i.i.d.}}{\sim}\operatorname{Unif}({\mathbb{S}^2})$, and the measurement errors are distributed as $\epsilon_{ij}\stackrel{\text{i.i.d.}}{\sim}\mathcal{N}(0, \sigma^2)$ for some $\sigma>0$. Moreover, the random fields, sampling locations, and measurement errors are all assumed mutually independent. Realizations of the random variables $W_{ij}$ and $U_{ij}$ are denoted in lower-case notation, \emph{i.e.}, $w_{ij}$ and $u_{ij}$ respectively.

\paragraph{Simulation Parameters.}
In our simulations, we set the various parameters listed above to the following values:
\begin{figure}[p!]
\centering
\subfloat[][Diagonal slice $u\mapsto R(u,u), \, u \in \mathbb{S}^2$ (\emph{i.e.}, variance function) of the second-order moment kernel. The locations of the sources $\{v_1, \ldots, v_Q\}$ in  \eqref{eq:sparse_random_filed} are overlaid as black scatters.]{
\includegraphics[width=0.65\linewidth]{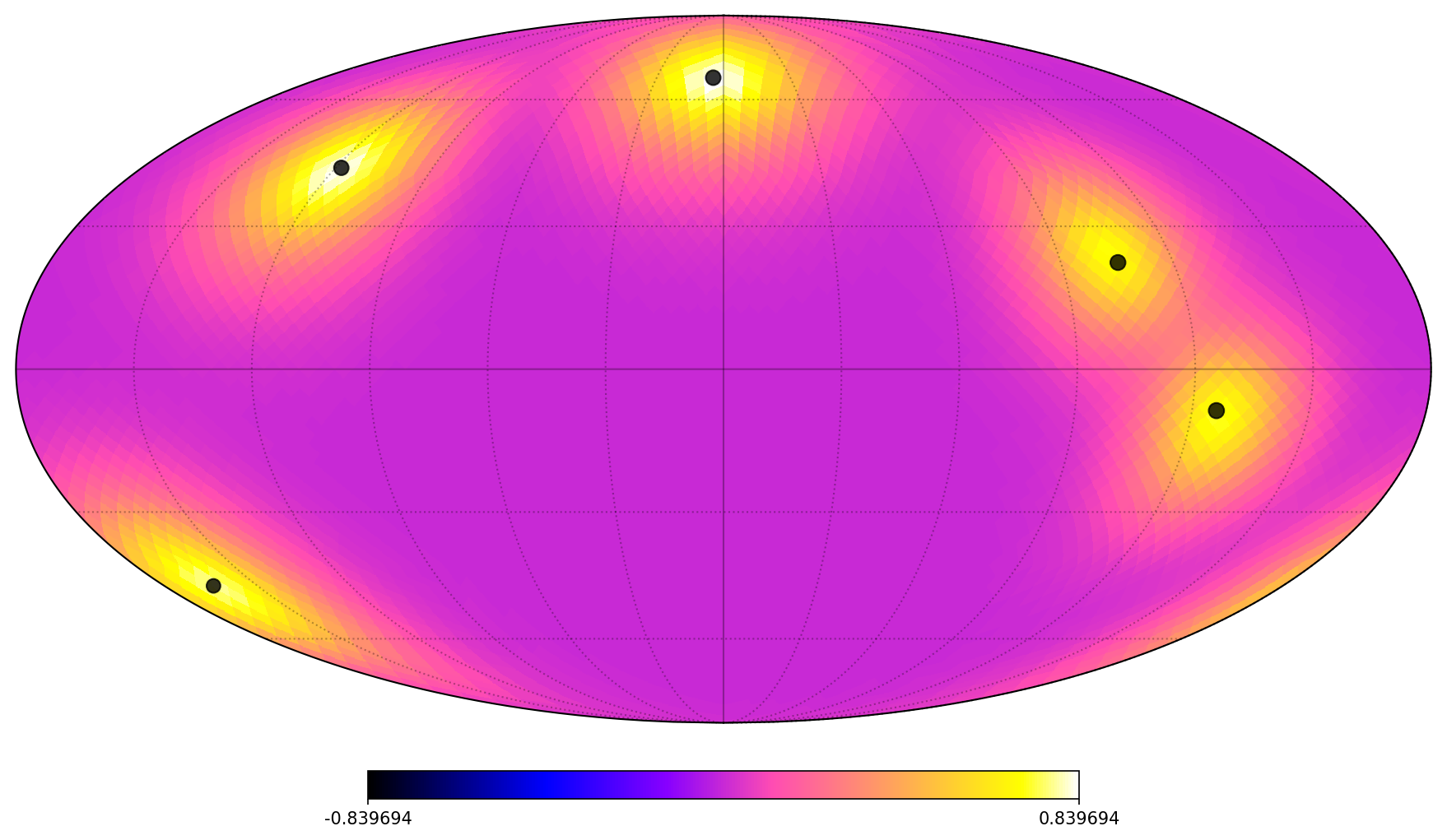}
}

\subfloat[][Slices $u\mapsto R(u,u_p), \, u \in \mathbb{S}^2$ of the second-order moment kernel for 12 points $\{u_1, \ldots, u_P\}$ (overlaid as black crosses).]{
\includegraphics[width=0.9\linewidth]{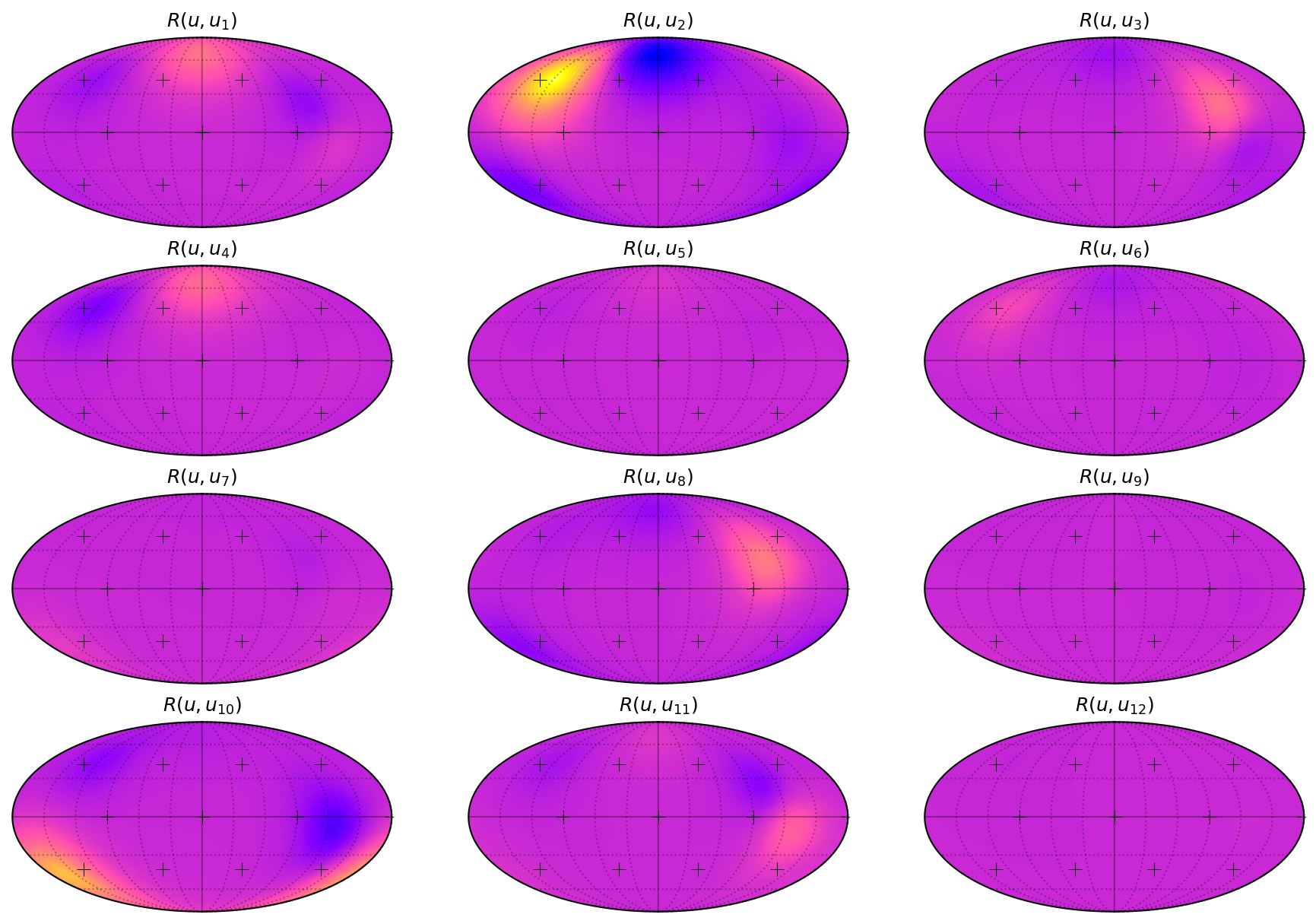}
}
\caption{Mollweide projections of covariance kernel slices $R:\mathbb{S}^2\times\mathbb{S}^2\to \mathbb{R}$ of the Gaussian random field $X$  \eqref{eq:sparse_random_filed} used in our simulations. }
\label{fig:gt_cov_kernel}
\end{figure}
\begin{itemize}
    \item \textbf{Matérn Function:} The smoothing parameter is set to $\nu=5/2$. For such a value of $\nu$, the Matérn function admits the following simple closed-form expression: 
    $$ S^\varepsilon_{5/2}(t)=\left(1+\frac{\sqrt{5} t}{\varepsilon} + \frac{5t^2}{3\varepsilon^2}\right)\exp\left(-\frac{\sqrt{5} t}{\varepsilon}\right), \qquad \forall t \geq 0.$$
    The scale parameter is set to an arbitrary value of $\varepsilon=0.4.$
    \item \textbf{Random Field:} We set the number of sources to $Q=6$, draw the spherical point set $\mathcal{V}_0$ uniformly at random, and choose the covariance matrix $\mathbf{R}$ in \eqref{eq:sparse_random_filed} as: 
    $$\mathbf{R}=\left[\begin{array}{ccccc} 0.812&-0.013&-0.209&-0.416&-0.028\\-0.013&0.974&-0.008&-0.632&-0.372\\-0.209&-0.008&0.909&-0.095&-0.588\\-0.416&-0.632&-0.095&1.000&0.235\\-0.028&-0.372&-0.588&0.235&0.929 \end{array}\right].$$
    We plot in Figure \ref{fig:gt_cov_kernel}  slices of the second-order moment kernel $R$ of the random field $X$ obtained this way. 
    \item \textbf{Data Simulation:} We set the number of replicates of $X$ to $n=64$ and consider a fixed number $r_i=r=12$ of spatial samples per replicates.  The noise level is chosen as $\sigma=0.1$, yielding a \emph{peak signal-to-noise ratio (PSNR)} of 10 dB.
    \end{itemize}

\paragraph{Numerical Estimation Procedure.}
We consider the following estimator for $R$: 
\begin{equation}
R_\eta=\argmin_{g \in \Hprod} \frac{4\pi^2}{L} \sum_{i=1}^n \sum_{1\le j \ne k \le r} (w_{ij}w_{ik}  - g(u_{ij}, u_{ik}))^2 + \pen \| (\mathscr{D}_\xi^\varepsilon \otimes \mathscr{D}_\xi^\varepsilon)  g\|^2_\rangleLL,
\label{eq:R_estimator_numerics}
\end{equation}
for $L=nr(r-1)$ and some regularization parameter $\eta>0$ (the selection of this regularization parameter is addressed  in the next section). The smoothing parameter $\xi$ of the Matérn factors $\mathscr{D}_\xi^\varepsilon$ in \eqref{eq:R_estimator_numerics} is set to $\xi =(\nu-1)/2$, which is the  minimal value so that the unknown second-order moment function $R$ defined in \eqref{R_parametric_form} belongs to the search space $\Hprod$ with $p=2(\nu+1)$. Note that \eqref{eq:R_estimator_numerics} coincides indeed with the second-order moment estimator \eqref{eq::R-est} discussed in Section \ref{sec:model} when the number of random spatial samples $r_i$ per replicate $X_i$ is constant and equal to $r$ (as assumed in this simulation setup). In which case indeed, the normalising constants involved in the least-square term of \eqref{eq::R-est} reduce to a single factor $4\pi^2/L$, hence yielding the simpler expression \eqref{eq:R_estimator_numerics} above. Theorem \ref{theo:Resti} applied to \eqref{eq:R_estimator_numerics} reveals that $R_\eta$ is given by: 
\begin{equation}    
        R_{\eta}(u,v) =  \sum_{i=1}^n \sum_{1 \leq j \neq k \leq r} \beta_{ijk} \psi_\nu^\varepsilon ( \langle u ,u_{ij}\rangle)   \psi_\nu^\varepsilon ( \langle v ,u_{ik}\rangle), \qquad \forall (u,v)\in\Sph^2\times \Sph^2. 
        \label{estimate_closed_form}
\end{equation}
for some coefficients $(\beta_{ijk})_{1 \leq i \leq n, \ 1\leq j \neq k \leq r}$. The latter are moreover obtained  as solutions of the following linear system of size $L$:
\begin{equation}\label{eq:betas_linear_sys}
        \left( \boldsymbol{\mathrm{H}} + \frac{\eta L}{4\pi^2}  \boldsymbol{\mathrm{I}}_L \right)\boldsymbol{\beta} = \boldsymbol{z},
    \end{equation}
where $\boldsymbol{\beta} \in \mathbb{R}^L$ and $\boldsymbol{z} \in \mathbb{R}^L$ are vectorized versions of the coefficients $(\beta_{ijk})_{1 \leq i \leq n, \ 1\leq j \neq k \leq r}$ and data $(w_{ij}w_{ik})_{1 \leq i \leq n, \ 1\leq j \neq k \leq r}$ respectively. In practice, this vectorization is performed by mapping the multi-index triplet  $(i,j,k)$ to a single index $\ell$ as follows: 
\begin{equation}
\ell=(i-1)r(r-1) + (k-1)(r-1) + j - \bm{1}\{j>k\}, \quad i\in \llbracket 1, n\rrbracket, \; k\in \llbracket 1, r\rrbracket, \; j\in \llbracket 1, r\rrbracket\backslash\{k\},
        \label{eq:vectorization_index_map}
\end{equation}
where $\llbracket 1, n\rrbracket=\{1, \ldots, n\}\;\forall n\in\mathbb{N}$ and $\bm{1}\{j>k\}=1$ if $j>k$ and 0 otherwise. This vectorization can easily be performed in practice by filling the off-diagonal terms of $n$ matrices with size $r\times r$ with the coefficients/data (\emph{e.g.}, $\beta_{ijk}$ is put in the $i$-th $r\times r$ matrix, at row $j$ and column $k$), and then considering the vector formed by stacking vertically flattened versions of the matrices, obtained by stacking for each matrix the $r$ columns on top of one another and skipping the diagonal terms. This process is of course reversible, and the map to obtain the multi-index triplets $(i,j,k)$ from the index $\ell$ can be shown to be given by: 
\begin{equation}
\begin{dcases}
i=1 + (\ell\sslash (r(r-1)))\\
k= 1 + (\ell\remainder (r(r-1)))\sslash(r-1)\\
h= (\ell\remainder (r(r-1)))\remainder(r-1) \\ 
j=h + \bm{1}\{h \geq k\}
\end{dcases}, \qquad \ell \in \llbracket 1, L\rrbracket,
\label{eq:vectorization_index_inverse_map}
\end{equation}
where $\sslash$ and $\remainder$ are the \emph{floor-divide} and \emph{remainder} operators defined as:
$$n\sslash m= \left\lfloor n/m \right\rfloor, \qquad n\remainder m= n - \left\lfloor n/m \right\rfloor, \qquad \forall (n,m)\in\mathbb{N}\times \mathbb{N}.$$
With the vectorization scheme \eqref{eq:vectorization_index_map} and \eqref{eq:vectorization_index_inverse_map}, it is possible to show that the matrix $\boldsymbol{\mathrm{H}} \in \mathbb{R}^{L \times L}$ is given by 
\begin{equation}
    \boldsymbol{\mathrm{H}} =\boldsymbol{\mathrm{S}}\left[\begin{array}{ccc}\boldsymbol{\mathrm{J}}_{11}\otimes \boldsymbol{\mathrm{J}}_{11} & \cdots &  \boldsymbol{\mathrm{J}}_{1n}\otimes \boldsymbol{\mathrm{J}}_{1n}\\
    \vdots & \ddots & \vdots\\\boldsymbol{\mathrm{J}}_{n1}\otimes \boldsymbol{\mathrm{J}}_{n1} & \cdots &  \boldsymbol{\mathrm{J}}_{nn}\otimes \boldsymbol{\mathrm{J}}_{nn} \end{array}\right]\boldsymbol{\mathrm{S}}^T
    = \boldsymbol{\mathrm{S}}\left(\boldsymbol{\mathrm{J}}\ast \boldsymbol{\mathrm{J}}\right) \boldsymbol{\mathrm{S}}^T,
    \label{eq:factorization_H}
\end{equation}
where $\otimes$ and $\ast$ denote the \emph{Kronecker} and \emph{block Kronecker} or \emph{Khatri-Rao} products respectively, and: 
\begin{itemize}
    \item $\boldsymbol{\mathrm{S}}\in\mathbb{R}^{L\times nr^2}$ is a subsampled $nr^2$ identity matrix,   where the rows with indices given by $m=(i-1)r^2 + (k-1)r + k, \, i\in \llbracket 1, n\rrbracket, \; k\in \llbracket 1, r\rrbracket$ have been removed.
    \item  $\{\boldsymbol{\mathrm{J}}_{pq}\in\mathbb{R}^{r\times r}, \, (p,q)\in \llbracket 1, n\rrbracket^2\}$ are Gram matrices, with entries given by: 
    \begin{equation}(\boldsymbol{\mathrm{J}}_{pq})_{gh}=\psi_\nu^\varepsilon ( \langle u_{pg} ,u_{qh}\rangle), \qquad \forall (g,h)\in \llbracket 1, r\rrbracket^2.
    \label{eq:submatrix_factor}
    \end{equation}
    \item  $\boldsymbol{\mathrm{J}}\in \mathbb{R}^{nr\times nr}$ is a block matrix given by: 
    $$\boldsymbol{\mathrm{J}}= \left[\begin{array}{ccc}\boldsymbol{\mathrm{J}}_{11} & \cdots &  \boldsymbol{\mathrm{J}}_{1n}\\
    \vdots & \ddots & \vdots\\\boldsymbol{\mathrm{J}}_{n1} & \cdots &  \boldsymbol{\mathrm{J}}_{nn} \end{array}\right]. $$
\end{itemize}
Notice moreover that the submatrices $\boldsymbol{\mathrm{J}}_{pq}$ defined in \eqref{eq:submatrix_factor} are \emph{sparse}\footnote{In practice we observe that fewer than 9 \% of the entries of $\boldsymbol{\mathrm{J}}$ are significantly different from zero.}. Indeed, the Matérn zonal function $\psi_\nu^\varepsilon$ is such that $\psi_\nu^\varepsilon(\cos{\theta})\simeq 0$ when $ |\theta|\geq 3\arccos{\varepsilon}$ \cite[Section 5.3.1]{simeoni2021functional}, and hence the entries of the submatrices $\boldsymbol{\mathrm{J}}_{pq}$ are zero whenever the angle $\theta_{pg,qh}=\arccos{\langle u_{pg} ,u_{qh}\rangle}$ is sufficiently large, \emph{i.e.}, the sampling locations $u_{pg} ,u_{qh}$ are sufficiently far apart on the sphere. In practice, we can leverage this sparsity and the block-Kronecker structure to compute \eqref{eq:factorization_H} efficiently and represent it as a sparse matrix with a relatively low memory footprint. Our procedure is as follows: 
\begin{enumerate}
    \item Compute the submatrices defined in \eqref{eq:submatrix_factor} and sparsify them by setting to zeros all the entries smaller than $0.01 \times \psi_\nu^\varepsilon(1)$. 
    \item Build $\boldsymbol{\mathrm{J}}$ and store it as a sparse matrix \cite{shaikh2015efficient}. 
    \item Compute in a multi-threaded fashion the sparse block Kronecker product $\boldsymbol{\mathrm{J}}\ast \boldsymbol{\mathrm{J}}$. 
    \item Form $\boldsymbol{\mathrm{H}}$ by multiplying the sparse output of the block Kronecker product by $\boldsymbol{\mathrm{S}}$ and  $\boldsymbol{\mathrm{S}}^T$ from the left and right respectively.
\end{enumerate}
These steps are performed in practice by leveraging the \texttt{sparse} \cite{abbasi2018sparse} and \texttt{dask} \cite{rocklin2015dask} Python libraries, used respectively for sparse and distributed  computations. 
Note that without this memory- and compute-efficient procedure, computing $\boldsymbol{\mathrm{H}}$ in practice could quickly become intractable, due to its high dimensionality. For $n=128$ and $r=64$ already, the matrix $\boldsymbol{\mathrm{H}}$ with size $524'288\times 524'288$ would require TFlops/TBytes of computation/memory to be computed/stored as a dense matrix. Leveraging the sparse and block Kronecker structure allows us to bring down this computation/memory footprint by  two to three orders of magnitude approximately, \emph{i.e.}, GFlops/GBytes of computation/memory required.


Upon calculation of the matrix $\boldsymbol{\mathrm{H}}$, we solve \eqref{eq:betas_linear_sys} by means of conjugate gradient descent. Our implementation leverages the routine \texttt{cg()} from Scipy's module \texttt{scipy.sparse.linalg}~\cite{virtanen2020scipy}, which is compatible with sparse matrices. For the simulation setup under consideration, solving \eqref{eq:betas_linear_sys} numerically took on the order of a few dozen seconds.

\paragraph{Accuracy and Selection of the Penalty Parameter.}
We assess the accuracy of our estimator by means of \emph{4-fold cross-validation}. More specifically, we shuffle the data vector $\mathbf{z}$ by applying a random permutation to the latter, split it in 4 groups of equal size, train the model on three of these groups and compute the \emph{out-of-sample mean squared error (MSE)} on the last group. We moreover loop (in parallel) through every combination of training/test configurations and consider the average MSE score over all configurations, yielding the final cross-validation score. We investigate different values of penalty parameters in the range $[1, 6]$ and select the one yielding an estimator with minimal cross-validation score, that is $\tilde{\eta}\simeq 2.363$. The cross-validation scores of the various estimators are plotted in Figure~\ref{fig:cv_score}.

\begin{figure}[t!]
\centering
\includegraphics[width=0.6\linewidth]{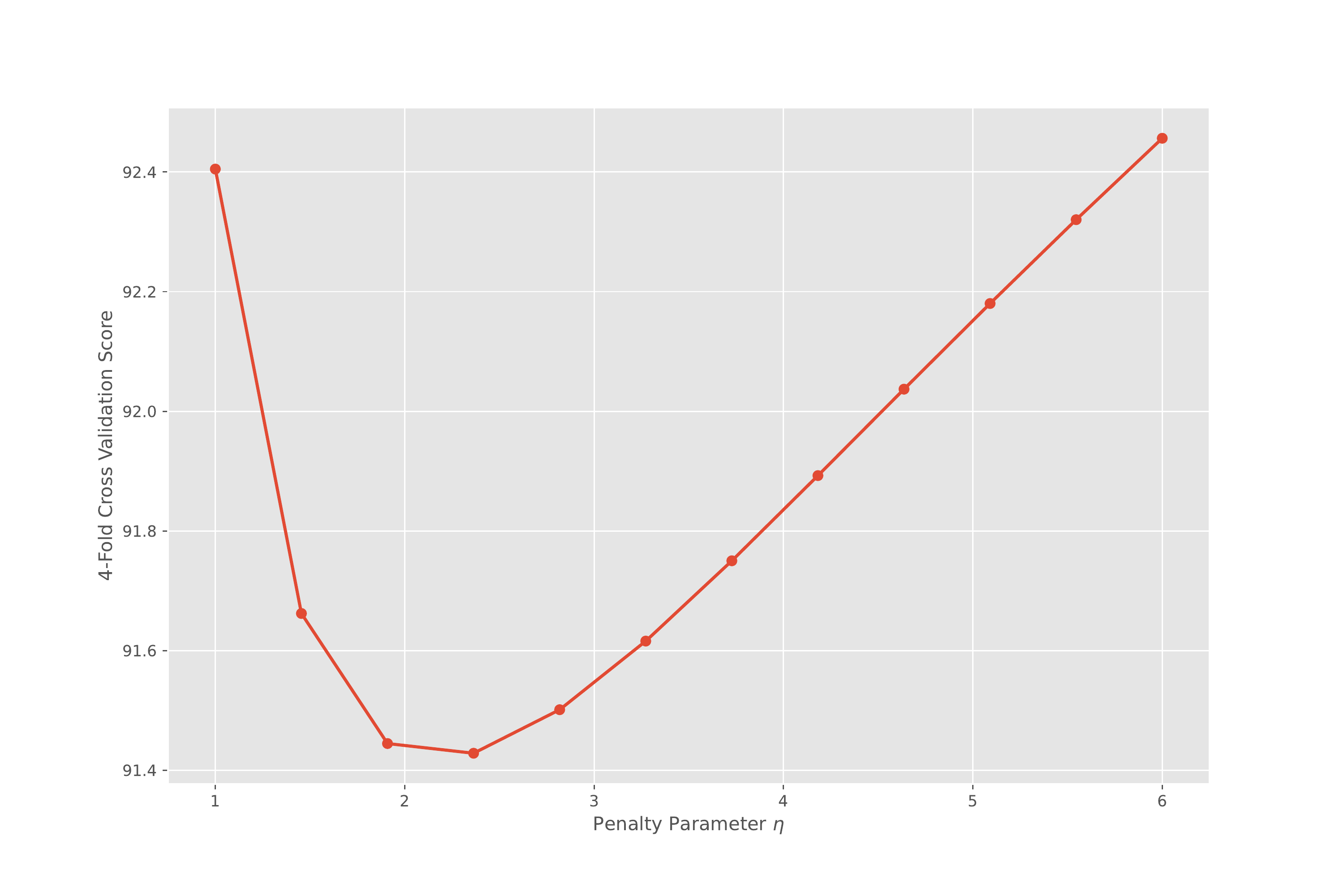}
\caption{4-fold Cross-validation score of various estimates $R_\eta$ for $\eta\in[1, 6]$.}
\label{fig:cv_score}
\end{figure}

\paragraph{Results and Discussion.}
\begin{figure}[p!]
\centering
\subfloat[][Diagonal slice $u\mapsto R_{\tilde{\eta}}(u,u), \, u \in \mathbb{S}^2$ of the second-order moment estimator. Negative values in the field are marked by a semi-transparent overlay. The locations of the sources in  \eqref{eq:sparse_random_filed} are overlaid as black scatters.]{
\includegraphics[width=0.65\linewidth]{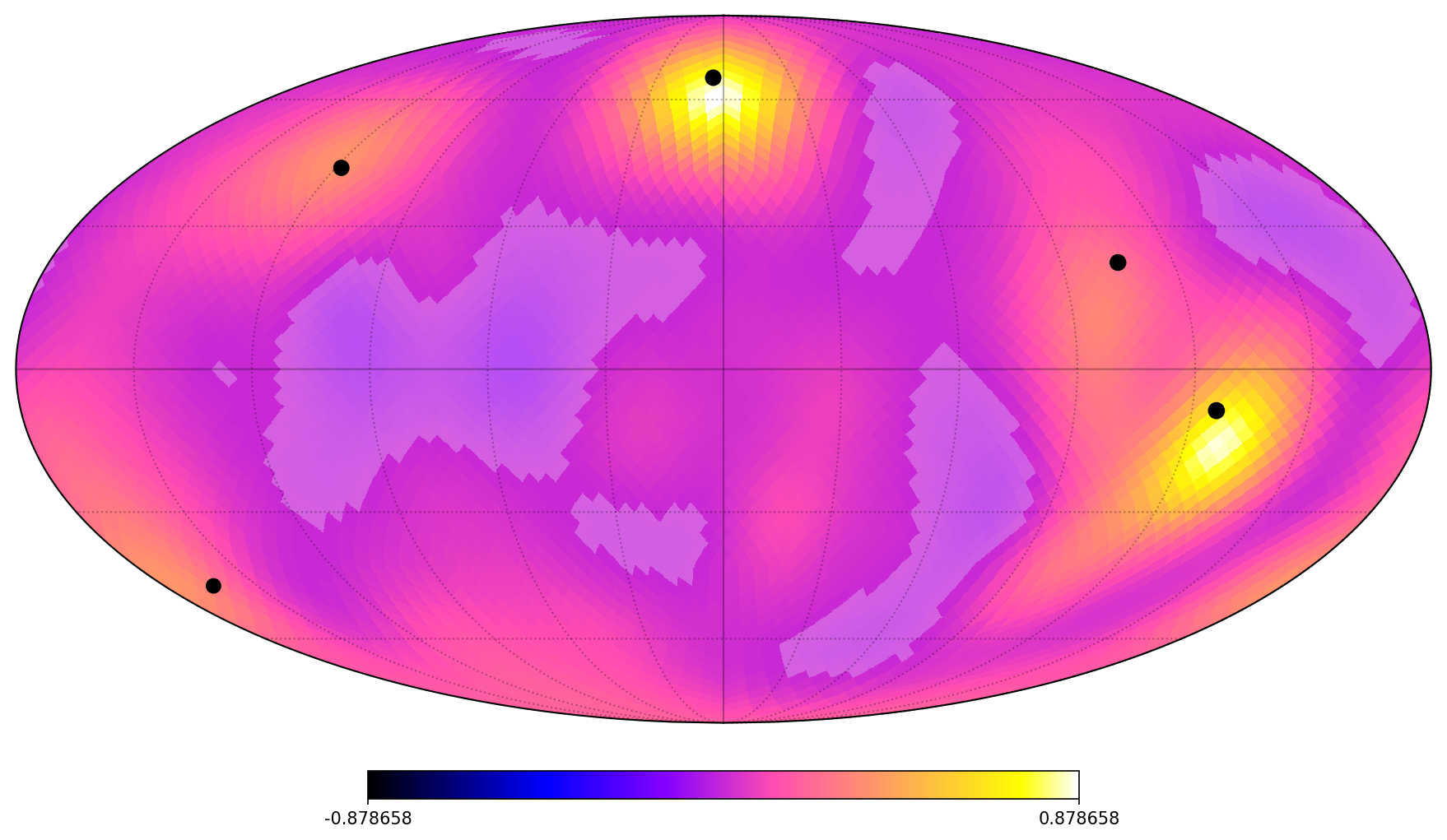}
\label{subfig:intensity_est}
}

\subfloat[][Slices $u\mapsto R_{\tilde{\eta}}(u,u_p), \, u \in \mathbb{S}^2$ of the second-order moment estimator for 12 points $\{u_1, \ldots, u_P\}$ (overlaid as black crosses).]{
\includegraphics[width=0.9\linewidth]{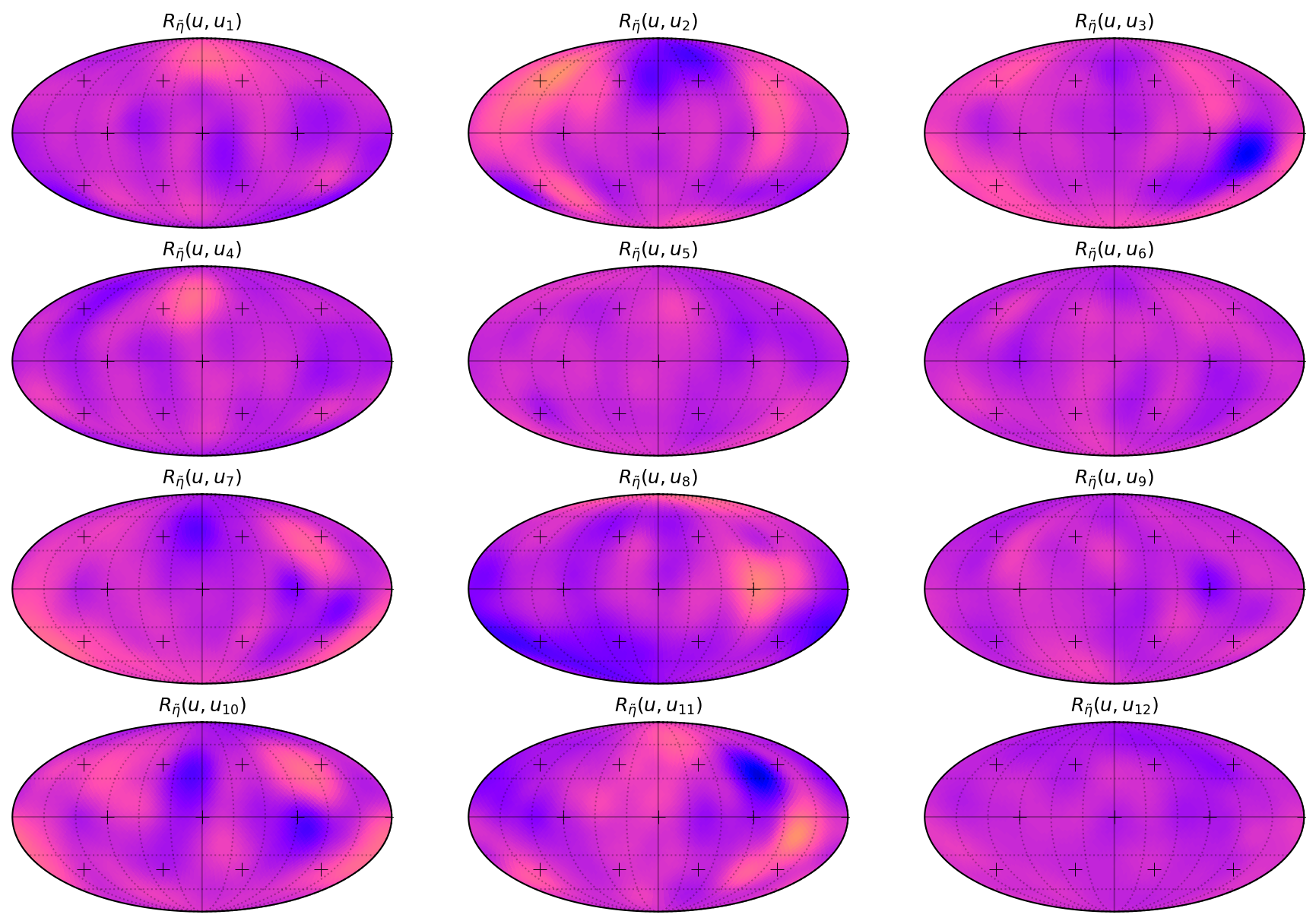}
}
\caption{Mollweide projections of second-order moment estimator slices $R_{\tilde{\eta}}:\mathbb{S}^2\times\mathbb{S}^2\to \mathbb{R}$. }
\label{fig:est_cov_kernel}
\end{figure}

 We plot in Figure \ref{fig:est_cov_kernel}  slices of the second-order moment estimator $R_{\tilde{\eta}}$ corresponding to the optimal penalty parameter value $\tilde{\eta}\simeq 2.363$ determined by cross-validation. A comparison of Figures \ref{fig:gt_cov_kernel}  and \ref{fig:est_cov_kernel} reveals that the estimate $R_{\tilde{\eta}}$ approximates fairly well the actual  second-order function $R$. Indeed, most of the main features of the variance/covariance  maps in Figure \ref{fig:gt_cov_kernel} are still clearly visible in Figure \ref{fig:est_cov_kernel}, although slightly blurred and/or less contrasted. One issue with the estimator however is that it is \emph{not} positive definite, which can be problematic if the latter is used in a kriging context. As a matter of fact, the diagonal $R_{\tilde{\eta}}(u,u)$ is not even always positive, as can be seen in Figure \ref{subfig:intensity_est} where negative values in the field are marked by a semi-transparent overlay. This phenomenon is likely to arise in practical setups with finite sample sizes, since the positive definiteness of our estimator is only guaranteed asymptotically (see Theorems \ref{theo:Resti} and \ref{th::R}). One potential remedy consists in projecting the estimator $R_{\tilde{\eta}}$ on the positive semi-definite cone. This can be achieved approximately in practice by discretizing $R_{\tilde{\eta}}$ on a fine (quasi-)uniform spherical grid~\cite{simeoni2020functional} (\emph{e.g.}, the HEALPix grid \cite{gorski2005healpix}), and setting the negative eigenvalues of the discretized operator to zero. This however, requires computing the eigenvalue decomposition of a potentially very high dimensional operator, which can reveal very computationally and memory intensive. Figure \ref{fig:est_cov_kernel_pd} shows the projection $R^+_{\tilde{\eta}}$ of the  estimator $R_{\tilde{\eta}}$ onto the positive semi-definite cone. The  effect of the projection onto the estimate is mostly visible onto the variance map, which no longer exhibits negative regions and aligns slightly better with the underlying sources.  
 
 \begin{figure}[p!]
\centering
\subfloat[][Diagonal slice $u\mapsto R^+_{\tilde{\eta}}(u,u), \, u \in \mathbb{S}^2$ of the projected second-order moment estimator. The locations of the sources in  \eqref{eq:sparse_random_filed} are overlaid as black scatters.]{
\includegraphics[width=0.65\linewidth]{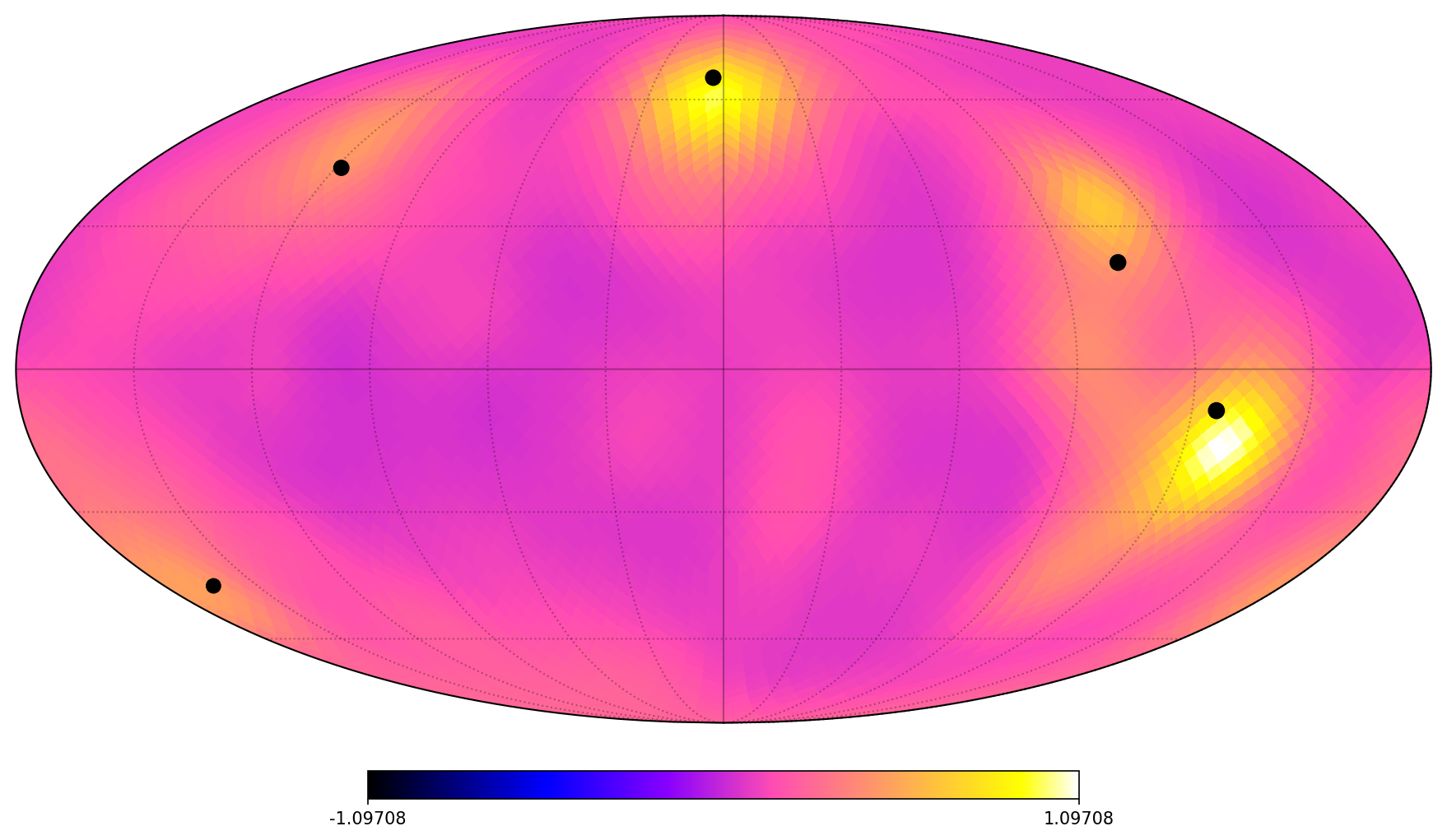}
\label{subfig:intensity_est_proj}
}

\subfloat[][Slices $u\mapsto R^+_{\tilde{\eta}}(u,u_p), \, u \in \mathbb{S}^2$ of the projected second-order moment estimator for 12 points $\{u_1, \ldots, u_P\}$ (overlaid as black crosses).]{
\includegraphics[width=0.9\linewidth]{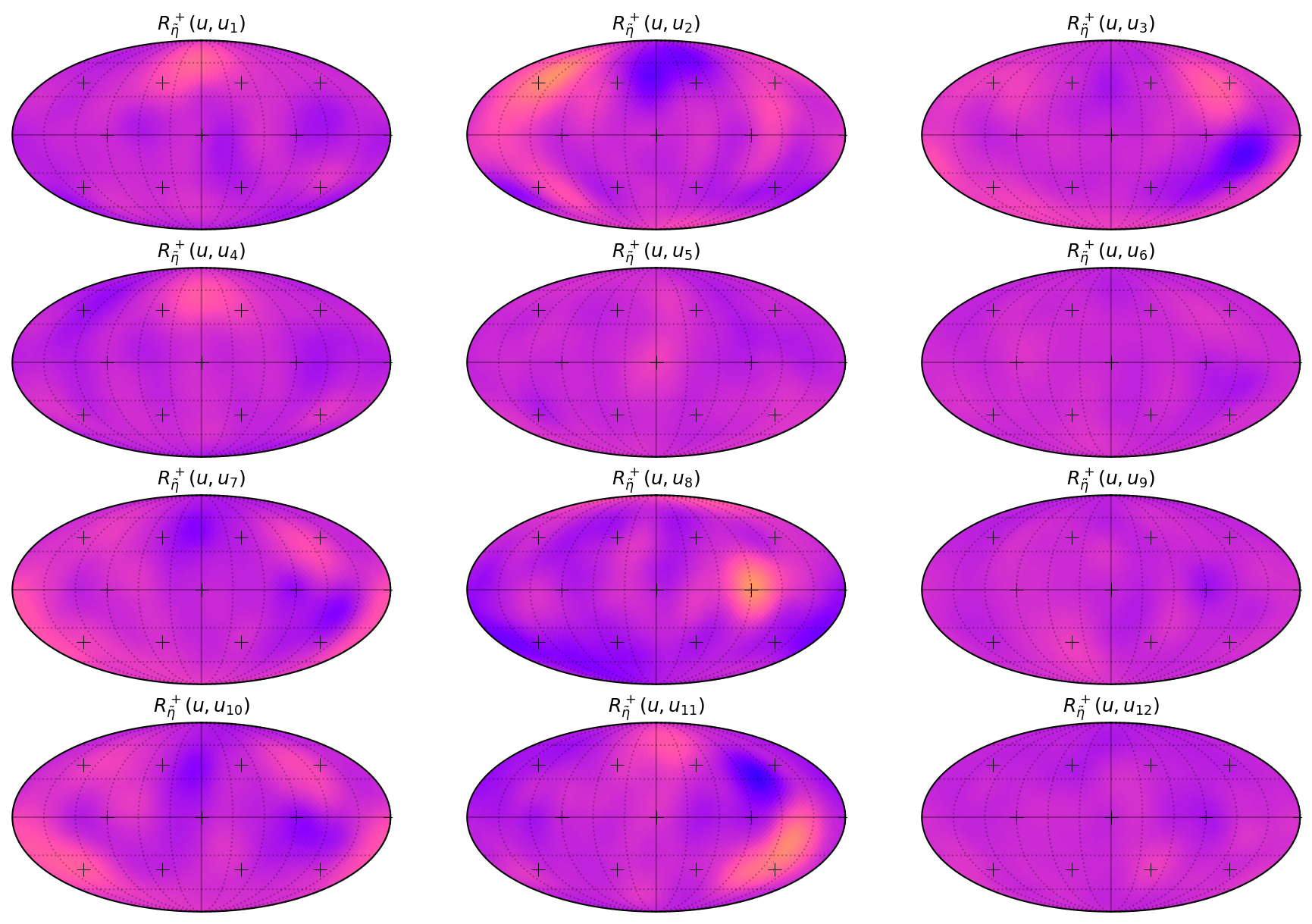}
}
\caption{Mollweide projections of the projected second-order moment estimator slices $R^+_{\tilde{\eta}}:\mathbb{S}^2\times\mathbb{S}^2\to \mathbb{R}$. }
\label{fig:est_cov_kernel_pd}
\end{figure}

\section{Extension to the Serially Dependent Case}\label{sec:temporal}

This section provides an extension of the previous results in the presence of serial dependence. We thus relax the assumption of independence among the $X_1,\dots,X_n$ and we introduce serial dependence, assuming that the indices represent \emph{time}. In doing so, we assume that the resulting temporal sequence is time-stationary (\emph{i.e.}, a stationary functional time series).

Consider the collection of second order processes $X_t=\{X_t(u), \, u \in \mathbb{S}^2\}$, $t \in \mathbb{Z}$, which are also random elements of $\Hq$, $q >1$. 
Throughout, we assume that the sequence $\X=\{X_t, \, t \in \mathbb{Z}\}$ is \emph{strictly stationary}: for any
finite set of indices $I \subset \mathbb{Z}$ and any $h \in  \mathbb{Z}$, the joint law of $\{X_t, \, t \in I \}$ coincides with that of $\{X_{t+h}, \, t \in I \}$.
Then, the first and second order properties of $\mathcal{X}$ are summarized by
$$
\mathbb{E}[X_t (u)] = \mu(u), \qquad  \mathbb{E}[X_{t+\lag}(u)X_t(v)] = R_\lag(u,v),
$$
$$
C_\lag(u,v) = R_\lag(u,v) - \mu(u)\mu(v),
$$
for $u,v \in \mathbb{S}^2$, respectively the mean function, the second-order moment function at \textit{lag} $\lag \in \mathbb{Z}$ and autocovariance function at \textit{lag} $\lag \in \mathbb{Z}$.

If $\mathbb{E}\|X_0\|_\rangleL^k < \infty$ (or $\mathbb{E} \| X_0 \|^k_\Hq< \infty$) for $k \ge 1$, the $k$-th order cumulant kernel is well defined in a $L^2$ sense (or pointwise)
$$
\operatorname{Cum} [X_{h_1}(u_1), \dots,X_{h_{k}}(u_{k})] = \sum_{\pi} (-1)^{|\pi|-1} (|\pi|-1)! \prod_{B \in \pi} \mathbb{E}\left[ \prod_{j \in B} X_{h_j}(u_j)\right],
$$
where the first sum is taken over the list of all unordered partitions of $\{1,\dots,k\}$, and $|\pi|$ is the cardinality of the partition $\pi$.
For $k\ge 1$, we can also define the $2k$-th order cumulant operator $\CL_{h_1,\dots,h_{2k-1}}: L^2( (\Sph^2)^k) \to L^2((\Sph^2)^k)$ as
\begin{align*}
    (\CL_{h_1,\dots,h_{2k-1}} f)(u_1,\dots,u_k)= \int_{(\Sph^2)^k} &\operatorname{Cum} [X_{h_1}(u_1), \dots,X_{h_{2k-1}}(u_{2k-1}), X_0(u_{2k})]\\
    &\times f(u_{k+1}, \dots, u_{2k}) du_{k+1}\cdots du_{2k},
\end{align*}
and the $(2k+1)$-th order cumulant operator $\CL_{h_1,\dots,h_{2k}}: L^2((\Sph^2)^{k+1}) \to L^2((\Sph^2)^{k})$ as 
\begin{align*}
    (\CL_{h_1,\dots,h_{2k}} f)(u_1,\dots,u_k)= \int_{(\Sph^2)^{k+1}} &\operatorname{Cum} [X_{h_1}(u_1), \dots,X_{h_{2k}}(u_{2k}), X_0(u_{2k+1})]\\
    &\times f(u_{k+1}, \dots, u_{2k+1}) du_{k+1}\cdots du_{2k+1}.
\end{align*}

Cumulant kernels and operators in $L^2$ spaces were first introduced in \cite{Panaretos2} for even orders. We extended the definition to odd-order cumulants and we also give an alternative definition in the Sobolev spaces described in Section \ref{sec:background}.

Denote with $\otimes^k \Hq$ the tensor product space of $k$ copies of $\Hq$. 
Clearly, $\otimes^1 \Hq= \Hq$ and  $\otimes^2 \Hq=\mathbb{H}_q$. The inner product $\langle \cdot, \cdot \rangle_{\otimes^{k} \Hq}$ and the norm $\|\cdot \|_{\otimes^{k} \Hq}$ are defined as natural extension of \eqref{Hp} and \eqref{Hprod}.

If $K$ is the reproducing kernel of $\Hq$, we can define $\CHq_{h_1,\dots,h_{2k-1}}: \otimes^k \Hq \to \otimes^k \Hq$ as the operator such that
\begin{align*}
&\operatorname{Cum} [X_{h_1}(u_1), \dots,X_{h_{2k-1}}(u_{2k-1}), X_0(u_{2k})] \\=& 
    \langle \CHq_{h_1,\dots,h_{2k-1}} K(\cdot, u_{k+1})\otimes \cdots \otimes K(\cdot, u_{2k}), K(\cdot, u_{1})\otimes \cdots \otimes K(\cdot, u_{k})  \rangle_{\otimes^k \Hq},
\end{align*}
and $\CHq_{h_1,\dots,h_{2k}}:\otimes^{k+1} \Hq \to \otimes^k \Hq$ such that
\begin{align*}
&\operatorname{Cum} [X_{h_1}(u_1), \dots,X_{h_{2k}}(u_{2k}), X_0(u_{2k+1})] \\=& 
    \langle \CHq_{h_1,\dots,h_{2k}} K(\cdot, u_{k+1})\otimes \cdots \otimes K(\cdot, u_{2k+1}), K(\cdot, u_{1})\otimes \cdots \otimes K(\cdot, u_{k})  \rangle_{\otimes^{k} \Hq}.
\end{align*}
This mapping can be extended to every $f_1,\dots,f_{2k+1} \in \Hq$, so that
\begin{align*}
&\operatorname{Cum} [\langle X_{h_1}, f_1 \rangle_\Hq, \dots , \langle X_{h_{2k-1}}, f_{2k-1} \rangle_\Hq,  \langle X_0, f_{2k} \rangle_\Hq ] \\=& 
    \langle \CHq_{h_1,\dots,h_{2k-1}} f_{k+1} \otimes \cdots \otimes f_{2k}, f_1 \otimes \cdots \otimes f_k  \rangle_{\otimes^k \Hq},
\end{align*}
and
\begin{align*}
&\operatorname{Cum} [\langle X_{h_1}, f_1 \rangle_\Hq, \dots , \langle X_{h_{2k}}, f_{2k} \rangle_\Hq,  \langle X_0, f_{2k+1} \rangle_\Hq ] \\=& 
    \langle \CHq_{h_1,\dots,h_{2k}} f_{k+1}\otimes \cdots \otimes f_{2k+1}, f_1 \otimes \cdots \otimes f_k  \rangle_{\otimes^{k} \Hq}.
\end{align*}

We can define the Hilbert-Schmidt and trace norms of $\CHq_{h_1,\dots,h_{2k-1}}$ and $\CHq_{h_1,\dots,h_{2k}}$ as usual for operators defined in Hilbert spaces, considering the fact that the set
$$\left\{\frac{Y_{\ell_1, m_1}}{(1+\ell_1(\ell_1+1))^{q/2}} \otimes \cdots \otimes \frac{Y_{\ell_k,m_k}}{(1+\ell_k(\ell_k+1))^{q/2}}\right\}_{\ell_1,\dots, \ell_k, m_1,\dots,m_k}$$
forms an orthonormal basis for $(\otimes^k \Hq, \|\cdot\|_{\otimes^k \Hq})$. In particular, the trace norms will be denoted by $\|\CHq_{h_1,\dots,h_{2k-1}}\|_{\operatorname{TR}, \otimes^k \Hq}$ and $\|\CHq_{h_1,\dots,h_{2k}}\|_{\operatorname{TR}, \otimes^{k+1} \Hq}$.

There is a link between the operators defined in $L^2$ and the operators defined in $\Hq$. For instance, for the $2k$-th order cumulants we have
\begin{align*}
&  \operatorname{Cum} [\langle X_{h_1}, \Y_{\ell_1,m_1} \rangle_\rangleL,\dots, \langle X_{h_k}, \Y_{\ell_k,m_k} \rangle_\rangleL, \langle X_{h_{k+1}}, \Y_{\ell_1,m_1} \rangle_\rangleL, \dots, \langle X_0, \Y_{\ell_k,m_k} \rangle_\rangleL] \\=& \frac{1}{\prod_{i=1}^k (1+\ell_i (\ell_i+1))^{2q}} \operatorname{Cum} [\langle X_{h_1}, \Y_{\ell_1,m_1} \rangle_\Hq,\dots, \langle X_{h_k}, \Y_{\ell_k,m_k} \rangle_\Hq, \langle X_{h_{k+1}}, \Y_{\ell_1,m_1} \rangle_\Hq, \dots, \langle X_0, \Y_{\ell_k,m_k} \rangle_\Hq],
\end{align*}
which implies
\begin{align}\label{eq::covops}
& \langle \CL_{h_1,\dots,h_{2k-1}} \Y_{\ell_{k+1},m_{k+1}} \otimes \dots \otimes \Y_{\ell_{2k},m_{2k}},  \Y_{\ell_{1},m_{1}} \otimes \dots \otimes \Y_{\ell_{k},m_{k}} \rangle_{L^2((\Sph^2)^k)} \\=& \frac{1}{\prod_{i=1}^k (1+\ell_i (\ell_i+1))^{2q}} \langle \CHq_{h_1,\dots,h_{2k-1}} \Y_{\ell_{k+1},m_{k+1}} \otimes \dots \otimes \Y_{\ell_{2k},m_{2k}},  \Y_{\ell_{1},m_{1}} \otimes \dots \otimes \Y_{\ell_{k},m_{k}} \rangle_{\otimes^k \Hq}. \notag
\end{align}
Analogous properties hold if we equip $\otimes^k \Hp$ with $\|\otimes^k \mathscr{D} \cdot\|_{L^2((\Sph^2)^k)}$.

Consider now the regression problem
$$
\W_{tj} = X_t(U_{tj}) + \epsilon_{tj}, \qquad j =1,\dots,r_t, \, t =1,\dots,n,
$$
where the the $U_{tj}$'s are independently drawn from a common distribution on $\mathbb{S}^2$, and the $\epsilon_{tj}$'s are independent and identically distributed measurement errors of mean 0 and variance $0<\sigma^2<\infty$. As before, the process $\X=\{X_t, \, t \in \mathbb{Z}\}$, the measurement locations and the measurement errors are assumed to be mutually independent.
Then,
$$
\mathbb{E}[\W_{tj}|U_{tj}=u_{tj}] = \mu(u_{tj}),
$$
and
$$
\mathbb{E}[\W_{t+\lag,j}\W_{tk}|U_{t+\lag,j}=u_{t+\lag,j}, U_{tk}=u_{tk}] = \begin{cases}
R_{\lag}(u_{t+\lag,j}, u_{tk}) \quad \text{if } \lag \ne 0\\
R_0(u_{tj}, u_{tk}) + \sigma^2\delta_j^k \qquad  \text{if }   \lag = 0
\end{cases} .
$$

The estimators for $\mu$ and $R_0$ are defined as in Equations \eqref{eq::mu-est} and \eqref{eq::R-est}, respectively.
In addition here, we define the estimators of the lag-$h$ autocovariance kernels for $h > 0$, as follows. For  $\lag = 1,\dots,n-1$, we first compute
\begin{equation}\label{eq::Rh-est} 
    R_{\lag;\pen} := \argmin_{g \in \Hprod}  \frac{(4\pi)^2}{n-\lag} \sum_{t=1}^{n-\lag} \frac{1}{r_t^2} \sum_{j=1}^{r_t} \sum_{k=1}^{r_t} (\W_{t+\lag,j}\W_{tk}  - g(U_{t+\lag,j}, U_{tk}))^2 + \pen \|(\mathscr{D}\otimes\mathscr{D})g\|^2_\rangleLL,
\end{equation}
and then we obtain
\begin{equation}\label{eq::Ch-est}
C_{\lag;\pen}(u,v) = R_{\lag;\pen}(u,v) - \mu_\pen(u)\mu_\pen(v).
\end{equation}
For $h <0$, we set $C_{\lag;\pen}: (u,v)\mapsto C_{-\lag;\pen} (v,u)$.
 Observe that the diagonal terms are not removed when $h\ne0$. 
Similarly to Theorem~\ref{theo:Resti}, we can obtain a representer theorem for the unique solution of \eqref{eq::Rh-est}. In particular, $R_{\lag;\pen}$ is a continuous function in $\Sph^2\times \Sph^2$ and is a $(\mathscr{D}\otimes\mathscr{D})$-spline.


Differently from the $\text{i.i.d.}$ case, consistency and rates of convergence for the mean and autocovariance estimators can be proved with additional conditions in the time-domain.  For the mean, a form of functional weak dependence is assumed, \emph{i.e.}, summability of the autocovariance operators (in trace norm). For the autocovariance kernels this concept is extended up to the fourth-order cumulant operators. Note that cumulants mixing conditions arise naturally in this context, since we are essentially dealing with moment-based estimators. See for instance \cite{Panaretos, rubin2020} . 

\begin{definition}
Consider $p>2$, $1<q\le p$. Let $\piotemp$ be the collection of probability measures for stationary sequences of $\Hq$-valued processes such that for any $\X=\{X_t, \, t \in \mathbb{Z}\}$ with probability law $\mathbb{P}_\X \in \piotemp$
$$\sum_{\lag \in \mathbb{Z}}\|\CHq_\lag\|_{\operatorname{TR}, \Hq} \le M,\qquad
\|\mu\|^2_{\Hp} \le K,
$$
for some constants $M, K > 0$.
\end{definition}
\noindent Note that, for any $t\in \mathbb{Z}$,
$$\mathbb{E}\|X_t - \mu\|^2_{\Hq} = \|\CHq_0\|_{\operatorname{TR}, \Hq}.
$$

\begin{definition}
Consider $p>2$, $1<q\le p$.  Let $\pittemp$ be the collection of probability measures for stationary sequences of $\Hq$-valued processes such that for any $\X=\{X_t, \, t \in \mathbb{Z}\}$ with probability law $\mathbb{P}_\X \in \pittemp$
$$
\sum_{h_1\in \mathbb{Z}} \sum_{h_2 \in \mathbb{Z}} \sum_{h_3\in \mathbb{Z}} \|\CHq_{h_1,h_2, h_3} \|_{\operatorname{TR}, \Hqprod} \le L_1, \qquad \sum_{h_1\in \mathbb{Z}} \sum_{h_2 \in \mathbb{Z}} \|\CHq_{h_1,h_2} \|_{\operatorname{TR}, \Hqprod} \le L_2, \qquad  \sum_{\lag \in \mathbb{Z}} \|\CHq_\lag \|_{\operatorname{TR}, \Hq}  \le M,
$$
$$
 \sup_{\lag \in \mathbb{Z} }\|R_\lag\|^2_{\rangleHH} \le K_1, \qquad \|\mu\|^2_{\Hp} \le K_2,
$$
for some constants $L_1,L_2,M,K_1, K_2 > 0$.
\end{definition} 

\begin{remark}
A weaker but less interpretable condition for the autocovariances is to replace the summability of $\CHq_h, \CHq_{h_1,h_2}, \CHq_{h_1,h_2,h_3}$ with the summability of the operator associated with the kernel $\operatorname{Cov}[X_{h_1}(u)X_{h_2}(v), X_{h_3}(w) X_0(z)]$.
 Note also that, when the mean is zero, conditions on odd-order cumulants can be discarded. 
\end{remark}

Now we are able to state the analogue of Theorem \ref{th::mu} and Theorem \ref{th::R} in the time-dependent setting.
Proofs are provided in Section~\ref{sec:proofs4}.

\begin{theorem}\label{th::mu-time}
Assume that the $U_{ij}, i=1,\dots,n, j=i,\dots,r_i$, are independent copies of $U\sim \operatorname{Unif}(\Sph^2)$. Let $p>2$ and $1<q\le p$, and consider the estimation problem in Equation \eqref{eq::mu-est} for an admissible operator $\mathscr{D}$ of spectral growth order $p$. If $\pen \asymp (n r)^{-p/(p+1)}$, then
\begin{equation*}
\lim_{D\to \infty} \limsup_{n\to\infty} \sup_{\mathbb{P}_\X \in \piotemp} \mathbb{P}(\| \mu_\pen  - \mu \|^2_\rangleL > D((n r)^{-p/(p+1)} + n^{-1}) ) = 0.
\end{equation*}
\end{theorem}

\begin{theorem}\label{th::R-time}
Assume that 
$\mathbb{E}[\epsilon_{11}^4]<\infty$ and the $U_{ij}, i=1,\dots,n, j=i,\dots,r_i$, are independent copies of $U\sim \operatorname{Unif}(\Sph^2)$. Let $p>2$ and $1<q\le p$, and consider the estimation problem in Equation \eqref{eq::Ch-est} for an admissible operator $\mathscr{D}$ of spectral growth order $p$.  If $\pen \asymp (n r/\log n )^{-p/(p+1)}$, then
\begin{equation*}
\lim_{D\to \infty}  \limsup_{n\to\infty} \sup_{\mathbb{P}_\X \in \pittemp} \mathbb{P}\left (\| C_{\lag;\pen}  - C_\lag \|^2_\rangleLL > D\left (\left( \frac{\log n}{nr}\right)^{p/(p+1)} + \frac1n\right ) \right ) = 0,
\end{equation*}
for any $h \in \mathbb{Z}$.
\end{theorem}

\begin{remark}
The previous result gives a rate of convergence which holds for any fixed lag $h \in \mathbb{Z}$.
However, its dependence on the lag order $h$ can be made explicit. Write $n_h=n-h$. If $\pen \asymp (n_h r/\log n_h )^{-p/(p+1)}$, then
\begin{equation*}
\| C_{\lag;\pen}  - C_\lag \|^2_\rangleLL =O_{\mathbb{P}}\left( \left( \frac{\log n_h}{n_h r}\right)^{p/(p+1)} + \frac{1}{n_h}\right),
\end{equation*}
provided that $n_h \to \infty.$ This informs us on how many lags we can estimate uniformly with high level of precision.
\end{remark}

\section{Proofs of Formal Statements}\label{sec:proofsformal}

    \subsection{Proofs of Section~\ref{sec:background}}\label{sec:proofs1}

    \begin{proof}[Proof of Proposition~\ref{prop:RKHSprop}]
    The fact that $\Hp$ is a RKHS if and only if $p>1$ is classical; it is for instance a particular case of~\cite[Lemma 5.5]{simeoni2020functional}. We therefore simply have to prove that $\Hp$ is a RKHS if and only if $\Hprod$ is.

    Let $\Hprod'$ be the topological dual of $\Hprod$. It is itself a Hilbert space
    \begin{equation} \label{eq:Hprod'}
        \Hprod' = \left\{ g \in  \mathcal{S}'(\Sph^2\times \Sph^2) , \
       \| g \|_{\Hprod'}^2 :=  \sum_{\ell,\ell'=0}^{\infty} 
       \sum_{m = - \ell}^{\ell}  \sum_{m' = - \ell'}^{\ell'}   
       \frac{\langle g , Y_{\ell,m}\otimes Y_{\ell',m'} \rangle^2}{(1 + \ell(\ell + 1))^p(1 + \ell'(\ell' + 1))^p}    < \infty  \right\} 
    \end{equation}
    for the norm $\|  \cdot  \|_{\Hprod'}$.
    
    Then, $\Hprod$ is a RKHS if and only if the evaluation functionals $g \mapsto g(u,v)$ are continuous from $\Hprod$ to $\R$ for any $(u,v) \in \Sph^2 \times \Sph^2$. This is therefore equivalent to $\delta_{(u,v)} \in \Hprod'$, \emph{i.e.}, $\| \delta_{(u,v)}  \|_{\Hprod'} < \infty$ for any $(u,v) \in \Sph^2 \times \Sph^2$. 
    Since $\langle \delta_{(u,v)} , Y_{\ell,m}\otimes Y_{\ell',m'} \rangle_\rangleL = Y_{\ell,m} (u) Y_{\ell',m'}(v)$, we have that
    \begin{equation} \label{eq:normdirac}
      \| \delta_{(u,v)}  \|_{\Hprod'}^2 
      = 
      \left(  \sum_{\ell=0}^{\infty} 
       \sum_{m = - \ell}^{\ell}    
       \frac{ Y_{\ell,m}(u)^2}{(1 + \ell(\ell + 1))^p} \right) 
      \left(  \sum_{\ell'=0}^{\infty} 
       \sum_{m' = - \ell'}^{\ell'}    
       \frac{Y_{\ell',m'}(v) ^2}{(1 + \ell'(\ell' + 1))^p} \right)       
       = \| \delta_x \|_{\Hp'}\| \delta_y \|_{\Hp'}  
    \end{equation}    
    where $\|  \cdot  \|_{\Hp'}$ is the dual norm on the topological dual $\Hp'$ of $\Hp$. 
    
    The relation \eqref{eq:normdirac} then shows that $\| \delta_{(u,v)}  \|_{\Hprod'} < \infty$ for any $(u,v) \in \Sph^2 \times \Sph^2$ if and only if $\| \delta_{x}  \|_{\Hp'} < \infty$ for any $x \in \Sph^2$, \emph{i.e.}, $\Hprod$ is a RKHS if and only if $\Hp$ is. 
    \end{proof}

    In Proposition~\ref{prop:Disgood}, we formalize the invertibility properties of admissible operators. This allows us to deduce that $f\mapsto \|\mathscr{D} f \|_{L^2(\Sph^2)}$ is a norm equivalent to $f\mapsto \|f\|_\rangleH$ on $\Hp$, as was used in Sections~\ref{sec:sphoperators} and~\ref{sec:representer}. 
    
    \begin{proposition} \label{prop:Disgood}
Let $\mathscr{D}$ be an admissible operator with spectral growth order $p \geq 0$. Then, $\mathscr{D}$ can be uniquely extended as a diffeomorphism  
$\mathscr{D} : \Hp \rightarrow L^2(\Sph^2)$
and  there exist constants $0 < A_1 \leq A_2$ such that, for any $f \in \Hp$,
\begin{equation} \label{eq:lowerupperbounds}
    A_1 \|f\|_\rangleH \leq \| \mathscr{D} f \|_\rangleL \leq    A_2 \|f\|_\rangleH.
\end{equation}
Similarly, $\mathscr{D} \otimes \mathscr{D}$ can be uniquely extended as a diffeomorphism 
$   \mathscr{D} \otimes \mathscr{D} : \Hprod \rightarrow L^2(\Sph^2 \times \Sph^2) $
and there exists $0< B_1 \leq B_2$  
such that, for any $g \in \Hprod$,
\begin{equation} 
    B_1 \|g\|_\rangleHH \leq \| ( \mathscr{D} \otimes \mathscr{D}) g \|_\rangleL \leq    B_2 \|g\|_\rangleHH.
\end{equation}
\end{proposition}

\begin{proof}
Let $f \in \mathcal{S}(\Sph^2)$. We have that 
$$
\| \mathscr{D} f \|_\rangleL^2 = 
\sum_{\ell=0}^\infty |D_\ell|^2 \sum_{m=-\ell}^\ell |f_{\ell,m}|^2.$$
Then, the existence of $A_1, A_2$ in \eqref{eq:lowerupperbounds} for any $f \in \mathcal{S}(\Sph^2)$ easily follows from the condition \eqref{eq:conditionDn} and the definition of $\Hp$. 
The space $\mathcal{S}(\Sph^2)$ being dense in the Hilbert space $\Hp$, this means in particular that $\mathscr{D}$ can be extended in $\Hp$ and that $\mathscr{D} : \Hp \rightarrow L^2(\Sph^2)$ linearly and continuously. Finally, \eqref{eq:lowerupperbounds} also implies that $\mathscr{D}$ is a continuous bijection with continuous inverse $\mathscr{D}^{-1} : L^2(\Sph^2) \rightarrow \Hp$. 
The arguments for tensorial operators $\mathscr{D}\otimes \mathscr{D}$ are similar.
\end{proof}

\subsection{Proofs of Section~\ref{sec:representer}}
\label{sec:proofs2}

Both Theorem~\ref{theo:meanesti} and Theorem~\ref{theo:Resti} can be deduced from general principles on regularized cost functionals over Hilbert spaces, as recalled in the next theorem.

    \begin{theorem}[Representer Theorem on Hilbert Spaces] \label{theo:RTHilbert}
    Let $\mathcal{H}$ be a Hilbert space with inner product $\langle \cdot , \cdot \rangle_{\mathcal{H}}$ and norm 
    $\| \cdot \|_{\mathcal{H}}$, $\lambda > 0$, $L \geq 1$, $\boldsymbol{y} = (y_1,\ldots , y_L) \in \R^L$, and $\nu_1, \ldots , \nu_L$ be linearly independent elements in $\mathcal{H}$. Then, the optimization problem 
    \begin{equation} \label{eq:newoptipbabstract}
        \min_{f \in \mathcal{H}} \quad \sum_{\ell=1}^L (y_\ell - \langle \nu_\ell , f \rangle_{\mathcal{H}} )^2 + \lambda \| f \|_{\mathcal{H}}^2
    \end{equation}
    has a unique solution $\widehat{f}$ such that
    \begin{equation}
        \widehat{f} = \sum_{\ell=1}^L \alpha_\ell \nu_\ell \quad \text{where} \quad \boldsymbol{\alpha} = \left( \boldsymbol{\mathrm{G}} + \lambda \boldsymbol{\mathrm{I}}_L \right)^{-1} \boldsymbol{y}
    \end{equation}
    where $\boldsymbol{\alpha}= (\alpha_1,\ldots , \alpha_L) \in \R^L$ and $\boldsymbol{\mathrm{G}} \in \R^{L \times L}$ is such that $G_{\ell_1,\ell_2} = \langle \nu_{\ell_1}, \nu_{\ell_2}\rangle_{\mathcal{H}}$ for $1\leq \ell_1, \ell_2 \leq L$.     
    \end{theorem}
    
    \begin{remark}[On Theorem~\ref{theo:RTHilbert}]
    The existence and uniqueness of the solution of \eqref{eq:newoptipbabstract} can be deduced from the Hilbert projection theorem. 
    The exact form \eqref{eq:newoptipbabstract} of the solution is classical. It
    is for instance proved in~\cite[Section 3.2]{unser2021unifying} in two steps. First, the optimization problem~\eqref{eq:newoptipbabstract} is shown to admit a unique solution of the form 
    $\widehat{f} = \sum_{\ell=1}^L \alpha_\ell \nu_\ell$. Second, injecting the form of this solution in \eqref{eq:newoptipbabstract}, we observe that $\| \widehat{f} \|_{\mathcal{H}} = \langle \boldsymbol{\alpha}, \boldsymbol{\mathrm{G}} \boldsymbol{\alpha} \rangle$ and $\langle \nu_\ell , \widehat{f} \rangle = ( \boldsymbol{\mathrm{G}} \boldsymbol{\alpha} )_\ell$. Hence, $\boldsymbol{\alpha}$ is solution of the optimization problem 
    \begin{equation}\label{eq:findalpha}
        \min_{\boldsymbol{\alpha} \in \R^L} \| \boldsymbol{y} - \boldsymbol{\mathrm{G}} \boldsymbol{\alpha} \|_2^2 + \lambda \langle \boldsymbol{\alpha}, \boldsymbol{\mathrm{G}} \boldsymbol{\alpha} \rangle. 
    \end{equation}
    Then, we easily show that the optimizer of the finite-dimensional quadratic optimization problem \eqref{eq:findalpha} is $\boldsymbol{\alpha} = \left( \boldsymbol{\mathrm{G}} + \lambda \boldsymbol{\mathrm{I}}_L \right)^{-1} \boldsymbol{y}$. Note that this requires the invertibility of $\boldsymbol{\mathrm{G}}$, which is true because the $\nu_\ell$ are assumed to be linearly independent. 
    \end{remark}
    
        \begin{proof}[Proof of Theorem~\ref{theo:Resti}]
    The proof is divided in two steps. We first prove that the unique solution is given by \eqref{eq:RTformR} and then show that $R_\eta$ is symmetric.  \\
    
    \vspace{-0.6cm}
    \textit{Form of the solution.}
        We show that Theorem~\ref{theo:Resti} is a particular case of Theorem~\ref{theo:RTHilbert}. 
   According to Proposition~\ref{prop:Disgood}, 
   the norms $\|
(\mathscr{D}\otimes \mathscr{D}) \cdot \|_{L^2(\Sph^2 \times \Sph^2)}$ and $\|\cdot \|_\rangleHH$ being equivalent,
   $ (\mathbb{H}_p, \| (\mathscr{D}\otimes \mathscr{D}) \cdot \|_\rangleLL )$ is a Hilbert space. 
   
Let $\psi_{\mathscr{D}^*\mathscr{D}}$ the zonal Green's kernel of $\mathscr{D}^* \mathscr{D}$ (see \eqref{eq:zonalgreen}). For any $1 \leq i \leq n$, $1 \leq j\neq k \leq r_i$, we set 
$\nu_{ijk} =   \psi_{\mathscr{D}^* \mathscr{D}} ( \langle \cdot , u_{ij}\rangle) \otimes \psi_{\mathscr{D}^* \mathscr{D}} ( \langle \cdot , u_{ik}\rangle)$ and we  observe that 
\begin{equation*}
    (\mathscr{D} \otimes \mathscr{D})^* (\mathscr{D} \otimes \mathscr{D}) \{\nu_{ijk}\} = 
    ( \mathscr{D}^* \mathscr{D} \{ \psi_{\mathscr{D}^* \mathscr{D}} ( \langle \cdot , u_{ij}\rangle) \} ) \otimes ( \mathscr{D}^* \mathscr{D} \{ \psi_{\mathscr{D}^* \mathscr{D}} ( \langle \cdot , u_{ik}\rangle) \}  )  = \delta_{u_{ij}} \otimes \delta_{u_{ik}}. 
\end{equation*}
This implies that, for any $g \in \mathcal{H} = \Hprod$,
\begin{align*}
    \langle g, \nu_{ijk} \rangle_{\mathcal{H}} 
    &= 
    \langle (\mathscr{D} \otimes \mathscr{D}) \{g\}, (\mathscr{D} \otimes \mathscr{D}) \{ \nu_{ijk} \} )\rangle_\rangleLL
    = g(u_{ij},u_{ik}).
\end{align*}
    We deduce that \eqref{eq::R-estbis} can be recast as
    \begin{equation} \label{eq:recastsecondmoment}
         \underset{g \in \mathcal{H}}{\min} \sum_{i=1}^n \sum_{1 \leq j \neq k \leq r_i} 
         \left( \frac{w_{ij}w_{ik} }{\sqrt{r_i(r_i-1)}} - \left\langle g , \frac{\nu_{ijk}}{\sqrt{r_i(r_i-1)}} \right\rangle_{\mathcal{H}} \right)^2 + \lambda \| g\|_{\mathcal{H}}
    \end{equation}
    with $\lambda = \frac{n \eta}{(4\pi)^2}$. 
    
    The optimization problem \eqref{eq:recastsecondmoment} corresponds to \eqref{eq:newoptipbabstract} where, for any $1 \leq \ell \leq L = \sum_{i=1}^n r_i(r_i-1)$ and its corresponding $(i,j,k)$ in the vectorization, we set 
    $\nu_\ell =  \nu_{ijk} / \sqrt{r_i(r_i-1)}$. 
    Note that the $\nu_\ell$ are linearly independent since the Dirac impulses $\delta_{u_{ij}}\otimes \delta_{u_{ik}} = (\mathscr{D}^* \mathscr{D} \otimes \mathscr{D}^* \mathscr{D}) \nu_{ijk}$ are.
    We moreover observe that 
    \begin{align*}
        \langle \nu_{i_1j_1k_1}, \nu_{i_2j_2k_2}\rangle_{\mathcal{H}} 
        &= 
        \langle   \nu_{i_1j_1k_1} , (\mathscr{D} \otimes \mathscr{D})^*(\mathscr{D} \otimes \mathscr{D}) \nu_{i_2j_2k_2}\rangle_\rangleLL
        = \langle   \nu_{i_1j_1k_1} , \delta_{u_{i_2 j_2}} \otimes \delta_{u_{i_2k_2}} \rangle_\rangleLL \\
        &=  \nu_{i_1j_1k_1} (u_{i_2 j_2},  u_{i_2k_2}) 
        = \psi_{\mathscr{D}^* \mathscr{D}} ( \langle u_{i_1 j_1} , u_{i_2 j_2} \rangle) \times 
        \psi_{\mathscr{D}^* \mathscr{D}} ( \langle u_{i_1 k_1} , u_{i_2 k_2} \rangle),
    \end{align*}
    which implies that
    \begin{align*}
        H_{\ell_1,\ell_2} &= \langle \nu_{\ell_1} , \nu_{\ell_2} \rangle_{\mathcal{H}} = 
        \frac{\langle \nu_{i_1j_1k_1}, \nu_{i_2j_2k_2}\rangle_{\mathcal{H}}}{\sqrt{r_{i_1}(r_{i_1}-1) r_{i_2}(r_{i_2}-1)}}
        =   \frac{\psi_{\mathscr{D}^* \mathscr{D}} ( \langle u_{i_1 j_1} , u_{i_1 k_1} \rangle) \times 
        \psi_{\mathscr{D}^* \mathscr{D}} ( \langle u_{i_2 j_2} , u_{i_2 k_2} \rangle)}{\sqrt{r_{i_1}(r_{i_1}-1) r_{i_2}(r_{i_2}-1)}}
    \end{align*}   
    where $\ell_1$  (resp. $\ell_2$) corresponds to $(i_1,j_1,k_1)$  (resp. $(i_2,j_2,k_2)$) in the vectorization. With these identifications, Theorem~\ref{theo:RTHilbert} gives \eqref{eq:RTformR} and \eqref{eq:betas}.\\
   
       \vspace{-0.6cm}
    \textit{Symmetry of $R_\eta$.} Any $g \in \Hprod$ can be uniquely decomposed as
    \begin{equation}
        g(u,v) = g_s (u,v) + g_a (u,v)  = \left( \frac{g(u,v) + g(v,u)}{2} \right)  + \left( \frac{g(u,v) - g(v,u)}{2} \right) 
    \end{equation}
    where $g_s$ is symmetric ($g_s(u,v) = g_s(v,u)$) and 
    $g_a$ is antisymmetric ($g_a(u,v) = - g_a(v,u)$). Then, we have that $(\mathscr{D} \otimes \mathscr{D}) g_s$ (resp. $(\mathscr{D} \otimes \mathscr{D}) g_a$) is symmetric (resp. antisymmetric).
    Indeed, the operator $(\mathscr{D} \otimes \mathscr{D})$ maps symmetric (resp. antisymmetric) functions to symmetric (resp. antisymmetric) ones. This fact is obvious for functions $g = f_1 \otimes f_2$ since $(\mathscr{D} \otimes \mathscr{D}) g = \mathscr{D} f_1 \otimes \mathscr{D} f_2$ and extended to any function by density of the span of separable functions in $\Hprod$. 
    Therefore 
    \begin{align} \label{eq:firstbrique}
        \| (\mathscr{D} \otimes \mathscr{D}) g  \|_\rangleLL^2 
        = 
         \| (\mathscr{D} \otimes \mathscr{D}) g_a  \|_\rangleLL^2
         +  \| (\mathscr{D} \otimes \mathscr{D}) g_s  \|_\rangleLL^2
         \end{align}
    due to the fact that the inner product between symmetric and antisymmetric bivariate functions is $0$, hence $\langle (\mathscr{D} \otimes \mathscr{D}) g_s, (\mathscr{D} \otimes \mathscr{D}) g_a \rangle_\rangleLL = 0$. 
    
    For $i=1,\ldots , n$, let $\Sigma_i \in \R^{r_i \times r_i}$ be the matrix such that $\Sigma_i[j,k] =  {w_{ij} w_{ik}}  \delta_j^k$. 
    We also define the operator $\Phi_i : \Hprod \rightarrow \R^{r_i\times r_i}$ such that $\Phi_i(g)[j,k] = g(u_{ij}, u_{ik}) \delta_j^k$. Then, the data fidelity term in \eqref{eq::R-estbis} can be rewritten as 
    \begin{equation}
        \sum_{i=1}^n \frac{1}{r_i(r_i-1)} \sum_{1\le j \ne k \le r_i } (\W_{ij}\W_{ik}  - g(U_{ij}, U_{ik}))^2
        = \sum_{i=1}^n \frac{1}{r_i(r_i-1)} \| \Sigma_i - \Phi_i(g) \|_2^2. 
    \end{equation}
    Then, the matrix $\Sigma_i - \Phi_i(g_s)$ (resp. $\Phi_i (g_a)$) is symmetric (resp. antisymmetric), which implies that, for any $i =1, \ldots , n$, 
    \begin{equation} \label{eq:secondbrique}
        \| \Sigma_i - \Phi_i(g) \|_2^2 = \| (\Sigma_i - \Phi_i(g_s)) - \Phi_i(g_a) \|_2^2
        = \| (\Sigma_i - \Phi_i(g_s)) \|^2_2 + \| \Phi_i(g_a) \|_2^2,
    \end{equation}
    the inner product between one symmetric and one antisymmetric matrix being $0$. 
    
    If we denote by $J(g)$ the cost functional to be optimized in \eqref{eq::R-estbis}, we deduce from \eqref{eq:firstbrique} and \eqref{eq:secondbrique} that 
    \begin{equation}
        J(g) = J(g_s) + \frac{(4\pi)^2}{n} \sum_{i=1}^n \frac{\| \Phi_i(g_a) \|_2^2}{r_i(r_i-1)}  + \eta  \| (\mathscr{D} \otimes \mathscr{D}) g_a  \|_\rangleLL^2 \geq J(g_s). 
    \end{equation}
    In other term, for any $g \in \Hprod$, there exists $g_s \in \Hprod$ symmetric such that $J(g_s) \leq J(g)$. This ensures that the unique minimizer of \eqref{eq::R-estbis} is symmetric, as expected. 
    \end{proof}
    
\subsection{Proofs of Section~\ref{sec:asymptoticanalysis}}\label{sec:proofs3}

The following lemma will be used extensively for proving the rates of the mean and covariance estimators.

\begin{lemma}\label{lemma::sup}
Let $ p > 1$. 
There exists $B_1>0$ such that, for any $f \in \Hp$,
\begin{equation}
    \label{eq:controluniformf}
\sup_{u \in \mathbb{S}^2} |f(u)| \le B_1 \| f \|_\rangleH.
\end{equation}
Similarly, there exists $B_2>0$ such that, for any $g \in \Hprod$,
\begin{equation}
    \label{eq:controluniformg}
\sup_{(u,v) \in \mathbb{S}^2 \times \mathbb{S}^2} |g(u,v)| \le B_2 \| g \|_\rangleHH.
\end{equation}
\end{lemma}

\begin{proof}
According to Proposition~\ref{prop:RKHSprop}, the condition $p > 1$ ensures that $\Hp$ is a RKHS.
For any $u \in \Sph^2$, we denote by $K : \Sph^2 \times \Sph^2 \rightarrow \R$ the reproducing kernel such that $f(u) = \langle K(\cdot , u) , f \rangle_{\rangleH} $. Then, the Cauchy-Schwarz inequality implies that
\begin{equation} \label{eq:intermediatefu}
    |f(u)| = | \langle K(\cdot , u ), f \rangle_{\rangleH} | \leq \| K(\cdot , u) \|_{\rangleH} \| f \|_{\rangleH}.
\end{equation}
The norm of $\Hp$ is isotropic in the sense that $\| \mathcal{R} f \|_{\rangleH} = \| f \|_{\rangleH}$ for any rotation $\mathcal{R}$. This implies that $\| K(\cdot , u) \|_{\rangleH}$ does not depend on $u\in \Sph^2$. Hence, \eqref{eq:intermediatefu} implies \eqref{eq:controluniformf}.
The proof for Equation \eqref{eq:controluniformg} is similar.
\end{proof}

\begin{proof}[Proof of Theorem \ref{th::mu}]
Without loss of generality, we will prove the theorem for $D_\ell= (1+\ell(\ell+1))^{p/2}$, which leads to a penalization term in the $\Hp$ norm. However, all the following steps can be generalized to every admissible operator with spectral growth order $p$ in Definition \ref{sph-pseudo-diff}.

Define
\begin{equation*}
    F_{ \pen} (g) := \frac{4\pi}{n} \sum_{i=1}^n \frac{1}{r_i} \sum_{j=1}^{r_i} (\W_{ij}  - g(U_{ij}))^2 + \pen \|g \|^2_\rangleH,
\end{equation*}
$$\bar{F}_{  \pen}(g) := \mathbb{E}[F_{ \pen} (g)] = 4\pi \mathbb{E}|\W_{11} - \mu(U_{11}) |^2 + \|\mu - g\|^2_\rangleL + \pen \|g \|^2_\rangleH,$$ and let
\begin{equation*}
    \bar{\mu}_\pen := \argmin_{g \in \Hp}  \bar{F}_{  \pen} (g).
\end{equation*}
Also, define 
\begin{equation}\label{eq:mutilda}
\Tilde{\mu}_\pen= \bar{\mu}_\pen - (\bar{F}''_{  \pen})^{-1} F'_{ \pen}(\bar{\mu}_\pen).
\end{equation}
The definitions of $F'_{ \pen}$ and $\bar{F}''_{  \pen}$ will be given in Lemma \ref{lemma::1}, while the existence of $(\bar{F}''_{  \pen})^{-1}$ will be discussed in Lemma \ref{lemma::2}.

Now, write $
\mu_\pen  - \mu  = \mu_\pen - \Tilde{\mu}_\pen + \Tilde{\mu}_\pen - \bar{\mu}_\pen +  \bar{\mu}_\pen - \mu.
$
We shall prove that
\begin{enumerate}
    \item $\|\bar{\mu}_\pen - \mu\|^2_\rangleL \le M_1 \left ((n r)^{-p/(p+1)} + n^{-1}\right)$, \label{proof1}
    \item $\mathbb{E}\|\Tilde{\mu}_\pen - \bar{\mu}_\pen\|^2_\rangleL \le M_2 \left ((n r)^{-p/(p+1)} + n^{-1}\right)$, \label{proof2}
   \item  $\forall \varepsilon>0$, $\lim_{n\to \infty} \sup_{\mathbb{P}_X \in \pio} \mathbb{P}\left(  \|\mu_\pen - \Tilde{\mu}_\pen\|^2_\rangleL > \varepsilon \left ( (n r)^{-p/(p+1)} + n^{-1}\right )\right) = 0$,  \label{proof3}
\end{enumerate}
whenever $\pen \asymp (n r)^{-p/(p+1)}$, for any choice of the sampling distribution in $\pio$.


Note that, for $t>0$, 
\begin{align*}
\mathbb{P}(\| \mu_\pen  - \mu \|^2_\rangleL > t ) &= \mathbb{P}(\| \mu_\pen  - \mu \|_\rangleL > \sqrt{t} )\\
&\le \mathbb{P}(\| \mu_\pen  - \Tilde{\mu}_\pen \|_\rangleL > \sqrt{t}/2 ) + \mathbb{P}(\| \Tilde{\mu}_\pen - \mu \|_\rangleL > \sqrt{t}/2 ).
\end{align*}
The second term satisfies
$$
\mathbb{P}(\| \Tilde{\mu}_\pen - \mu \|_\rangleL > t' ) \le \frac{\mathbb{E}\| \Tilde{\mu}_\pen - \mu \|_\rangleL}{t'} \le \frac{\mathbb{E}\|\Tilde{\mu}_\pen - \bar{\mu}_\pen\|_\rangleL}{t'} + \frac{\|\bar{\mu}_\pen - \mu\|_\rangleL}{t'}
$$
where $t'=\sqrt{t}/2$. By choosing $t=D\left ((n r)^{-p/(p+1)} + n^{-1}\right)$, then
$$
\mathbb{P}(\| \Tilde{\mu}_\pen - \mu \|_\rangleL > t' ) \le \frac{c_0}{\sqrt{D}},
$$
where $c_0$ is a positive constant not depending on the choice of $\mathbb{P}_X \in \pio$.
Moreover, from \ref{proof3}, $\forall \varepsilon >0$,
$$
\lim_{n\to \infty} \sup_{\mathbb{P}_X \in \pio}  \mathbb{P}(\|\mu_\pen - \Tilde{\mu}_\pen\|^2_\rangleL > \varepsilon \left ((n r)^{-p/(p+1)} + n^{-1}\right))=0.
$$
Hence, by taking $\varepsilon=D/4$ and $t=D\left ((n r)^{-p/(p+1)} + n^{-1}\right)$,
$$
\lim_{n\to \infty} \sup_{\mathbb{P}_X \in \pio}  \mathbb{P}(\| \mu_\pen  - \Tilde{\mu}_\pen \|_\rangleL > \sqrt{t}/2 ) = \lim_{n\to \infty} \sup_{\mathbb{P}_X \in \pio}  \mathbb{P}(\| \mu_\pen  - \Tilde{\mu}_\pen \|^2_\rangleL > t/4 ) =0.
$$

For the rest of the proof, it is useful to define an \emph{intermediate} norm $\|\cdot\|_{\alpha}$, $\alpha \in [0,1]$, between $\|\cdot\|_\rangleL$ and $\|\cdot\|_\rangleH$. Let $g \in L^2(\Sph^2)$, 
$$
\|g\|^2_\alpha := \sum_{\ell=0}^\infty \sum_{m=-\ell}^\ell D_\ell^{2\alpha} \langle g, \Y_{\ell,m} \rangle^2_\rangleL.$$
Since $|D_\ell|\ge 1$, for all $\ell \in \mathbb{N}$,
$$
\|g\|_\rangleL = \|g\|_0 \le \|g\|_\alpha \le \|g\|_1 = \|g\|_\rangleH;
$$
moreover, $g \mapsto \|g\|_\alpha$ specifies a norm on $\mathcal{H}_{p\alpha}$ which is equivalent to $\|\cdot\|_{\mathcal{H}_{p\alpha}}$. Note that, since we are considering $D_\ell = (1+\ell(\ell+1))^{p/2}$, the two norms $\|\cdot\|_\alpha$ and $\|\cdot\|_{\mathcal{H}_{p\alpha}}$ are actually identical; however, for generality purposes, we will maintain such distinction in our notation.

Proof of \ref{proof1} follows immediately from the fact that, for $\alpha \in [0,1]$,
\begin{equation}\label{eq::mubar-alphabound}
\| \mu - \bar{\mu}_\pen \|^2_\alpha \le \pen^{1-\alpha} \|\mu\|^2_\rangleH.
\end{equation}
Indeed, for $g \in \Hp$,
$$
\|\mu - g\|_\rangleL^2 + \pen \|g\|_\rangleH^2 = \sum_{\ell=0}^\infty\sum_{m=-\ell}^\ell |\mu_{\ell,m} - g_{\ell,m}|^2 + \pen  \sum_{\ell=0}^\infty\sum_{m=-\ell}^\ell D_\ell^{2} |g_{\ell,m}|^2,
$$
where $\mu_{\ell,m}= \langle \mu, \Y_{\ell,m} \rangle_\rangleL$ and $g_{\ell,m}= \langle g, \Y_{\ell,m} \rangle_\rangleL$. The minimizer is then given by
$$
\bar{\mu}_\pen = \sum_{\ell=0}^\infty\sum_{m=-\ell}^\ell \frac{\mu_{\ell,m}}{1+\pen D_\ell^2}Y_{\ell,m}. 
$$
Hence,
\begin{align*}
    \|\mu - \bar{\mu}_\pen\|^2_\alpha &=\sum_{\ell=0}^\infty\sum_{m=-\ell}^\ell D_\ell^{2\alpha} \left | \mu_{\ell,m} -  \mu_{\ell,m} (1+\pen D_\ell^2)^{-1} \right |^2\\
    &\le \pen^{1-\alpha} \sum_{\ell=0}^\infty\sum_{m=-\ell}^\ell D_\ell^2 |\mu_{\ell,m}|^2 (\pen D_\ell^2)^{1+\alpha} (1+\pen D_\ell^2)^{-2}\\
    &\le \pen^{1-\alpha} \|\mu\|^2_\rangleH.
\end{align*}
and, setting $\alpha=0$,
\begin{equation*}
\|\mu - \bar{\mu}_\pen\|^2_\rangleL \le \pen \|\mu\|^2_\rangleH\le K \pen,
\end{equation*}
which gives the claimed result under the assumptions on $\pen$.

In order to prove \ref{proof2}, we first show that 
\begin{equation}\label{eq::tilde-bar}
\|\Tilde{\mu}_\pen - \bar{\mu}_\pen \|^2_\alpha = \frac14 \sum_{\ell=0}^\infty \sum_{m=-\ell}^\ell \frac{D_\ell^{2\alpha}}{(1+ \pen D_\ell^2)^2} ( F'_{ \pen}(\bar{\mu}_\pen)\Y_{\ell,m})^2.
\end{equation}
By the definitions of $\Tilde{\mu}_\pen$ and $\|\cdot\|_\alpha$,
\begin{align*}
\|\Tilde{\mu}_\pen - \bar{\mu}_\pen \|^2_\alpha &= \|(\bar{F}''_{\pen})^{-1}F'_{ \pen}(\bar{\mu}_\pen) \|^2_\alpha\\
& =\sum_{\ell=0}^\infty \sum_{m=-\ell}^\ell D_\ell^{2\alpha} \langle (\bar{F}''_{  \pen})^{-1}F'_{ \pen}(\bar{\mu}_\pen), \Y_{\ell,m}\rangle_\rangleL^2.
\end{align*}
However,
\begin{align*}
    \langle (\bar{F}''_{  \pen})^{-1}F'_{ \pen}(\bar{\mu}_\pen), \Y_{\ell,m}\rangle_\rangleL &= \frac{1}{D_\ell^2}  \langle (\bar{F}''_{  \pen})^{-1}F'_{ \pen}(\bar{\mu}_\pen), \Y_{\ell,m}\rangle_\rangleH\\
     &=\frac{1}{D_\ell^2}  \langle \mathcal{Q}F'_{ \pen}(\bar{\mu}_\pen), (\Tilde{F}''_{  \pen})^{-1}\Y_{\ell,m}\rangle_\rangleH,
\end{align*}
since $\mathcal{Q}F'_{  \pen}(\bar{\mu}_\pen)$ is the representer of $F'_{  \pen}(\bar{\mu}_\pen)$ and $(\Tilde{F}''_{  \pen})^{-1}$ is self-adjoint.
From Lemma \ref{lemma::2},
\begin{align*}
    (\Tilde{F}''_{  \pen})^{-1}\Y_{\ell,m} = \frac{1}{2}\frac{D_\ell^2}{1+\pen D_{\ell}^2} \Y_{\ell,m};
\end{align*}
thus,
\begin{align*}
    \langle (\bar{F}''_{  \pen})^{-1}F'_{ \pen}(\bar{\mu}_\pen), \Y_{\ell,m}\rangle_\rangleL
     &=\frac{1}{2(1+\pen D_{\ell}^2)}  \langle \mathcal{Q}F'_{ \pen}(\bar{\mu}_\pen), \Y_{\ell,m}\rangle_\rangleH \\&= \frac{1}{2(1+\pen D_{\ell}^2)} F'_{ \pen}(\bar{\mu}_\pen)\Y_{\ell,m}.
\end{align*}
Now observe that $\bar{F}'_{   \pen} (\bar{\mu}_\pen) = 0$ (see \cite[Theorem 3.6.3]{Hsing}). Then, an application of Lemma \ref{lemma::1} reveals that, for any $g \in \Hp$,
\begin{align}\label{eq::Fprime}
   \notag F'_{  \pen} (\bar{\mu}_\pen) g &= F'_{  \pen} (\bar{\mu}_\pen)g - \bar{F}'_{   \pen} (\bar{\mu}_\pen) g\\
    &= -\frac{8\pi}{n} \sum_{i=1}^n  \frac{1}{r_i} \sum_{j=1}^{r_i} (\W_{ij} - \bar{\mu}_\pen(U_{ij}))g(U_{ij}) + 2 \langle \mu - \bar{\mu}_\pen, g \rangle_\rangleL.
\end{align}
Consequently,
$$
\mathbb{E}[F'_{  \pen} (\bar{\mu}_\pen)\Y_{\ell,m}] = \mathbb{E}_\U \mathbb{E}[F'_{  \pen} (\bar{\mu}_\pen)\Y_{\ell,m}|\U] = 0
$$
and
\begin{align*}
    \sum_{m=-\ell}^\ell \mathbb{E}|F'_{  \pen} (\bar{\mu}_\pen)\Y_{\ell,m}|^2 &= \sum_{m=-\ell}^\ell \operatorname{Var}[F'_{  \pen} (\bar{\mu}_\pen)\Y_{\ell,m}]\\
    &=  \frac{(8\pi)^2}{n^2} \sum_{i=1}^n  \frac{1}{r_i^2} \sum_{m=-\ell}^\ell \operatorname{Var}\left[ \sum_{j=1}^{r_i} (\W_{1j} - \bar{\mu}_\pen(U_{1j}))\Y_{\ell,m}(U_{1j})  \right] .
\end{align*}
Using the law of total variance, for a generic $i$ we can write 
\begin{align*}
\sum_{m=-\ell}^\ell \operatorname{Var}\left[ \sum_{j=1}^{r_i} (\W_{1j} - \bar{\mu}_\pen(U_{1j}))\Y_{\ell,m}(U_{1j})  \right] &= \sum_{m=-\ell}^\ell \operatorname{Var} \left [\sum_{j=1}^{r_i} (\mu(U_{1j}) - \bar{\mu}_\pen(U_{1j}))\Y_{\ell,m}(U_{1j}) \right] \\&+ \sum_{m=-\ell}^\ell \mathbb{E}_\U\!\left[ \operatorname{Var}\left [\sum_{j=1}^{r_i} (\W_{1j} - \bar{\mu}_\pen(U_{1j}))\Y_{\ell,m}(U_{1j}) \Bigg| \U \right]\right].
\end{align*}
For the first term on the right hand side of this expression, we have
\begin{align*}
 \sum_{m=-\ell}^\ell \operatorname{Var}\left [\sum_{j=1}^{r_i} (\mu(U_{1j}) - \bar{\mu}_\pen(U_{1j}))\Y_{\ell,m}(U_{1j}) \right] &= r_i \sum_{m=-\ell}^\ell \operatorname{Var}\left [ (\mu(U_{11}) - \bar{\mu}_\pen(U_{11}))\Y_{\ell,m}(U_{11}) \right]  \\
&\le \frac{r_i}{4\pi} \sum_{m=-\ell}^\ell  \int_{\mathbb{S}^2} |\mu(u) - \bar{\mu}_\pen(u) |^2 |\Y_{\ell,m}(u)|^2 du \\
  &\le \frac{B r_i}{4\pi} (2\ell+1) \|\mu - \bar{\mu}_\pen \|^2_\rangleH \\
 &\le  \frac{BK r_i}{4\pi}  (2\ell+1),
\end{align*}
where the last two inequalities are justified by Lemma \ref{lemma::sup}, for some $B>0$, and Equation \eqref{eq::mubar-alphabound} with $\alpha=1$. Then, for the second term,
\begin{align*}
    &\sum_{m=-\ell}^\ell \mathbb{E}_\U\!\left[ \operatorname{Var}\left [\sum_{j=1}^{r_i} (\W_{1j} - \bar{\mu}_\pen(U_{1j}))\Y_{\ell,m}(U_{1j}) \Bigg| \U \right]\right] \\
    \le& \sum_{m=-\ell}^\ell \sum_{j=1}^{r_i} \sum_{j'=1}^{r_i} \mathbb{E}_\U\!\left[ \Y_{\ell,m}(U_{1j}) \Y_{\ell,m}(U_{1j'}) \mathbb{E}[\W_{1j}\W_{1j'}| \U] \right]\\
    =&\,  \frac{r_i(r_i-1)}{(4\pi)^2} \sum_{m=-\ell}^\ell \mathbb{E}\langle X, \Y_{\ell, m}  \rangle^2_\rangleL  \\
    +& \, \frac{r_i}{4\pi} \sum_{m=-\ell}^\ell  \int_{\mathbb{S}^2}  R(u,u)  |\Y_{\ell,m}(u)|^2 du +  \frac{r_i}{4\pi} (2\ell+1) \sigma^2  \\
    \le &\,  \frac{r_i(r_i-1)}{(4\pi)^2} \sum_{m=-\ell}^\ell \mathbb{E}\langle X, \Y_{\ell, m}  \rangle^2_\rangleL  \\
    +& \, \frac{B' r_i}{4\pi} (2\ell+1) \mathbb{E}\|X\|^2_\Hq \, + \,  \frac{r_i}{4\pi} (2\ell+1) \sigma^2 ,
\end{align*}
again by applying Lemma \ref{lemma::sup}, for some $B'>0$. Hence, combining all the bounds,
$$
\sum_{m=-\ell}^\ell \mathbb{E}|F'_{  \pen} (\bar{\mu}_\pen)\Y_{\ell,m}|^2 \le \frac{4}{n} \sum_{m=-\ell}^\ell \mathbb{E}\langle X, \Y_{\ell, m}  \rangle^2_\rangleL + (2\ell+1) O\left( \frac{1}{nr} \right),
$$
and
\begin{equation*}
\mathbb{E}\|\Tilde{\mu}_\pen - \bar{\mu}_\pen \|^2_\alpha \le \frac{1}{n} \sum_{\ell=0}^\infty \sum_{m=-\ell}^\ell \frac{D_\ell^{2\alpha}}{(1+ \pen D_\ell^2)^2} \mathbb{E}\langle X, \Y_{\ell, m}  \rangle^2_\rangleL + O\left( \frac{1}{nr} \right) \sum_{\ell=0}^\infty \frac{D_\ell^{2\alpha}}{(1+ \pen D_\ell^2)^2} (2\ell+1),
\end{equation*}
where $r$ is the harmonic mean of $r_1,\dots,r_n$.
Now, for $\alpha \le q/p$,
\begin{align*}
   \sum_{\ell,m}  \frac{D_\ell^{2\alpha} }{(1+ \pen D_\ell^2 )^2} \mathbb{E}\langle X, \Y_{\ell, m}  \rangle^2_\rangleL  &\le \mathbb{E} \left [\sum_{\ell,m} D_\ell^{2\alpha}  \langle X, \Y_{\ell, m}  \rangle^2_\rangleL \right]\le \mathbb{E}\|X\|^2_{q/p},
\end{align*}
which is bounded (possibly up to an arbitrary constant) by $\mathbb{E}\|X\|^2_{\Hq}$.
Moreover, from \cite{Lin2000},
\begin{align*}
\sum_{\ell=0}^\infty  \frac{D_\ell^{2\alpha} }{(1+ \pen D_\ell^2 )^2} (2\ell+1)  =O\left(1+ \pen^{-(\alpha+1/p)}   \right).
\end{align*}
Thus, we have that
$$
\mathbb{E}\|\Tilde{\mu}_\pen - \bar{\mu}_\pen \|^2_\alpha \le M_2 \left( (nr)^{-1}\pen^{-(\alpha+1/p)}+  n^{-1}  \right),
$$
where $M_2$ is a positive constant not depending on the choice of $\mathbb{P}_X \in \pio$. By choosing $\alpha=0$, we obtain the claimed rate.

Now, let us prove \ref{proof3}.
The first step is to obtain a useful analytic form for $\mu_\pen - \Tilde{\mu}_\pen$, by observing that
\begin{align*}
\mu_\pen - \Tilde{\mu}_\pen &= \mu_\pen - \bar{\mu}_\pen +  (\bar{F}''_{  \pen})^{-1}F'_{ \pen}(\bar{\mu}_\pen) \\
&= (\bar{F}''_{  \pen})^{-1} \left [\bar{F}''_{  \pen}(\mu_\pen - \bar{\mu}_\pen) + F'_{ \pen}(\bar{\mu}_\pen)\right].
\end{align*}
Since $\mu_\pen$ minimizes $F'_{ \pen}$, it holds $F'_{ \pen}(\mu_\pen) = 0$ (see \cite[Theorem 3.6.3]{Hsing}).
Then, for any $g \in \Hp$,
\begin{align*}
     \left [\bar{F}''_{  \pen}(\mu_\pen - \bar{\mu}_\pen) + F'_{ \pen}(\bar{\mu}_\pen)\right] g&=  \left [\bar{F}''_{  \pen}(\mu_\pen - \bar{\mu}_\pen) + F'_{ \pen}(\bar{\mu}_\pen) - F'_{ \pen}(\mu_\pen)\right]g\\
    &=\left[\bar{F}''_{   0}(\mu_\pen - \bar{\mu}_\pen) - F''_{  0}(\mu_\pen - \bar{\mu}_\pen)  \right]g.
\end{align*}
where we used Lemma \ref{lemma::1}; in other words,
$$
\mu_\pen - \Tilde{\mu}_\pen  = (\bar{F}''_{  \pen})^{-1}  \left[\bar{F}''_{   0}(\mu_\pen - \bar{\mu}_\pen) - F''_{  0}(\mu_\pen - \bar{\mu}_\pen)  \right].
$$
Now, the same argument that leads to \eqref{eq::tilde-bar} gives us
$$
\| \mu_\pen - \Tilde{\mu}_\pen \|^2_\alpha = \frac14 \sum_{\ell=0}^\infty \sum_{m=-\ell}^\ell \frac{D_\ell^{2\alpha}}{(1+ \pen D_\ell^2)^2} \left ( \left[\bar{F}''_{   0}(\mu_\pen - \bar{\mu}_\pen) - F''_{  0}(\mu_\pen - \bar{\mu}_\pen)  \right]\Y_{\ell,m}\right)^2,
$$
with
\begin{align*}
\left[\bar{F}''_{   0}(\mu_\pen - \bar{\mu}_\pen) - F''_{  0}(\mu_\pen - \bar{\mu}_\pen)  \right]\Y_{\ell,m} &= 2\langle\mu_\pen - \bar{\mu}_\pen, \Y_{\ell,m}\rangle_\rangleL \\
&- \frac{8\pi}{n} \sum_{i=1}^n \frac{1}{r_i} \sum_{j=1}^{r_i} (\mu_\pen(U_{ij}) - \bar{\mu}_\pen(U_{ij})) \Y_{\ell,m}(U_{ij}).
\end{align*}
Now since $\mu_\pen - \bar{\mu}_\pen \in \Hp$, we can write $\mu_\pen - \bar{\mu}_\pen = \sum_{\ell',m'} h_{\ell',m'} \Y_{\ell',m'}$, where the convergence is both in $L^2(\mathbb{S}^2)$ and pointwise.
Then,
\begin{align*}
    \| \mu_\pen - \Tilde{\mu}_\pen \|^2_\alpha =  \sum_{\ell,m} \frac{D_\ell^{2\alpha}}{(1+ \pen D_\ell^2)^2} \left ( \sum_{\ell',m'} h_{\ell',m'} V_{\ell,\ell',m,m'} \right)^2,
\end{align*}
where $$
V_{\ell,\ell',m,m'} = \delta_\ell^{\ell'} \delta_m^{m'} - \frac{4\pi}{n} \sum_{i=1}^n  \frac{1}{r_i} \sum_{j=1}^{r_i} \Y_{\ell,m} (U_{ij} ) \Y_{\ell',m'} (U_{ij} ).
$$
By applying the Cauchy–Schwarz inequality for arbitrary $\theta \in (1/p, 1]$, we obtain
$$
\| \mu_\pen - \Tilde{\mu}_\pen \|^2_\alpha \le \| \mu_\pen - \bar{\mu}_\pen \|^2_\theta \sum_{\ell,m} \frac{D_\ell^{2\alpha}}{(1+ \pen D_\ell^2)^2} \sum_{\ell',m'} D_{\ell'}^{-2\theta} V^2_{\ell,\ell',m,m'}.
$$
It is readily seen that $\mathbb{E}[V_{\ell,\ell',m,m'}]=0$ and 
\begin{align*}
\sum_{m,m'}\mathbb{E}[V^2_{\ell,\ell',m,m'}] &= \sum_{m,m'} \operatorname{Var}[V_{\ell,\ell',m,m'}] = \frac{(4\pi)^2}{nr} \sum_{m,m'} \operatorname{Var}[\Y_{\ell,m} (U_{11})  \Y_{\ell',m'} (U_{11})] \\
&\le \frac{4\pi}{nr} \sum_{m,m'} \int_{\mathbb{S}^2} |\Y_{\ell,m} (u)|^2   |\Y_{\ell',m'} (u)|^2 du\\
&\le D_{\ell}^{2\theta}(2\ell+1)(2\ell'+1) O\left (\frac{1}{nr}\right),
\end{align*}
by Lemma \ref{lemma::sup}, since $\Y_{\ell,m} \in \mathcal{H}_{p\theta}$, and $p\theta >1$. Hence, we obtain
$$
\sum_{\ell,m} \frac{D_\ell^{2\alpha}}{(1+ \pen D_\ell^2)^2} \sum_{\ell',m'} D_{\ell'}^{-2\theta} \mathbb{E}[V^2_{\ell,\ell',m,m'}] \le \frac{M_3}{nr \pen^{\alpha+\theta+1/p}}.
$$
 Let us define
 $$A_{n} := \sum_{\ell,m} \frac{D_\ell^{2\alpha}}{(1+ \pen D_\ell^2)^2} \sum_{\ell',m'} D_{\ell'}^{-2\theta} V^2_{\ell,\ell',m,m'},$$ $a_n:= \frac{1}{nr \pen^{\alpha+\theta+1/p}}$ and $\gamma_n :=  \frac{1}{nr \pen^{\alpha+1/p}} + \frac1n$.
Note that $A_n = \sup_{\mathbb{P}_X \in \pio} A_n$ and $A_n = o_\mathbb{P}(1)$.
Then,
$
\| \mu_\pen - \Tilde{\mu}_\pen \|^2_\alpha \le  A_n \, \| \mu_\pen - \bar{\mu}_\pen \|^2_\theta
$
and therefore
\begin{align*}
\mathbb{P}\left( \| \mu_\pen - \Tilde{\mu}_\pen \|^2_\alpha > \varepsilon \, \gamma_n \right) &\le \mathbb{P}\left( A_n \, \| \mu_\pen - \bar{\mu}_\pen \|^2_\theta > \varepsilon \, \gamma_n \right)\\
&= \mathbb{P}\left( A_n \, \| \mu_\pen - \bar{\mu}_\pen \|^2_\theta > \varepsilon \, \gamma_n, A_n < 1 \right) + \mathbb{P}\left( A_n \, \| \mu_\pen - \bar{\mu}_\pen \|^2_\theta > \varepsilon \, \gamma_n , A_n \ge 1 \right) \\
&\le \mathbb{P}\left( A_n \, \| \mu_\pen - \bar{\mu}_\pen \|^2_\theta > \varepsilon \, \gamma_n, A_n < 1 \right) + \mathbb{P}\left( A_n \ge 1 \right).
\end{align*}
If $A_n < 1$,
\begin{align*}
\|  \Tilde{\mu}_\pen - \bar{\mu}_\pen \|_\theta &\ge \| \mu_\pen - \bar{\mu}_\pen \|_\theta - \| \mu_\pen - \Tilde{\mu}_\pen \|_\theta\\
& \ge (1-\sqrt{A_n})\| \mu_\pen - \bar{\mu}_\pen \|_\theta,
\end{align*}
which allows to write
\begin{align*}
\mathbb{P}\left( A_n \, \| \mu_\pen - \bar{\mu}_\pen \|^2_\theta > \varepsilon \, \gamma_n, A_n < 1 \right)
&\le \mathbb{P}\left( A_n |1-\sqrt{A_n}|^{-2} \|  \Tilde{\mu}_\pen - \bar{\mu}_\pen \|^2_\theta > \varepsilon \, \gamma_n, A_n < 1 \right) \\
&\le  \mathbb{P}\left( A_n |1-\sqrt{A_n}|^{-2} \|  \Tilde{\mu}_\pen - \bar{\mu}_\pen \|^2_\theta> \varepsilon \, \gamma_n \right).
\end{align*}
Let us now define $B_n:= \| \Tilde{\mu}_\pen - \bar{\mu}_\pen \|^2_\theta$ and $b_n:=\frac{1}{nr \pen^{\theta+1/p}} + \frac1n$. Recall that $\mathbb{E}[B_n] \le M_2 b_n$, for $\theta \le q/p$. Moreover, $|1-\sqrt{A_n}|^{-2}=O_\mathbb{P}(1)$.
We can observe that
$$
a_n b_n = \frac{1}{nr \pen^{2\theta + 1/p}} \left ( \frac{1}{nr \pen^{ \alpha + 1/p}} + \frac{\eta^{\theta-\alpha}}{n} \right),
$$
so that $c_n:= a_nb_n/\gamma_n \to 0$, assuming $\theta < 1/2$ (recall that $p>2$) and $\alpha \in [0, \theta]$. Then,
\begin{align*}
\mathbb{P}\left( A_n |1-\sqrt{A_n}|^{-2}   B_n > \varepsilon \, \gamma_n \right) &= \mathbb{P}\left( \frac{A_n |1-\sqrt{A_n}|^{-2}  B_n}{a_n b_n} > \frac{\varepsilon}{c_n} \right)\\
&\le \mathbb{P}\left( \frac{B_n}{b_n} > \frac{\varepsilon^{1/3}}{c_n^{1/3}} \right)+ \mathbb{P}\left( \frac{A_n}{a_n} > \frac{\varepsilon^{1/3}}{c_n^{1/3}} \right) +  \mathbb{P}\left( |1-\sqrt{A_n}|^{-2}   > \frac{\varepsilon^{1/3}}{c_n^{1/3}} \right) \\
&\le M_2 \frac{c_n^{1/3}}{\varepsilon^{1/3}} + M_3 \frac{c_n^{1/3}}{\varepsilon^{1/3}} + \mathbb{P}\left( |1-\sqrt{A_n}|^{-2}   > \frac{\varepsilon^{1/3}}{c_n^{1/3}} \right).
\end{align*}
Clearly $c_n^{1/3}|1-\sqrt{A_n}|^{-2}=o_\mathbb{P}(1)$, hence
$$
\lim_{n\to\infty} \sup_{\mathbb{P}_X \in \pio} \mathbb{P}\left( \| \mu_\pen - \Tilde{\mu}_\pen \|^2_\alpha > \varepsilon \, \gamma_n \right) = 0.
$$
By taking $\alpha=0$ we obtain the claimed result.
\end{proof}

The next two lemmas refer to 
\begin{equation*}
    F_{ \pen} (g) := \frac{4\pi}{n} \sum_{i=1}^n  \frac{1}{r_i} \sum_{j=1}^{r_i} (\W_{ij}  - g(U_{ij}))^2 + \pen \|\mathscr{D} g \|^2_\rangleL,
\end{equation*}
$$\bar{F}_{  \pen}(g) := \mathbb{E}[F_{ \pen} (g)] = 4\pi \mathbb{E}|\W_{11} - \mu(U_{11}) |^2 + \|\mu - g\|^2_\rangleL + \pen \| \mathscr{D} g \|^2_\rangleL.$$

Let $\mathbb{X}_1$ and $\mathbb{X}_2$ be normed spaces. We will use $\mathcal{B}(\mathbb{X}_1, \mathbb{X}_2)$
to denote the set of all linear and bounded operators
from $\mathbb{X}_1$ to $\mathbb{X}_2$.
Here, we consider $\Hp$ endowed with $\|\mathscr{D}\cdot \|_\rangleL$.
\begin{lemma}\label{lemma::1}
Let $f,g,g_1,g_2$ be arbitrary elements of $\Hp$. 
\begin{enumerate}
    \item 
The Fréchet derivative of $F_{\pen}$ at $f$ is the element $F'_{\pen}(f)$ of $\mathcal{B}(\Hp,\mathbb{R})$ characterized by
\begin{equation*}
    F'_{\pen}(f)g = -\frac{8\pi}{nr} \sum_{i=1}^n \sum_{j=1}^r (\W_{ij} - f(U_{ij}))g(U_{ij}) + 2 \pen \langle \mathscr{D} f, \mathscr{D}g \rangle_\rangleL
\end{equation*}
The second Fréchet derivative $F''_{\pen}\in \mathcal{B}(\Hp, \mathcal{B}(\Hp,\mathbb{R}))$ is
characterized by
\begin{equation*}
        F''_{\pen}g_1g_2 = \frac{8\pi}{nr} \sum_{i=1}^n \sum_{j=1}^r g_1(U_{ij}) g_2(U_{ij}) + 2 \pen \langle \mathscr{D} g_1, \mathscr{D}g_2 \rangle_\rangleL.
\end{equation*}

    \item 
The Fréchet derivative of $\bar{F}_{\pen}$ at $f$ is the element $\bar{F}'_{\pen}(f)$ of $\mathcal{B}(\Hp,\mathbb{R})$ characterized by
\begin{equation*}
    \bar{F}'_{\pen}(f)g = - 2 \langle \mu-f,g \rangle_\rangleL + 2 \pen \langle \mathscr{D} f, \mathscr{D}g \rangle_\rangleL
\end{equation*}
The second Fréchet derivative $\bar{F}''_{\pen}\in \mathcal{B}(\Hp, \mathcal{B}(\Hp,\mathbb{R}))$ is
characterized by
\begin{equation}\label{eq::2dv}
        \bar{F}''_{\pen}g_1g_2 = 2 \langle g_1,g_2 \rangle_\rangleL + 2 \pen \langle \mathscr{D} g_1, \mathscr{D}g_2 \rangle_\rangleL.
\end{equation}
\end{enumerate}
\end{lemma}

\begin{proof}
The proof is a direct application of Theorem 3.6.4 in \cite{Hsing}. See also Lemma 8.3.3.
\end{proof}

The evaluation of $\Tilde{\mu}_\pen$, given in Equation \eqref{eq:mutilda}, involves the inverse of the operator $\bar{F}''_{  \pen}$. To this purpose, it is convenient to invoke the Riesz representation theorem (see \cite[Theorem 3.2.1]{Hsing}), which tells us that there is an invertible norm-preserving mapping $\mathscr{Q}$ such that $\mathscr{Q}\mathcal{B}(\Hp, \mathbb{R})= \Hp$. Thus,
\begin{equation}\label{eq::ftilde}
\Tilde{F}''_{   \pen} := \mathscr{Q}\bar{F}''_{  \pen}
\end{equation}
is an element of $\mathcal{B}(\Hp, \Hp)$ and it is invertible if and only if $\bar{F}''_{  \pen}$ is invertible. 

\begin{lemma}\label{lemma::2}
The operator $\Tilde{F}''_{   \pen}$ in \eqref{eq::ftilde} is an invertible element of $\mathcal{B}(\Hp, \Hp)$ and, for any $g \in \Hp$,
$$
(\Tilde{F}''_{   \pen})^{-1}g = \frac12  \sum_{\ell=0}^\ell \sum_{m=-\ell}^\ell \frac{D_\ell^2}{1+\pen D_\ell^2} \langle g, \Y_{\ell,m} \rangle_\rangleL \Y_{\ell,m}.
$$
\end{lemma}
Notice that $(\Tilde{F}''_{  \pen})^{-1}$ is a self-adjoint operator.

\begin{proof}
Take $g_1 \in \Hp$, so that $\bar{F}''_{  \pen}g_1$ belongs to $\mathcal{B}(\Hp,\mathbb{R})$, with representer $\Tilde{F}''_{  \pen}g_1 \in \Hp$. Then, for any $g_2 \in \Hp$,
$$
\bar{F}''_{  \pen}g_1g_2 = \langle \mathscr{D} \Tilde{F}''_{  \pen}g_1, \mathscr{D} g_2 \rangle_\rangleL = 2 \langle g_1,g_2 \rangle_\rangleL + 2 \pen \langle \mathscr{D} g_1, \mathscr{D}g_2 \rangle_\rangleL,
$$
where the last equality comes from \eqref{eq::2dv}. We can hence write the expansion in $\Hp$
$$
\Tilde{F}''_{  \pen}g_1 = 2\sum_{\ell=0}^\infty \sum_{m=-\ell}^\ell \frac{1+\pen D_\ell^2}{D_\ell^2}  \langle g_1, \Y_{\ell,m} \rangle_\rangleL \Y_{\ell,m},
$$
which suggests that $\Tilde{F}''_{  \pen}$ is invertible and 
$$
(\Tilde{F}''_{   \pen})^{-1}g_1 = \frac12 \sum_{\ell=0}^\ell \sum_{m=-\ell}^\ell \frac{D_\ell^2}{1+\pen D_\ell^2} \langle g_1, \Y_{\ell,m} \rangle_\rangleL \Y_{\ell,m}.
$$
\end{proof}


\begin{proof}[Proof of Theorem \ref{th::R}]

As for Theorem \ref{th::mu}, without loss of generality, we will consider $D_\ell= (1+\ell(\ell+1))^{p/2}$, which leads to a penalization term in the $\Hprod$ norm. However, all the following steps can be generalized to every spherical pseudo-differential operator in Definition \ref{sph-pseudo-diff}.

In keeping with the notation that was used for proving Theorem \ref{th::mu}, we first define
\begin{equation*}
    F_{ \pen} (g) := \frac{(4\pi)^2}{n} \sum_{i=1}^n \frac{1}{r_i(r_i-1)} \sum_{1\le j \ne k \le r_i } (\W_{ij}\W_{ik}  - g(U_{ij}, U_{ik}))^2 + \pen \|g\|^2_\rangleHH,
\end{equation*}
$$\bar{F}_{  \pen}(g) := \mathbb{E}[F_{ \pen} (g)] = (4\pi)^2 \operatorname{Var}[\W_{11}\W_{12}] + \|R - g\|^2_\rangleLL + \pen \|g \|^2_\rangleHH,$$
and we let
\begin{equation*}
    \bar{R}_\pen := \argmin_{g \in \Hprod}  \bar{F}_{  \pen} (g).
\end{equation*}
We also define 
\begin{equation}\label{eq:Rtilda}
\Tilde{R}_\pen= \bar{R}_\pen - (\bar{F}''_{  \pen})^{-1} F'_{ \pen}(\bar{R}_\pen).
\end{equation}
The definitions of $F'_{ \pen}$ and $\bar{F}''_{  \pen}$ will be given in Lemma \ref{lemma::1-cov}, while the existence of $(\bar{F}''_{  \pen})^{-1}$ will be discussed in Lemma \ref{lemma::2-cov}.

Now, we can write $
R_\pen  - R  = R_\pen - \Tilde{R}_\pen + \Tilde{R}_\pen - \bar{R}_\pen +  \bar{R}_\pen - R.
$
In parallel to Proof of Theorem \ref{th::mu}, we must show the following
\begin{enumerate}
    \item $\|\bar{R}_\pen - R\|^2_\rangleLL \le M_1 \left ((n r/\log n)^{-p/(p+1)} + n^{-1}\right)$, \label{proof1-R}
    \item $\mathbb{E}\|\Tilde{R}_\pen - \bar{R}_\pen\|^2_\rangleLL \le M_2 \left ((n r/\log n)^{-p/(p+1)} + n^{-1}\right)$, \label{proof2-R}
   \item  $\forall \varepsilon>0$, $\lim_{n\to \infty} \sup_{\mathbb{P}_X \in \pit} \mathbb{P}\left(  \|R_\pen - \Tilde{R}_\pen\|^2_\rangleLL > \varepsilon \left ( (n r/\log n)^{-p/(p+1)} + n^{-1}\right )\right) = 0$, \label{proof3-R}
\end{enumerate}
whenever $\pen \asymp (n r/\log n)^{-p/(p+1)}$, for any choice of the sampling distribution in $\pit$.

At this point, we define the \emph{intermediate} norm $\|\cdot\|_\alpha$, $\alpha \in [0,1]$, between $\|\cdot\|_\rangleLL$ and $\|\cdot\|_\rangleHH$. Let $g \in L^2(\Sph^2 \times \Sph^2)$, then
$$
\|g\|^2_\alpha := \sum_{\ell=0}^\infty \sum_{m=-\ell}^\ell \sum_{\ell'=0}^\infty \sum_{m'=-\ell'}^{\ell'} D_\ell^{2\alpha} D_{\ell'}^{2\alpha} \langle g, \Y_{\ell,m} \otimes \Y_{\ell',m'} \rangle^2_\rangleLL
,$$
which satisfies
$$
\|g\|_\rangleLL = \|g\|_0 \le \|g\|_\alpha \le \|g\|_1 = \|g\|_\rangleHH.
$$
Similarly as in Proof of Theorem \ref{th::mu}, $g \mapsto \|g\|_\alpha$ specifies a norm on $\mathbb{H}_{p\alpha}$ which is equivalent to $\|\cdot\|_{\mathbb{H}_{p\alpha}}$.

An argument analogous to that used for proving Equation \eqref{eq::mubar-alphabound} shows that, for $\alpha \in [0,1]$,
\begin{equation}\label{eq::Rbar-alphabound}
\| R - \bar{R}_\pen \|^2_\alpha \le \pen^{1-\alpha} \|R\|^2_\rangleHH.
\end{equation}
Hence, setting $\alpha=0$,
\begin{equation*}
\|R - \bar{R}_\pen\|^2_\rangleLL \le \pen \|R\|^2_\rangleHH \le K_1 \pen,
\end{equation*}
which gives the claimed result under the assumptions on $\pen$.

We now prove \ref{proof2-R}, by first showing that
\begin{equation}\label{eq::tilde-bar-cov}
\|\Tilde{R}_\pen - \bar{R}_\pen \|^2_\alpha = \frac14 \sum_{\ell,m} \sum_{\ell',m'} \frac{D_\ell^{2\alpha} D_{\ell'}^{2\alpha}}{(1+ \pen D_\ell^2 D_{\ell'}^2)^2} ( F'_{ \pen}(\bar{R}_\pen)\Y_{\ell,m} \otimes \Y_{\ell',m'})^2.
\end{equation}
By the definitions of $\Tilde{R}_\pen$ and $\|\cdot\|_\alpha$,
\begin{align*}
\|\Tilde{R}_\pen - \bar{R}_\pen \|^2_\alpha &= \|(\bar{F}''_{  \pen})^{-1}F'_{ \pen}(\bar{R}_\pen) \|^2_\alpha\\
&= \sum_{\ell,m} \sum_{\ell',m'} D_\ell^{2\alpha} D_{\ell'}^{2\alpha} \langle (\bar{F}''_{  \pen})^{-1}F'_{ \pen}(\bar{R}_\pen), \Y_{\ell,m} \otimes\Y_{\ell',m'} \rangle_\rangleLL^2.
\end{align*}
However,
\begin{align*}
    \langle (\bar{F}''_{  \pen})^{-1}F'_{ \pen}(\bar{R}_\pen), \Y_{\ell,m}\otimes \Y_{\ell',m'}\rangle_\rangleLL &= \frac{1}{D_\ell^2 D_{\ell'}^2}  \langle (\bar{F}''_{  \pen})^{-1}F'_{ \pen}(\bar{R}_\pen), \Y_{\ell,m} \otimes \Y_{\ell',m'}\rangle_\rangleHH\\
     &=\frac{1}{D_\ell^2 D_{\ell'}^2} \langle \mathcal{Q}F'_{ \pen}(\bar{R}_\pen), (\Tilde{F}''_{  \pen})^{-1}\Y_{\ell,m} \otimes \Y_{\ell',m'}\rangle_\rangleHH,
\end{align*}
since $\mathcal{Q}F'_{  \pen}(\bar{R}_\pen)$ is the representer of $F'_{  \pen}(\bar{R}_\pen)$ and $(\Tilde{F}''_{  \pen})^{-1}$ is self-adjoint.
From Lemma \ref{lemma::2-cov},
\begin{align*}
    (\Tilde{F}''_{  \pen})^{-1}\Y_{\ell,m} \otimes \Y_{\ell',m'} = \frac{1}{2} \frac{D_\ell^2 D_{\ell'}^2}{1+\pen D_{\ell}^2D_{\ell'}^2} \Y_{\ell,m}\otimes \Y_{\ell',m'};
\end{align*}
thus,
\begin{align*}
    \langle (\bar{F}''_{  \pen})^{-1}F'_{ \pen}(\bar{R}_\pen), \Y_{\ell,m}\otimes \Y_{\ell',m'}\rangle_\rangleLL
     &=\frac{1}{2(1+\pen D_{\ell}^2D_{\ell'}^2)}  \langle \mathcal{Q}F'_{ \pen}(\bar{R}_\pen), \Y_{\ell,m}\otimes \Y_{\ell',m'}\rangle_\rangleHH \\&= \frac{1}{2(1+\pen D_{\ell}^2D_{\ell'}^2)} F'_{ \pen}(\bar{R}_\pen)\Y_{\ell,m}\otimes \Y_{\ell',m'}.
\end{align*}
Now observe that $\bar{F}'_{   \pen} (\bar{R}_\pen) = 0$ (see \cite[Theorem 3.6.3]{Hsing}). Then, an application of Lemma \ref{lemma::1-cov} reveals that, for any $g \in \Hprod$,
\begin{align*}
    &F'_{  \pen} (\bar{R}_\pen) g = F'_{  \pen} (\bar{R}_\pen)g - \bar{F}'_{   \pen} (\bar{R}_\pen) g\\
    =& -\frac{2(4\pi)^2}{n} \sum_{i=1}^n \frac{1}{r_i(r_i-1)} \sum_{1\le j \ne k \le r_i } (\W_{ij}\W_{ik} - \bar{R}_\pen(U_{ij}, U_{ik}))g(U_{ij},U_{ik}) + 2 \langle R - \bar{R}_\pen, g \rangle_\rangleLL.
\end{align*}
Consequently,
$$
\mathbb{E}[F'_{  \pen} (\bar{R}_\pen)\Y_{\ell,m} \otimes \Y_{\ell',m'}] = \mathbb{E}_\U \mathbb{E}[F'_{  \pen} (\bar{R}_\pen)\Y_{\ell,m} \otimes \Y_{\ell',m'}|\U] = 0
$$
and
\begin{align*}
    &\sum_{m,m'} \mathbb{E}|F'_{  \pen} (\bar{R}_\pen)\Y_{\ell,m} \otimes \Y_{\ell',m'}|^2\\ =& \sum_{m,m'} \operatorname{Var}[F'_{  \pen} (\bar{R}_\pen)\Y_{\ell,m}\otimes \Y_{\ell',m'}]\\
    =&  \frac{4(4\pi)^4}{n^2}  \sum_{i=1}^n \frac{1}{r^2_i(r_i-1)^2}  \sum_{m,m'} \operatorname{Var}\left[\sum_{1 \le j \ne k \le r_i} (\W_{1j}\W_{1k} - \bar{R}_\pen(U_{1j}, U_{1k}))\Y_{\ell,m}(U_{1j})\Y_{\ell',m'}(U_{1k}) \right] .
\end{align*}

Using the law of total variance, for a generic $i$ we can write 
\begin{align}
&\sum_{m,m'} \operatorname{Var}\left[\sum_{1 \le j \ne k \le r_i} (\W_{1j}\W_{1k} - \bar{R}_\pen(U_{1j}, U_{1k}))\Y_{\ell,m}(U_{1j})\Y_{\ell',m'}(U_{1k}) \right] \notag \\=& \sum_{m,m'} \operatorname{Var} \left [\sum_{1 \le j \ne k \le r_i} (R(U_{1j},U_{1k}) - \bar{R}_\pen(U_{1j}, U_{1k}))\Y_{\ell,m}(U_{1j})\Y_{\ell',m'}(U_{1k}) \right] \notag \\+& \sum_{m,m'} \mathbb{E}_\U\!\left[ \operatorname{Var}\left [\sum_{1 \le j \ne k \le r_i} \W_{1j}\W_{1k} \Y_{\ell,m}(U_{1j})\Y_{\ell',m'}(U_{1k})\Bigg| \U \right]\right]. \label{eq::ltv-1}
\end{align}

In what follows, we will handle sums over four indices $j,k,j',k'$.
It is then useful to identify the distinct cases which lead to terms of different orders. Recall that $j\ne k,\ j' \ne k'$, then we have
\begin{enumerate}
\item terms of order $r_i(r_i-1)$:
\begin{enumerate}
    \item $j=j', \ k=k'$
    \item $j=k', \ j'=k$
    \end{enumerate}
\item terms of order $r_i(r_i-1)(r_i-2)$:
\begin{enumerate}
    \item $j=j', \ k\ne k'$
    \item $j\ne j', \ k=k'$
    \item $j\ne j', \ k\ne k', \ j=k', \ j'\ne k$
    \item $j\ne j', \ k\ne k', \ j\ne k', \ j'= k$
\end{enumerate}
\item terms of order $r_i(r_i-1)(r_i-2)(r_i-3)$:
\begin{enumerate}
    \item $j\ne j', \ k\ne k', \ j\ne k', \ j'\ne k$
\end{enumerate}
\end{enumerate}

Now, for the first term on the right hand side of Equation \eqref{eq::ltv-1}, we obtain
\begin{align*}
&\sum_{m,m'} \operatorname{Var} \left [\sum_{1 \le j \ne k \le r_i} (R(U_{1j},U_{1k}) - \bar{R}_\pen(U_{1j}, U_{1k}))\Y_{\ell,m}(U_{1j})\Y_{\ell',m'}(U_{1k}) \right]\\ 
= & \sum_{m,m'} \mathbb{E}\left (\sum_{1 \le j \ne k \le r_i} (R(U_{1j},U_{1k}) - \bar{R}_\pen(U_{1j}, U_{1k}))\Y_{\ell,m}(U_{1j})\Y_{\ell',m'}(U_{1k}) \right)^2  \\
-& \frac{r_i^2(r_i-1)^2}{(4\pi)^4} \sum_{m,m'} \langle R- \bar{R}_\pen, \Y_{\ell,m} \otimes \Y_{\ell',m'} \rangle^2_\rangleLL\\
\le& B  \frac{ 2r_i(r_i-1) + 4r_i(r_i-1)(r_i-2)}{(4\pi)^2}(2\ell+1)(2\ell'+1)\|R - \bar{R}_\pen\|_\rangleHH^2  \\
+&  \frac{r_i(r_i-1)(r_i-2)(r_i-3) - r_i^2(r_i-1)^2}{(4\pi)^4} \sum_{m,m'} \langle R- \bar{R}_\pen, \Y_{\ell,m} \otimes \Y_{\ell',m'} \rangle^2_\rangleLL\\
\le& B K_1 \frac{2r_i(r_i-1) + 4r_i(r_i-1)(r_i-2)}{(4\pi)^2}(2\ell+1)(2\ell'+1),
\end{align*}
where the last two inequalities are justified by Lemma \ref{lemma::sup}, for some $B>0$, and Equation \eqref{eq::Rbar-alphabound} with $\alpha=1$. For the second term, write
\begin{align*}
\sum_{1 \le j \ne k \le r_i} \W_{1j}\W_{1k} \Y_{\ell,m}(U_{1j})\Y_{\ell',m'}(U_{1k}) &= \sum_{1 \le j \ne k \le r_i} X_1(U_{1j})X_1(U_{1k}) \Y_{\ell,m}(U_{1j})\Y_{\ell',m'}(U_{1k}) \\
&+ \sum_{1 \le j \ne k \le r_i} \epsilon_{1j} X_1(U_{1k}) \Y_{\ell,m}(U_{1j})\Y_{\ell',m'}(U_{1k}) \\
&+ \sum_{1 \le j \ne k \le r_i} X_1(U_{1j})\epsilon_{1k}\Y_{\ell,m}(U_{1j})\Y_{\ell',m'}(U_{1k})\\
&+ \sum_{1 \le j \ne k \le r_i} \epsilon_{1j} \epsilon_{1k} \Y_{\ell,m}(U_{1j})\Y_{\ell',m'}(U_{1k}).
\end{align*}
Denote the four terms in this last expression by $S_1,S_2,S_3$, and $S_4$ with indices corresponding to their location in the sum. Then, by an application of the Cauchy-Schwartz inequality,
$$
\mathbb{E}_\U\!\left[ \operatorname{Var}\left [S_1+S_2+S_3+S_4| \U \right]\right] \le 4  \left (\mathbb{E}|S_1|^2 +  \mathbb{E}|S_2|^2 + \mathbb{E}|S_3|^2 + \mathbb{E}|S_4|^2 \right).
$$
We will illustrate how to derive the bound for $S_1$, since the other three are somewhat simpler to handle and of order at most $r_i^3$.
Thus, following the scheme previously described, we obtain
\begin{align*}
    &\sum_{m,m'} \mathbb{E}\left( \sum_{1 \le j \ne k \le r_i} X_1(U_{1j})X_1(U_{1k}) \Y_{\ell,m}(U_{1j})\Y_{\ell',m'}(U_{1k}) \right)^2\\
    =& \sum_{m,m'} \sum_{j,k,j',k'} \mathbb{E}_\U[\Y_{\ell,m}(U_{1j})\Y_{\ell',m'}(U_{1k}) \Y_{\ell,m}(U_{1j'})\Y_{\ell',m'}(U_{1k'})  \mathbb{E}[X_1(U_{1j})X_1(U_{1k})X_1(U_{1j'})X_1(U_{1k'})|\U]\,]\\
    =& \frac{r_i(r_i-1)(r_i-2)(r_i-3)}{(4\pi)^4} \sum_{m,m'} \mathbb{E}\langle X, \Y_{\ell,m} \rangle_\rangleL^2 \langle X, \Y_{\ell',m'} \rangle_\rangleL^2\\
    +& \frac{r_i(r_i-1)}{(4\pi)^2}\sum_{m,m'}  \int_{\mathbb{S}^2} \int_{\mathbb{S}^2}  \mathbb{E} \left [ |X(u)|^2|X(v)|^2 \right]  |\Y_{\ell,m}(u)|^2  |\Y_{\ell',m'}(v)|^2 du dv \\
    +& \frac{r_i(r_i-1)}{(4\pi)^2} \sum_{m,m'} \int_{\mathbb{S}^2} \int_{\mathbb{S}^2} \mathbb{E} \left[ |X(u)|^2|X(v)|^2 \right]  \Y_{\ell,m}(u)\Y_{\ell',m'}(v) \Y_{\ell,m}(v) \Y_{\ell',m'}(u)  du dv \\
    +& \frac{r_i(r_i-1)(r_i-2)}{(4\pi)^3} \sum_{m,m'}  \int_{\mathbb{S}^2} \int_{\mathbb{S}^2} \int_{\mathbb{S}^2} \mathbb{E} \left[ X(u) X(v) |X(w)|^2\right] | \Y_{\ell,m} (w)|^2 \Y_{\ell',m'} (u) \Y_{\ell',m'} (v) dudvdw \\
    +& \frac{r_i(r_i-1)(r_i-2)}{(4\pi)^3} \sum_{m,m'}   \int_{\mathbb{S}^2} \int_{\mathbb{S}^2} \int_{\mathbb{S}^2} \mathbb{E} \left[ X(u) X(v) |X(w)|^2  \right] \Y_{\ell,m} (u) \Y_{\ell,m} (v)  | \Y_{\ell',m'} (w)|^2 dudvdw \\
    +&2 \frac{r_i(r_i-1)(r_i-2)}{(4\pi)^3} \sum_{m,m'}   \int_{\mathbb{S}^2} \int_{\mathbb{S}^2} \int_{\mathbb{S}^2}\mathbb{E} \left[ X(u) X(v) |X(w)|^2\right] \Y_{\ell,m} (w)  \Y_{\ell',m'} (u) \Y_{\ell,m} (v)  \Y_{\ell',m'} (w) dudvdw \\
   \le&\frac{r_i(r_i-1)(r_i-2)(r_i-3)}{(4\pi)^4} \sum_{m,m'} \mathbb{E}\langle X, \Y_{\ell,m} \rangle_\rangleL^2 \langle X, \Y_{\ell',m'} \rangle_\rangleL^2\\
   +&B' \frac{2r_i(r_i-1) + 4r_i(r_i-1)(r_i-2)}{(4\pi)^2}(2\ell+1)(2\ell'+1) \mathbb{E}\|X\|^4_\Hq,
\end{align*}
for some $B'>0$.
Indeed, for instance,
\begin{align*}
    &\left | \sum_{m,m'} \int_{\mathbb{S}^2} \int_{\mathbb{S}^2} \int_{\mathbb{S}^2}\mathbb{E} \left[ X(u) X(v) |X(w)|^2\right] \Y_{\ell,m} (w)  \Y_{\ell',m'} (u) \Y_{\ell,m} (v)  \Y_{\ell',m'} (w) dudvdw \right| \\
       \le& \sum_{m,m'} \int_{\mathbb{S}^2} \int_{\mathbb{S}^2} \int_{\mathbb{S}^2}\mathbb{E} \left[ |X(u)| |X(v)| |X(w)|^2\right] |\Y_{\ell,m} (w)| | \Y_{\ell',m'} (u)| |\Y_{\ell,m} (v) | |\Y_{\ell',m'} (w)| dudvdw  \\
       \le& B'\, \mathbb{E}\|X\|^2_\Hq  \sum_{m,m'}  \int_{\mathbb{S}^2}  | \Y_{\ell',m'} (u)| du \int_{\mathbb{S}^2} |\Y_{\ell,m} (v) | dv \int_{\mathbb{S}^2}  |\Y_{\ell,m} (w)| |\Y_{\ell',m'} (w)| dw\\ 
         \le&  B'(4\pi)(2\ell+1)(2\ell'+1) \mathbb{E}\|X\|^4_\Hq ,
\end{align*}
where again we have used Lemma \ref{lemma::sup}.
Hence, combining all the bounds,
\begin{align*}
\sum_{m,m'} \mathbb{E}|F'_{  \pen} (\bar{R}_\pen)\Y_{\ell,m} \otimes\Y_{\ell',m'}|^2 &\le \frac{4}{n} \sum_{m,m'}  \mathbb{E}\langle X, \Y_{\ell, m}  \rangle^2_\rangleL \langle X, \Y_{\ell', m'}  \rangle^2_\rangleL \\&+ (2\ell+1)(2\ell'+1) O\left( \frac{1}{nr} \right),
\end{align*}
and
\begin{align*}
\mathbb{E}\|\Tilde{R}_\pen - \bar{R}_\pen \|^2_\alpha &\le \frac{1}{n} \sum_{\ell,m} \sum_{\ell',m'} \frac{D_\ell^{2\alpha} D_{\ell'}^{2\alpha}}{(1+ \pen D_\ell^2 D_{\ell'}^2)^2} \mathbb{E}\langle X, \Y_{\ell, m}  \rangle^2_\rangleL \langle X, \Y_{\ell', m'}  \rangle^2_\rangleL \\&+ O\left( \frac{1}{nr} \right) \sum_{\ell,\ell'} \frac{D_\ell^{2\alpha} D_{\ell'}^{2\alpha}}{(1+ \pen D_\ell^2 D_{\ell'}^2)^2} (2\ell+1)(2\ell'+1),
\end{align*}
where again $r$ is the harmonic mean of $r_1,\dots,r_n$.
Now, for $\alpha \le q/p$,
\begin{align*}
   \sum_{\ell,m} \sum_{\ell',m'} \frac{D_\ell^{2\alpha} D_{\ell'}^{2\alpha}}{(1+ \pen D_\ell^2 D_{\ell'}^2)^2} \mathbb{E}\langle X, \Y_{\ell, m}  \rangle^2_\rangleL \langle X, \Y_{\ell', m'}  \rangle^2_\rangleL &\le \mathbb{E} \left (\sum_{\ell,m} D_\ell^{2\alpha}  \langle X, \Y_{\ell, m}  \rangle^2_\rangleL \right)^2\\
   &\le \mathbb{E}\|X\|^4_{p/q},
\end{align*}
which is bounded (possibly up to an arbitrary constant) by $\mathbb{E}\|X\|^4_{\Hq}$. Moreover, from \cite{Lin2000},
\begin{align*}
\sum_{\ell=0}^\infty \sum_{\ell'=0}^\infty \frac{D_\ell^{2\alpha} D_{\ell'}^{2\alpha}}{(1+ \pen D_\ell^2 D_{\ell'}^2)^2} (2\ell+1)(2\ell'+1)  =O\left( \pen^{-(\alpha+1/p)} \log(1/\pen) + 1 \right).
\end{align*}
Thus, we have that
$$ 
\mathbb{E}\|\Tilde{R}_\pen - \bar{R}_\pen \|^2_\alpha \le M_2 \left( (nr)^{-1}\pen^{-(\alpha+1/p)} \log(1/\pen)+  n^{-1}  \right),
$$
where $M_2$ is a positive constant not depending on the choice of $\mathbb{P}_X \in \pit$. By choosing $\alpha=0$, we obtain the claimed rate.

Now, let us prove \ref{proof3-R}.
The first step is to obtain a useful analytic form for $R_\pen - \Tilde{R}_\pen$, by observing that
\begin{align*}
R_\pen - \Tilde{R}_\pen &= R_\pen - \bar{R}_\pen +  (\bar{F}''_{  \pen})^{-1}F'_{ \pen}(\bar{R}_\pen) \\
&= (\bar{F}''_{  \pen})^{-1} \left [\bar{F}''_{  \pen}(R_\pen - \bar{R}_\pen) + F'_{ \pen}(\bar{R}_\pen)\right].
\end{align*}
Since $R_\pen$ minimizes $F'_{ \pen}$, it holds $F'_{ \pen}(R_\pen) = 0$ (see \cite[Theorem 3.6.3]{Hsing}).
Then, for any $g \in \Hprod$,
\begin{align*}
     \left [\bar{F}''_{  \pen}(R_\pen - \bar{R}_\pen) + F'_{ \pen}(\bar{R}_\pen)\right] g&=  \left [\bar{F}''_{  \pen}(R_\pen - \bar{R}_\pen) + F'_{ \pen}(\bar{R}_\pen) - F'_{ \pen}(R_\pen)\right]g\\
    &=\left[\bar{F}''_{   0}(R_\pen - \bar{R}_\pen) - F''_{  0}(R_\pen - \bar{R}_\pen)  \right]g.
\end{align*}
where we used Lemma \ref{lemma::1-cov}; in other words,
$$
R_\pen - \Tilde{R}_\pen  = (\bar{F}''_{  \pen})^{-1}  \left[\bar{F}''_{   0}(R_\pen - \bar{R}_\pen) - F''_{  0}(R_\pen - \bar{R}_\pen)  \right].
$$
Now, the same argument that leads to \eqref{eq::tilde-bar-cov} gives us
$$
\| R_\pen - \Tilde{R}_\pen \|^2_\alpha = \frac14 \sum_{\ell,m} \sum_{\ell',m'} \frac{D_\ell^{2\alpha} D_{\ell'}^{2\alpha}}{(1+ \pen D_\ell^2D_{\ell'}^2)^2} \left ( \left[\bar{F}''_{   0}(R_\pen - \bar{R}_\pen) - F''_{  0}(R_\pen - \bar{R}_\pen)  \right]\Y_{\ell,m} \otimes \Y_{\ell',m'}\right)^2,
$$
with
\begin{align*}
&\left[\bar{F}''_{   0}(R_\pen - \bar{R}_\pen) - F''_{  0}(R_\pen - \bar{R}_\pen)  \right]\Y_{\ell,m} \otimes \Y_{\ell',m'} \\=& 2\langle R_\pen - \bar{R}_\pen, \Y_{\ell,m} \otimes \Y_{\ell',m'}\rangle_\rangleLL 
\\-& \frac{2(4\pi)^2}{n} \sum_{i=1}^n \frac{1}{r_i(r_i-1)} \sum_{1\le j \ne k \le r_i} (R_\pen(U_{ij}, U_{ik}) - \bar{R}_\pen(U_{ij}, U_{ik})) \Y_{\ell,m}(U_{ij}) \Y_{\ell',m'}(U_{ik}).
\end{align*}
Now, $R_\pen - \bar{R}_\pen = \sum_{\ell,m} \sum_{\ell',m'} h_{\ell,\ell',m,m'} \Y_{\ell,m} \Y_{\ell',m'}$. Then,
\begin{align*}
    \| R_\pen - \Tilde{R}_\pen \|^2_\alpha =  \sum_{\ell_1,m_1}  \sum_{\ell_2,m_2} \frac{D_{\ell_1}^{2\alpha}D_{\ell_2}^{2\alpha} }{(1+ \pen D_{\ell_1}^2 D_{\ell_2}^2)^2} \left ( \sum_{\ell_3,m_3} \sum_{\ell_4, m_4} h_{\ell_3,\ell_4,m_3,m_4} V_{\ell_{1:4},m_{1:4}} \right)^2,
\end{align*}
where $$
V_{\ell_{1:4},m_{1:4}} = \delta_{\ell_1}^{\ell_3} \delta_{m_1}^{m_3} \delta_{\ell_2}^{\ell_4} \delta_{m_2}^{m_4} - \frac{(4\pi)^2}{n} \sum_{i=1}^n \frac{1}{r_i(r_i-1)} \sum_{1\le j \ne k \le r_i} \Y_{\ell_1,m_1} (U_{ij} ) \Y_{\ell_2,m_2} (U_{ik} ) \Y_{\ell_3,m_3} (U_{ij} ) \Y_{\ell_4,m_4} (U_{ik} ).
$$
By applying the Cauchy–Schwarz inequality for arbitrary $\theta \in (1/p, 1]$, we obtain
$$
\| R_\pen - \Tilde{R}_\pen \|^2_\alpha \le \| R_\pen - \bar{R}_\pen \|^2_\theta \sum_{\ell_1,m_1} \sum_{\ell_2,m_2} \frac{D_{\ell_1}^{2\alpha}D_{\ell_2}^{2\alpha} }{(1+ \pen D_{\ell_1}^2 D_{\ell_2}^2)^2} \sum_{\ell_3,m_3} \sum_{\ell_4,m_4} D_{\ell_3}^{-2\theta} D_{\ell_4}^{-2\theta} V^2_{\ell_{1:4},m_{1:4}}.
$$
It is readily seen that $\mathbb{E}[V_{\ell_{1:4},m_{1:4}}]=0$ and 
\begin{align*}
\mathbb{E}[V^2_{\ell_{1:4}, m_{1:4}}]  
&= \frac{(4\pi)^4}{n^2} \sum_{i=1}^n \frac{1}{r^2_i(r_i-1)^2}\mathbb{E}\left( \sum_{1\le j \ne k \le r_i} \Y_{\ell_1,m_1} (U_{ij} ) \Y_{\ell_2,m_2} (U_{ik} ) \Y_{\ell_3,m_3} (U_{ij} ) \Y_{\ell_4,m_4} (U_{ik} ) \right)^2 \\&- \delta_{\ell_1}^{\ell_3} \delta_{m_1}^{m_3} \delta_{\ell_2}^{\ell_4} \delta_{m_2}^{m_4}.
\end{align*}
Thus, it is possible to show that
$$
\sum_{m_1,m_2,m_4,m_4}\mathbb{E}[V^2_{\ell_{1:4}, m_{1:4}}] \le D_{\ell_1}^{2\theta} D_{\ell_2}^{2\theta} (2\ell_1+1)(2\ell_2 + 1)(2\ell_3 + 1)(2\ell_4 + 1) \, O\!\left ( \frac{1}{nr}\right),
$$
and hence
$$
\sum_{\ell_1,m_1} \sum_{\ell_2,m_2} \frac{D_{\ell_1}^{2\alpha}D_{\ell_2}^{2\alpha} }{(1+ \pen D_{\ell_1}^2 D_{\ell_2}^2)^2} \sum_{\ell_3,m_3} \sum_{\ell_4,m_4} D_{\ell_3}^{-2\theta} D_{\ell_4}^{-2\theta} \mathbb{E}[V^2_{\ell_{1:4},m_{1:4}}] = O\!\left (  \frac{\log(1/\pen)}{nr \pen^{\alpha+\theta+1/p}}\right) .
$$
 Let us define $$A_{n} := \sum_{\ell_1,m_1} \sum_{\ell_2,m_2} \frac{D_{\ell_1}^{2\alpha}D_{\ell_2}^{2\alpha} }{(1+ \pen D_{\ell_1}^2 D_{\ell_2}^2)^2} \sum_{\ell_3,m_3} \sum_{\ell_4,m_4} D_{\ell_3}^{-2\theta} D_{\ell_4}^{-2\theta} V^2_{\ell_{1:4},m_{1:4}},$$  $a_n:= \frac{\log(1/\pen)}{nr \pen^{\alpha+ \theta +1/p}}$ and $\gamma_n:= \frac{\log(1/\pen)}{nr \pen^{\alpha+1/p}}+ \frac1n$ .
Note that $A_n = \sup_{\mathbb{P}_X \in \pit} A_n$ and $A_n = o_\mathbb{P}(1)$.

Then,
$
\| R_\pen - \Tilde{R}_\pen \|^2_\alpha \le A_n \, \| R_\pen - \bar{R}_\pen \|^2_\theta
$
and therefore
\begin{align*}
\mathbb{P}\left( \| R_\pen - \Tilde{R}_\pen \|^2_\alpha > \varepsilon \, \gamma_n \right) &\le \mathbb{P}\left( A_n \, \| R_\pen - \bar{R}_\pen \|^2_\theta > \varepsilon \, \gamma_n \right)\\
&= \mathbb{P}\left( A_n \, \| R_\pen - \bar{R}_\pen \|^2_\theta > \varepsilon \, \gamma_n, A_n < 1 \right) + \mathbb{P}\left( A_n \, \| R_\pen - \bar{R}_\pen \|^2_\theta > \varepsilon \, \gamma_n , A_n \ge 1 \right) \\
&\le \mathbb{P}\left( A_n \, \| R_\pen - \bar{R}_\pen \|^2_\theta > \varepsilon \, \gamma_n, A_n < 1 \right) + \mathbb{P}\left( A_n \ge 1 \right).
\end{align*}
If $A_n < 1$,
\begin{align*}
\|  \Tilde{R}_\pen - \bar{R}_\pen \|_\theta &\ge \| R_\pen - \bar{R}_\pen \|_\theta - \| R_\pen - \Tilde{R}_\pen \|_\theta\\
& \ge (1-\sqrt{A_n})\| R_\pen - \bar{R}_\pen \|_\theta,
\end{align*}
which allows to write
\begin{align*}
\mathbb{P}\left( A_n \, \| R_\pen - \bar{R}_\pen \|^2_\theta > \varepsilon \, \gamma_n, A_n < 1 \right)
&\le \mathbb{P}\left( A_n |1-\sqrt{A_n}|^{-2} \|  \Tilde{R}_\pen - \bar{R}_\pen \|^2_\theta > \varepsilon \, \gamma_n, A_n < 1 \right) \\
&\le  \mathbb{P}\left( A_n |1-\sqrt{A_n}|^{-2} \|  \Tilde{R}_\pen - \bar{R}_\pen \|^2_\theta> \varepsilon \, \gamma_n \right).
\end{align*}
Let us now define $B_n:= \| \Tilde{R}_\pen - \bar{R}_\pen \|^2_\theta$ and $b_n:=\frac{\log(1/\pen)}{nr \pen^{\theta+1/p}} + \frac1n$. Recall that $\mathbb{E}[B_n] \le M_2 b_n$, for $\theta \le q/p$. Moreover, $|1-\sqrt{A_n}|^{-2}=O_\mathbb{P}(1)$.
We can observe that
$$
a_n  b_n = \frac{\log(1/\pen)}{nr \pen^{2\theta + 1/p}} \left ( \frac{\log(1/\pen)}{nr \pen^{\alpha + 1/p}} + \frac{\pen^{\theta - \alpha}}{n} \right),
$$
so that $c_n:= a_n  b_n/\gamma_n \to 0$, assuming $\theta < 1/2$ (recall that $p>2$) and $\alpha \in [0,\theta]$. Then,
\begin{align*}
\mathbb{P}\left( A_n |1-\sqrt{A_n}|^{-2}   B_n > \varepsilon \, \gamma_n \right) &= \mathbb{P}\left( \frac{A_n |1-\sqrt{A_n}|^{-2}  B_n}{a_n b_n} > \frac{\varepsilon}{c_n} \right)\\
&\le \mathbb{P}\left( \frac{B_n}{b_n} > \frac{\varepsilon^{1/3}}{c_n^{1/3}} \right)+ \mathbb{P}\left( \frac{A_n}{a_n} > \frac{\varepsilon^{1/3}}{c_n^{1/3}} \right) +  \mathbb{P}\left( |1-\sqrt{A_n}|^{-2}   > \frac{\varepsilon^{1/3}}{c_n^{1/3}} \right) \\
&\le M_2 \frac{c_n^{1/3}}{\varepsilon^{1/3}} + M_3 \frac{c_n^{1/3}}{\varepsilon^{1/3}} + \mathbb{P}\left( |1-\sqrt{A_n}|^{-2}   > \frac{\varepsilon^{1/3}}{c_n^{1/3}} \right).
\end{align*}
Clearly $c_n^{1/3}|1-\sqrt{A_n}|^{-2}=o_\mathbb{P}(1)$, hence
$$
\lim_{n\to\infty} \sup_{\mathbb{P}_X \in \pit} \mathbb{P}\left( \| R_\pen - \Tilde{R}_\pen \|^2_\alpha > \varepsilon \, \gamma_n \right) = 0.
$$
By taking $\alpha=0$ we obtain the claimed result.

Recall now that, when the mean is $\mu \ne 0$, an estimate of the complete covariance kernel $C(u,v)=R(u,v) - \mu(u)\mu(v)$ is given by
$$
C_\pen(u,v) = R_\pen(u,v)  - \mu_\pen(u) \mu_\pen(v).
$$
Moreover, observe that
$$
\|C_\pen - C\|_\rangleLL \le \|R_\pen - R\|_\rangleLL + \|\mu_\pen \otimes \mu_\pen - \mu \otimes \mu\|_\rangleLL,
$$
by the triangle inequality; hence,
\begin{align*}
\mathbb{P} \left ( \|C_\pen - C\|^2_\rangleLL > t \right) \le \mathbb{P} \left ( \|R_\pen - R\|^2_\rangleLL > t \right) + \mathbb{P} \left ( \|\mu_\pen \otimes \mu_\pen - \mu \otimes \mu\|^2_\rangleLL > t \right).
\end{align*}
For the second term on the right-hand side, we have
\begin{align*}
 \|\mu_\pen \otimes \mu_\pen - \mu \otimes \mu\|_\rangleLL &=  \|\mu_\pen \otimes \mu_\pen  \pm \mu_\pen \otimes \mu  - \mu \otimes \mu\|_\rangleLL\\
 &\le \|  \mu_\pen \|_\rangleL \| \mu_\pen  - \mu \|_\rangleL + \| \mu_\pen  - \mu\|_\rangleL \|  \mu\|_\rangleL\\
  &\le \| \mu_\pen - \mu\|^2_\rangleL + 2\| \mu_\pen  - \mu\|_\rangleL \|  \mu\|_\rangleL\\
  &= O_\mathbb{P}\left ( \| \mu_\pen  - \mu\|_\rangleL \right).
\end{align*}
If we repeat the same steps in Proof of Theorem \ref{th::mu}, but this time with $\pen \asymp (nr/\log n)^{-p/(p+1)}$, we see that the (uniform) rate for $\mu_\pen$ is exactly $(nr/\log n)^{-p/(p+1)} + n^{-1}$, which concludes the proof.
\end{proof}

The next two lemmas are referred to
\begin{equation*}
    F_{ \pen} (g) := \frac{(4\pi)^2}{n} \sum_{i=1}^n \frac{1}{r_i(r_i-1)} \sum_{1\le j \ne k \le r_i } (\W_{ij}\W_{ik}  - g(U_{ij}, U_{ik}))^2 + \pen \|(\mathscr{D} \otimes \mathscr{D})  g\|^2_\rangleLL,
\end{equation*}
$$\bar{F}_{  \pen}(g) := \mathbb{E}[F_{ \pen} (g)] = (4\pi)^2 \operatorname{Var}[\W_{11}\W_{12}] + \|R - g\|^2_\rangleLL + \pen \|(\mathscr{D} \otimes \mathscr{D})  g \|^2_\rangleLL.$$

Recall that $\mathcal{B}(\mathbb{X}_1, \mathbb{X}_2)$ denote the set of all linear and bounded operators
from $\mathbb{X}_1$ to $\mathbb{X}_2$, being normed spaces.
Here, we consider $\Hprod$ endowed with $\|(\mathscr{D} \otimes \mathscr{D}) \cdot \|_\rangleLL$.
\begin{lemma}\label{lemma::1-cov}
Let $f,g,g_1,g_2$ be arbitrary elements of $\Hprod$. 
\begin{enumerate}
    \item 
The Fréchet derivative of $F_{ \pen}$ at $f$ is the element $F'_{ \pen}(f)$ of $\mathcal{B}(\Hprod,\mathbb{R})$ characterized by
\begin{align*}
    F'_{ \pen}(f)g &= -\frac{2(4\pi)^2}{n} \sum_{i=1}^n \frac{1}{r_i(r_i-1)} \sum_{1\le j \ne k \le r_i } (\W_{ij}\W_{ik} - f(U_{ij}, U_{ik}))g(U_{ij}, U_{ik}) \\&+ 2 \pen \langle (\mathscr{D} \otimes \mathscr{D}) f, (\mathscr{D} \otimes \mathscr{D}) g \rangle_\rangleLL
    \end{align*}
The second Fréchet derivative $F''_{ \pen}\in \mathcal{B}(\Hprod, \mathcal{B}(\Hprod,\mathbb{R}))$ is
characterized by
\begin{align*}
        F''_{ \pen}g_1g_2 &= \frac{2(4\pi)^2}{n} \sum_{i=1}^n \frac{1}{r_i(r_i-1)} \sum_{1\le j \ne k \le r_i } g_1(U_{ij},U_{ik}) g_2(U_{ij}, U_{ik}) \\&+ 2 \pen \langle (\mathscr{D} \otimes \mathscr{D})  g_1, (\mathscr{D} \otimes \mathscr{D}) g_2 \rangle_\rangleLL.
\end{align*}

    \item 
The Fréchet derivative of $\bar{F}_{  \pen}$ at $f$ is the element $\bar{F}'_{  \pen}(f)$ of $\mathcal{B}(\Hprod,\mathbb{R})$ characterized by
\begin{equation*}
    \bar{F}'_{  \pen}(f)g = - 2 \langle R-f,g \rangle_\rangleLL + 2 \pen \langle (\mathscr{D} \otimes \mathscr{D})  f, (\mathscr{D} \otimes \mathscr{D}) g \rangle_\rangleLL
\end{equation*}
The second Fréchet derivative $\bar{F}''_{  \pen}\in \mathcal{B}(\Hprod, \mathcal{B}(\Hprod,\mathbb{R}))$ is
characterized by
\begin{equation}\label{eq::2dv-cov}
        \bar{F}''_{  \pen}g_1g_2 = 2 \langle g_1,g_2 \rangle_\rangleLL + 2 \pen \langle (\mathscr{D} \otimes \mathscr{D}) g_1, (\mathscr{D} \otimes \mathscr{D}) g_2 \rangle_\rangleLL.
\end{equation}
\end{enumerate}
\end{lemma}

\begin{proof}
The proof is a direct application of Theorem 3.6.4 in \cite{Hsing}. See also Lemma 8.3.3.
\end{proof}

The evaluation of $\Tilde{R}_\pen$ involves the inverse of the operator $\bar{F}''_{  \pen}$. To this purpose, it is convenient to invoke the Riesz representation theorem (see \cite[Theorem 3.2.1]{Hsing}), which tells us that there is an invertible norm-preserving mapping $\mathscr{Q}$ such that $\mathscr{Q}\mathcal{B}(\Hprod, \mathbb{R})= \Hprod$. Thus,
\begin{equation}\label{eq::ftilde-cov}
\Tilde{F}''_{   \pen} := \mathscr{Q}\bar{F}''_{  \pen}
\end{equation}
is an element of $\mathcal{B}(\Hprod, \Hprod)$ and it is invertible if and only if $\bar{F}''_{  \pen}$ is invertible. 

\begin{lemma}\label{lemma::2-cov}
The operator $\Tilde{F}''_{   \pen}$ in \eqref{eq::ftilde-cov} is an invertible element of $\mathcal{B}(\Hprod, \Hprod)$ and, for any $g \in \Hprod$,
$$
(\Tilde{F}''_{   \pen})^{-1}g  = \frac12 \sum_{\ell=0}^\infty \sum_{\ell'=0}^\infty \sum_{m=-\ell}^\ell \sum_{m'=-\ell'}^{\ell'} \frac{D_\ell^2 D_{\ell'}^2}{1+\pen D_\ell^2 D_{\ell'}^2 }  \langle g, \Y_{\ell,m} \otimes \Y_{\ell',m'} \rangle_\rangleLL \Y_{\ell,m} \Y_{\ell'm'}.
$$
\end{lemma}
Notice that $(\Tilde{F}''_{  \pen})^{-1}$ is a self-adjoint operator.

\begin{proof}
Take $g_1 \in \Hprod$, so that $\bar{F}''_{  \pen}g_1$ belongs to $\mathcal{B}(\Hprod,\mathbb{R})$, with representer $\Tilde{F}''_{  \pen}g_1 \in \Hprod$. Then, for any $g_2 \in \Hprod$,
\begin{align*}
\bar{F}''_{  \pen}g_1g_2 &= \langle (\mathscr{D} \otimes \mathscr{D})  \Tilde{F}''_{  \pen}g_1, (\mathscr{D} \otimes \mathscr{D})  g_2 \rangle_\rangleLL \\&= 2 \langle g_1,g_2 \rangle_\rangleLL + 2 \pen \langle (\mathscr{D} \otimes \mathscr{D})  g_1, (\mathscr{D} \otimes \mathscr{D}) g_2 \rangle_\rangleLL,
\end{align*}
where the last equality comes from \eqref{eq::2dv-cov}. We can hence write the expansion in $\Hprod$
$$
\Tilde{F}''_{  \pen}g_1 = 2 \sum_{\ell=0}^\infty \sum_{\ell'=0}^\infty \sum_{m=-\ell}^\ell \sum_{m'=-\ell'}^{\ell'} \frac{1+\pen D_\ell^2 D_{\ell'}^2 }{D_\ell^2 D_{\ell'}^2}  \langle g_1, \Y_{\ell,m} \otimes \Y_{\ell',m'} \rangle_\rangleLL \Y_{\ell,m} \Y_{\ell'm'},
$$
which suggests that $\Tilde{F}''_{  \pen}$ is invertible and 
$$
(\Tilde{F}''_{   \pen})^{-1}g_1 = \frac12 \sum_{\ell=0}^\infty \sum_{\ell'=0}^\infty \sum_{m=-\ell}^\ell \sum_{m'=-\ell'}^{\ell'} \frac{D_\ell^2 D_{\ell'}^2}{1+\pen D_\ell^2 D_{\ell'}^2 }  \langle g_1, \Y_{\ell,m} \otimes \Y_{\ell',m'} \rangle_\rangleLL \Y_{\ell,m} \Y_{\ell'm'}.
$$
\end{proof}

We now obtain the asymptotic performance for the two classes of spherical random fields introduced in Section~\ref{sec:examples}.

\begin{proof}[Proof of Proposition \ref{prop:filternoise}]
We fix $q \leq p$ such that $2 < p < \beta - 1/2$ and $1 < q < \beta - 1$. Note that $p$ and $q$ exist due to the asumption that $\beta > 5/2$. The proof is divided in two parts. \\

\textit{1. $\mathcal{C}_\beta \subset \pit$.}
Let $\X \in \mathcal{C}_\beta$. 
The Gaussian white noise over the $d$-dimension hypersphere is almost surely in the Sobolev space with negative smoothness 
$W \in \mathcal{H}_{-d/2-\epsilon}$ for any $\epsilon > 0$. This result is well-known for Gaussian white noises over $\Sph^1$ or more generally the $d$-dimensional torus $\mathbb{T}^d = \Sph^1 \times \cdots \times \Sph^1$; see for instance~\cite[Theorem 3.4]{Veraar2010regularity} or~\cite[Theorem 5]{fageot2017besov}. 
More generally, it is true for the Gaussian white noise over a compact Riemannian manifold with no boundary, as is demonstrated for instance  in the proof of Proposition 3.8 in~\cite{dello2020discovery}. 
In our setting, we deduce that $W \in \mathcal{H}_{-1-\epsilon}$ almost surely for any $\epsilon > 0$. 
Moreover, any admissible operator $\mathscr{D}$ with spectral growth order $\beta$ is such that  
$\mathscr{D}^{-1} \mathcal{H}_{-1-\epsilon} = \mathcal{H}_{\beta -1-\epsilon}$, hence $X \in  \mathcal{H}_{\beta -1-\epsilon}$ almost surely. 
For $\epsilon >0$ small enough, we have that $\beta - 1 - \epsilon \geq q$, hence $\mathcal{H}_{\beta -1-\epsilon} \subset \mathcal{H}_q$. In particular, 
$X  \in \mathcal{H}_{q}$ almost surely.

Our goal is now to prove that the covariance $C$ of $X$ is in $\mathbb{H}_p$. 
We first prove that $C(u,v) = K(u,v)$, for any $u,v \in \mathbb{S}^2$, where $K:  \mathbb{S}^2 \times  \mathbb{S}^2 \to \mathbb{R}$ is the reproducing kernel of $(\mathcal{H}_\beta, \|\mathscr{D} \cdot \|_{L^2(\mathbb{S}^2)})$.   
Using that $X = \mathscr{D}^{-1} W$ and $X(u) = \langle X , \delta_{u} \rangle_{\rangleL}$, we have that, for $u,v \in \mathbb{S}^2$,
\begin{align*}
\mathbb{E}[X(u) X(v)] &= \mathbb{E}[\langle \mathscr{D}^{-1}W, \delta_u \rangle_{L^2(\mathbb{S}^2)} \langle \mathscr{D}^{-1}W, \delta_v \rangle_{L^2(\mathbb{S}^2)}]\\
&=
\mathbb{E}[\langle W,  (\mathscr{D}^{-1})^* \delta_u \rangle_{L^2(\mathbb{S}^2)} \langle W, (\mathscr{D}^{-1})^{*} \delta_v \rangle_{L^2(\mathbb{S}^2)}]\\
&= \langle (\mathscr{D}^{-1})^* \delta_u, (\mathscr{D}^{-1})^* \delta_v\rangle_{L^2(\mathbb{S}^2)}\\
&= \langle \mathscr{D} \mathscr{D}^{-1} (\mathscr{D}^{-1})^* \delta_u, \mathscr{D} \mathscr{D}^{-1}(\mathscr{D}^{-1})^* \delta_v\rangle_{L^2(\mathbb{S}^2)}\\
&= \langle (\mathscr{D}^* \mathscr{D})^{-1} \delta_u, (\mathscr{D}^* \mathscr{D})^{-1}  \delta_v\rangle_{\mathcal{H}_\beta}.
\end{align*}
We then observe that $(\mathscr{D}^* \mathscr{D})^{-1} \delta_u = K(\cdot,u)$, $u \in \mathbb{S}^2$ since it satisfy the reproducing property
$$
 \langle (\mathscr{D}^* \mathscr{D})^{-1} \delta_u, f \rangle_{\mathcal{H}_\beta} =  \langle \mathscr{D} (\mathscr{D}^* \mathscr{D})^{-1} \delta_u, \mathscr{D} f \rangle_{L^2(\mathbb{S}^2)} =   \langle \delta_u, f \rangle_{L^2(\mathbb{S}^2)} =f(u),
$$
for every $f \in \mathcal{H}_\beta$. Thus we have that
\begin{align*}
C(u,v) = \mathbb{E}[X(u) X(v)] &= \langle K(\cdot,u),K(\cdot,v) \rangle_{\mathcal{H}_\beta} = K(u,v).
\end{align*}
We therefore deduce that, since the Fourier coefficients of $K$ are given by $K_{\ell,m} = 1 / |D_\ell|^2$, we have that 
\begin{equation*}
\|C\|^2_{\mathbb{H}_p} = \|K\|^2_{\mathbb{H}_p} = \sum_{\ell=0}^\infty \sum_{m= -\ell}^{\ell} (1+\ell(\ell+1))^{2p} |K_{\ell,m}|^2  = 
\sum_{\ell=0}^\infty (2\ell+1)(1+\ell(\ell+1))^{2p} D_\ell^{-4}, 
\end{equation*}
Finally, using that the $D_\ell$ satisfies \eqref{eq:conditionDn} (with $p = \beta$), the previous sum is finite if and only if $4 \beta - 4p - 1 > 1$, which is true due to the assumption that $p < \beta - 1/2$. Finally, we have shown that $X \in \pit$. \\

\textit{2. Existence of $C_\eta$.}
Let $\mathscr{D}$ be an admissible operator of order $p$ such that $2<p<\beta - 1/2$. Then, we have seen that, for such $p$ and for $q \leq p$, $1 < q < \beta-1$, $X \in \pit$. Then, the estimator $C_\eta = R_\eta$ associated to the operator $\mathscr{D}$ verifies the assumptions of Theorem~\ref{th::R} and the bound of Corollary \ref{coro:asymptoticCetabigO} is achieved. This is true for $p$ arbitrarily close to $\beta - 1/2$, in particular for $p$ such that $\frac{p}{p-1} - \frac{\beta - 1/2}{\beta+1/2} = \epsilon > 0$, which gives \eqref{eq:firstclass}
and concludes the proof. 
\end{proof}

\begin{proof}[Proof of Proposition \ref{prop:filterdiracstream}]
Let $p < \beta - 1$ Then, for any $u\in \Sph^2$, the Green's function $\psi_u^{\mathscr{D}}$ of $\mathscr{D}$ is in $\mathcal{H}_{\beta - 1 - \epsilon}$. This follows from that, $\delta_u \in \mathcal{H}_{-1-\epsilon}$ and the fact that, as we have seen in the proof of Proposition~\ref{prop:filternoise}, $\psi_u^{\mathscr{D}} = \mathscr{D}^{-1} \delta_u \in \mathcal{H}_{\beta - 1 - \epsilon} \subset \mathcal{H}_p$ for $\epsilon  >0$ small enough.
In particular, $X \in \Hp$ as a (random) linear combination of Green's functions of $\mathscr{D}$. We moreover easily see that the covariance  $C$ of $X$ is a linear combinations of tensorial functions $\psi_u^{\mathscr{D}}\otimes \psi_v^{\mathscr{D}}$ for $u,v \in \Sph^2$ and is therefore in $\Hprod$. This shows that $X \in \pi(p,p) \subset \pit$ for any $1 \leq q\leq p$ and concludes the first part of Proposition~\ref{prop:filterdiracstream}. The existence of $C_\eta$ is then obtained in same way as for the proof of Proposition~\ref{prop:filternoise}, except that $p < \beta - 1$ and therefore \eqref{eq:secondclass} follows with $\epsilon = \frac{p}{p+1} - \frac{\beta-1}{\beta}$. 
\end{proof}




\subsection{Proofs of Section~\ref{sec:temporal}}\label{sec:proofs4}

\begin{proof}[Proof of Theorem \ref{th::mu-time}]
For simplicity, we consider the case $r_1=r_2=\cdots=r_n$. 

This proof follows the same lines of Proof of Theorem \ref{th::mu} in Section \ref{sec:proofs3}. The only different part is the one referred to Point \ref{proof2}. Starting from Equation \eqref{eq::Fprime}, we have that
$$
\mathbb{E}[F'_{  \pen} (\bar{\mu}_\pen)\Y_{\ell,m}] = \mathbb{E}_\U \mathbb{E}[F'_{  \pen} (\bar{\mu}_\pen)\Y_{\ell,m}|\U] = 0
$$
and
\begin{align*}
    \sum_{m=-\ell}^\ell \mathbb{E}|F'_{  \pen} (\bar{\mu}_\pen)\Y_{\ell,m}|^2 &= \sum_{m=-\ell}^\ell \operatorname{Var}[F'_{  \pen} (\bar{\mu}_\pen)\Y_{\ell,m}]\\
    &=  \frac{(8\pi)^2}{(nr)^2} \sum_{m=-\ell}^\ell \operatorname{Var}\left[ \sum_{t=1}^n \sum_{j=1}^r (\W_{tj} - \bar{\mu}_\pen(U_{tj}))\Y_{\ell,m}(U_{tj})  \right] .
\end{align*}
Using the law of total variance, we can write 
\begin{align*}
\sum_{m=-\ell}^\ell \operatorname{Var}\left[ \sum_{t=1}^n\sum_{j=1}^r (\W_{tj} - \bar{\mu}_\pen(U_{tj}))\Y_{\ell,m}(U_{tj})  \right] &= \sum_{m=-\ell}^\ell \operatorname{Var} \left [\sum_{t=1}^n \sum_{j=1}^r (\mu(U_{tj}) - \bar{\mu}_\pen(U_{tj}))\Y_{\ell,m}(U_{tj}) \right] \\&+ \sum_{m=-\ell}^\ell \mathbb{E}_\U\!\left[ \operatorname{Var}\left [\sum_{t=1}^n \sum_{j=1}^r (\W_{tj} - \bar{\mu}_\pen(U_{tj}))\Y_{\ell,m}(U_{tj}) \Bigg| \U \right]\right].
\end{align*}
For the first term on the right hand side of this expression, we have
\begin{align*}
 \sum_{m=-\ell}^\ell \operatorname{Var}\left [\sum_{t=1}^n \sum_{j=1}^r (\mu(U_{tj}) - \bar{\mu}_\pen(U_{tj}))\Y_{\ell,m}(U_{tj}) \right] &= nr \sum_{m=-\ell}^\ell \operatorname{Var}\left [ (\mu(U_{11}) - \bar{\mu}_\pen(U_{11}))\Y_{\ell,m}(U_{11}) \right]  \\
 &\le \frac{nr}{4\pi} \sum_{m=-\ell}^\ell  \int_{\mathbb{S}^2} |\mu(u) - \bar{\mu}_\pen(u) |^2 |\Y_{\ell,m}(u)|^2 du \\
 &\le \frac{B nr}{4\pi} (2\ell+1) \|\mu - \bar{\mu}_\pen \|^2_\rangleH \\
 &\le \frac{B K nr}{4\pi}  (2\ell+1),
\end{align*}
where the last two inequalities are justified by Lemma \ref{lemma::sup}, for some $B>0$, and Equation \eqref{eq::mubar-alphabound} with $\alpha=1$. Then, for the second term,
\begin{align*}
    &\sum_{m=-\ell}^\ell \mathbb{E}_\U\!\left[ \operatorname{Var}\left [\sum_{t=1}^n \sum_{j=1}^r (\W_{tj} - \bar{\mu}_\pen(U_{tj}))\Y_{\ell,m}(U_{tj}) \Bigg| \U \right]\right] \\
   =& \sum_{m=-\ell}^\ell \sum_{t=1}^n \sum_{t'=1}^n\sum_{j=1}^r \sum_{j'=1}^r \mathbb{E}_\U\!\left[ \Y_{\ell,m}(U_{tj}) \Y_{\ell,m}(U_{t'j'}) \operatorname{Cov}[\W_{tj},\W_{t'j'}| \U] \right]\\
    =&\,  \frac{r(r-1)}{(4\pi)^2} \sum_{m=-\ell}^\ell \sum_{t=1}^n \operatorname{Var}[\langle X_t, \Y_{\ell,m} \rangle_\rangleL ] \\
    +&\, \frac{r^2}{(4\pi)^2}  \sum_{m=-\ell}^\ell \sum_{t \ne t'}  \operatorname{Cov}[\langle X_t, \Y_{\ell,m} \rangle_\rangleL, \langle X_{t'}, \Y_{\ell,m} \rangle_\rangleL] \\
    +& \, \frac{nr}{4\pi} \sum_{m=-\ell}^\ell \int_{\mathbb{S}^2}  C_0(u,u)  |\Y_{\ell,m}(u)|^2 du +  \frac{nr}{4\pi} (2\ell+1) \sigma^2   \\
    \le&\, \frac{r^2}{(4\pi)^2}  \sum_{m=-\ell}^\ell \sum_{t =1}^n \sum_{t'=1}^n  \operatorname{Cov}[\langle X_t, \Y_{\ell,m} \rangle_\rangleL, \langle X_{t'}, \Y_{\ell,m} \rangle_\rangleL]  \\
    +& \, \frac{B' nr}{4\pi} (2\ell+1) \mathbb{E}\|X_0\|^2_{\Hq}  +  \frac{nr}{4\pi} (2\ell+1) \sigma^2 ,
\end{align*}
again by applying Lemma \ref{lemma::sup}, for some $B'>0$.
Now, observe that, by stationarity,
\begin{align*}
&\frac{1}{n}\sum_{t =1}^n \sum_{t'=1}^n | \operatorname{Cov}[\langle X_t, \Y_{\ell,m} \rangle_\rangleL, \langle X_{t'}, \Y_{\ell,m} \rangle_\rangleL]| \\=& \frac{1}{n} \sum_{t =1}^n \sum_{t'=1}^n | \langle \CL_{t-t'}\Y_{\ell,m},\Y_{\ell,m}\rangle_\rangleL |   \\=& \sum_{|\xi|< n} \left (1- \frac{|\xi|}{n} \right)|\langle \CL_\xi\Y_{\ell,m},\Y_{\ell,m}\rangle_\rangleL|\\
\le& \sum_{\xi \in \mathbb{Z}} |\langle \CL_\xi\Y_{\ell,m},\Y_{\ell,m}\rangle_\rangleL|.
\end{align*}
Hence, combining all the bounds,
$$
\sum_{m=-\ell}^\ell \mathbb{E}|F'_{  \pen} (\bar{\mu}_\pen)\Y_{\ell,m}|^2 \le \frac{4}{n} \sum_{m=-\ell}^\ell \sum_{\xi \in \mathbb{Z}} |\langle \CL_\xi\Y_{\ell,m},\Y_{\ell,m}\rangle_\rangleL| + (2\ell+1) O\left( \frac{1}{nr} \right),
$$
and
\begin{align*}
\mathbb{E}\|\Tilde{\mu}_\pen - \bar{\mu}_\pen \|^2_\alpha &\le \frac{1}{n} \sum_{\ell=0}^\infty \sum_{m=-\ell}^\ell \frac{D_\ell^{2\alpha}}{(1+ \pen D_\ell^2)^2} \sum_{\xi \in \mathbb{Z}} |\langle \CL_\xi\Y_{\ell,m},\Y_{\ell,m}\rangle_\rangleL| \\&+ O\left( \frac{1}{nr} \right) \sum_{\ell=0}^\infty \frac{D_\ell^{2\alpha}}{(1+ \pen D_\ell^2)^2} (2\ell+1)\\
&\le  \frac{1}{n} \sum_{\ell=0}^\infty \sum_{m=-\ell}^\ell D_\ell^{2\alpha} \sum_{\xi \in \mathbb{Z}} |\langle \CL_\xi\Y_{\ell,m},\Y_{\ell,m}\rangle_\rangleL|\\ &+ O\left( \frac{1}{nr} \right) \sum_{\ell=0}^\infty \frac{D_\ell^{2\alpha}}{(1+ \pen D_\ell^2)^2} (2\ell+1).
\end{align*}
Then, for $\alpha \le q/p$ and using the property in Equation \eqref{eq::covops}, we obtain
\begin{align*}
 \sum_{\xi \in \mathbb{Z}} \sum_{\ell=0}^\infty D_\ell^{2\alpha}   \sum_{m=-\ell}^\ell |\langle \CL_\xi\Y_{\ell,m},\Y_{\ell,m}\rangle_\rangleL|&=  \sum_{\xi \in \mathbb{Z}} \sum_{\ell=0}^\infty \frac{1}{D_\ell^{2\alpha}} \sum_{m=-\ell}^\ell  |\langle \CHq_\xi \Y_{\ell,m}, \Y_{\ell,m} \rangle_\Hq| \\&\le \sum_{\xi \in \mathbb{Z}} \|\CHq_\xi\|_{\operatorname{TR}, \Hq}
\end{align*}
(possibly up to an arbitrary constant). The rest of the proof follows exactly as in Proof of Theorem \ref{th::mu} in Section \ref{sec:proofs3}.
\end{proof}

\begin{proof}[Proof of Theorem \ref{th::R-time}]

For simplicity, we consider the case $r_1=r_2=\cdots=r_n$ and we define
$$
\num = \begin{cases}
nr(r-1) & \lag=0\\
(n-\lag)r^2 &\lag\ne0
\end{cases}.
$$
Recall that the estimator for $R_\lag$ at fixed lag $\lag \in \{0, 1,\dots,n-1\}$ is given by
\begin{equation*}
    R_{\lag;\pen} := \argmin_{g \in \Hprod}  F_{\lag; \pen} (g),
\end{equation*}
\begin{equation*}
    F_{\lag; \pen} (g) := \frac{(4\pi)^2}{N_\lag} \sum_{t=1}^{n-\lag} \sum_{j=1}^r \sum_{k=1}^r (\W_{t+\lag,j}\W_{tk}  - g(U_{t+\lag,j}, U_{tk}))^2 + \pen \|g\|^2_\rangleHH.
\end{equation*}

In keeping with the notation that was used in the previous proof, we first define
$$\bar{F}_{\lag;\pen}(g) := \mathbb{E}[F_{\lag; \pen} (g)] = (4\pi)^2 \operatorname{Var}[\W_{\lag+1, 1}\W_{1 2}] + \|R_{\lag} - g\|^2_\rangleLL + \pen \|g \|^2_\rangleHH$$ and we let
\begin{equation*}
    \bar{R}_{\lag;\pen} := \argmin_{g \in \Hprod}  \bar{F}_{\lag;\pen} (g).
\end{equation*}
Also, define $\Tilde{R}_{\lag;\pen}= \bar{R}_{\lag;\pen} - (\bar{F}''_{\lag;\pen})^{-1} F'_{\lag; \pen}(\bar{R}_{\lag;\pen}).$
Then, write $$
R_{\lag;\pen}  - R_{\lag}  = R_{\lag;\pen} - \Tilde{R}_{\lag;\pen} + \Tilde{R}_{\lag;\pen} - \bar{R}_{\lag;\pen} +  \bar{R}_{\lag;\pen} - R_{\lag}.
$$


The first part of the proof follows the same lines of Proof of Theorem \ref{th::R} in Section \ref{sec:proofs3}. We first define the intermediate norm $\|\cdot\|_\alpha, \ \alpha \in [0,1]$. We then show that 
\begin{equation}\label{eq::Rbar-alphabound-time}
\| R_\lag - \bar{R}_{\lag;\pen} \|^2_\alpha \le \pen^{1-\alpha} \|R_\lag\|^2_\rangleHH,
\end{equation}
and hence that
\begin{equation*}
\|R_\lag - \bar{R}_{\lag\pen}\|^2_\rangleLL \le \pen \|R_\lag\|^2_\rangleHH \le K_1 \pen,
\end{equation*}
which proves the analogous of Point \ref{proof1-R}.

Afterwards, we show that 
\begin{equation}\label{eq::tilde-bar-cov-time2}
\|\Tilde{R}_{\lag;\pen} - \bar{R}_{\lag;\pen} \|^2_\alpha = \frac14 \sum_{\ell,m} \sum_{\ell',m'} \frac{D_\ell^{2\alpha} D_{\ell'}^{2\alpha}}{(1+ \pen D_\ell^2 D_{\ell'}^2)^2} ( F'_{ \lag;\pen}(\bar{R}_{\lag;\pen})\Y_{\ell,m} \otimes \Y_{\ell',m'})^2.
\end{equation}
Thus, we consider the following quantity
\begin{align*}
    &F'_{ \lag; \pen} (\bar{R}_{\lag;\pen}) g = F'_{ \lag; \pen} (\bar{R}_{\lag;\pen})g - \bar{F}'_{  \lag; \pen} (\bar{R}_{\lag;\pen}) g\\
    =& -\frac{2(4\pi)^2}{\num} \sum_{t=1}^{n-\lag} \sumjk (\W_{t+\lag,j}\W_{tk} - \bar{R}_{\lag;\pen}(U_{t+\lag,j}, U_{tk}))g(U_{t+\lag,j},U_{tk}) + 2 \langle R_{\lag} - \bar{R}_{\lag;\pen}, g \rangle_\rangleLL,
\end{align*}
$g \in \Hprod$.
When $g=\Y_{\ell,m} \otimes \Y_{\ell',m'}$, we have
$$
\mathbb{E}[F'_{ \lag; \pen} (\bar{R}_{\lag;\pen}) \Y_{\ell,m} \otimes \Y_{\ell',m'}] = \mathbb{E}_\U \mathbb{E}[F'_{ \lag; \pen} (\bar{R}_{\lag;\pen})\Y_{\ell,m} \otimes \Y_{\ell',m'}|\U] = 0
$$
and
\begin{align*}
    &\sum_{m,m'} \mathbb{E}|F'_{ \lag; \pen} (\bar{R}_{\lag;\pen})\Y_{\ell,m} \otimes \Y_{\ell',m'}|^2\\ =& \sum_{m,m'} \operatorname{Var}[F'_{ \lag; \pen} (\bar{R}_{\lag;\pen})\Y_{\ell,m}\otimes \Y_{\ell',m'}]\\
    =&  \frac{4(4\pi)^4}{\num^2} \sum_{m,m'} \operatorname{Var}\left[\sum_{t=1}^{n-\lag} \sumjk (\W_{t+\lag,j}\W_{tk} - \bar{R}_{\lag;\pen}(U_{t+\lag,j}, U_{tk}))\Y_{\ell,m}(U_{t+\lag,j})\Y_{\ell',m'}(U_{tk}) \right] .
\end{align*}

Using the law of total variance, we can write 
\begin{align}
&\sum_{m,m'} \operatorname{Var}\left[\sum_{t=1}^{n-\lag} \sumjk (\W_{t+\lag,j}\W_{tk} - \bar{R}_{\lag;\pen}(U_{t+\lag,j}, U_{tk}))\Y_{\ell,m}(U_{t+\lag,j})\Y_{\ell',m'}(U_{tk}) \right] \notag \\=& \sum_{m,m'} \operatorname{Var} \left [ \sum_{t=1}^{n-\lag} \sumjk (R_{\lag}(U_{t+\lag,j},U_{tk}) - \bar{R}_{\lag;\pen}(U_{t+\lag,j}, U_{tk}))\Y_{\ell,m}(U_{t+\lag,j})\Y_{\ell',m'}(U_{tk}) \right] \notag \\+& \sum_{m,m'} \mathbb{E}_\U\!\left[ \operatorname{Var}\left [\sum_{t=1}^{n-\lag} \sumjk \W_{t+\lag,j}\W_{tk} \Y_{\ell,m}(U_{t+\lag,j})\Y_{\ell',m'}(U_{tk})\Bigg| \U \right]\right]. \label{eq::ltv-2}
\end{align}

In what follows, we will handle sums over four indices $j,k,j',k'$.
It is then useful to identify the distinct cases which lead to terms of different orders, that is,
\begin{enumerate}
\item terms of order $r$:
\begin{enumerate}
    \item $j=j'=k=k'$
    \end{enumerate}
\item terms of order $r(r-1)$:
\begin{enumerate}
    \item $j=j'\ne  k=k'$
    \item $j=k'\ne j'=k$
    \item $j=k \ne  j'= k'$
    \item $j=k=j'\ne k'$
    \item $j=k=k' \ne j'$
    \item $j'=k'=j \ne k$
    \item $j'=k'=k \ne j$
    \end{enumerate}
\item terms of order $r(r-1)(r-2)$:
\begin{enumerate}
    \item $j=j', \ k\ne k', \ j \ne k, \ j' \ne k'$
    \item $j\ne j', \ k=k', \ j \ne k, \ j' \ne k'$
    \item $j\ne j', \ k\ne k', \ j=k', \ j'\ne k, \ j \ne k, \ j' \ne k'$
    \item $j\ne j', \ k\ne k', \ j\ne k', \ j'= k, \ j \ne k, \ j' \ne k'$
    \item $j\ne j', \ k\ne k', \ j\ne k', \ j'\ne k, \ j = k, \ j' \ne k'$
    \item $j\ne j', \ k\ne k', \ j\ne k', \ j'\ne k, \ j \ne k, \ j' = k'$
\end{enumerate}
\item terms of order $r(r-1)(r-2)(r-3)$:
\begin{enumerate}
    \item $j\ne j', \ k\ne k', \ j\ne k', \ j'\ne k, \ j \ne k, \ j' \ne k'$
\end{enumerate}
\end{enumerate}
Recall that, when $h=0$, we have $j\ne k,\ j' \ne k'$. Thus, cases 1.a, 2.c--2.g, 3.e and 3.f have not to be considered.

For the first term on the right hand side of Equation \eqref{eq::ltv-2}, when $h=0$, the sum over $i$ is composed of independent terms, hence it follows exactly as in Theorem \ref{th::R} that
\begin{align*}
&\sum_{m,m'} \operatorname{Var} \left [ \sum_{t=1}^{n} \sum_{1 \le j \ne k \le r} (R_0(U_{tj},U_{tk}) - \bar{R}_{\lag;\pen}(U_{tj}, U_{tk}))\Y_{\ell,m}(U_{tj})\Y_{\ell',m'}(U_{tk}) \right]\\
 = & n \sum_{m,m'} \operatorname{Var} \left [\sum_{1 \le j \ne k \le r} (R_0(U_{1j},U_{1k}) - \bar{R}_{\lag;\pen}(U_{1j}, U_{1k}))\Y_{\ell,m}(U_{1j})\Y_{\ell',m'}(U_{1k}) \right]\\ 
\le&  (2\ell+1)(2\ell'+1) \,O(n r^3).
\end{align*}

When $\lag \ne 0$, for the last case 4.a, we have 
$$
\sum_{m,m'} \langle R_{\lag}- \bar{R}_{\lag;\pen}, \Y_{\ell,m} \otimes \Y_{\ell',m'} \rangle^2_\rangleLL,
$$
regardless of $t$ and $t'$ (cardinality $(n-\lag)^2$).
The same for the other cases when jointly $t\ne t', t\ne t'+\lag,$ $t'\ne t+\lag$ (cardinality $(n-\lag)^2 - (n-\lag) - 2(n-2\lag) \le (n-\lag)^2$ ).
Moreover, we have
\begin{itemize}
\item terms of order $(n-\lag)r ,\ (n-\lag)r(r-1) , \ (n-\lag)r(r-1)(r-2)$, when $t=t'$ 
\item terms of order $(n-2\lag)r ,\ (n-2\lag)r(r-1) , \ (n-2\lag)r(r-1)(r-2)$, when $t=t'+\lag$
\item terms of order $(n-2\lag)r ,\ (n-2\lag)r(r-1) , \ (n-2\lag)r(r-1)(r-2)$, when $t'=t+\lag$
\end{itemize}
Hence,
\begin{align*}
& \sum_{m,m'} \operatorname{Var} \left [ \sum_{t=1}^{n-\lag} \sum_{j=1}^r \sum_{k=1}^r (R_{\lag}(U_{t+\lag,j},U_{tk}) - \bar{R}_{\lag;\pen}(U_{t+\lag,j}, U_{tk}))\Y_{\ell,m}(U_{t+\lag,j})\Y_{\ell',m'}(U_{tk}) \right]\\ 
= & \sum_{m,m'} \mathbb{E}\left ( \sum_{t=1}^{n-\lag} \sum_{j=1}^r \sum_{k=1}^r (R_{\lag}(U_{t+\lag,j},U_{tk}) - \bar{R}_{\lag;\pen}(U_{t+\lag,j}, U_{tk}))\Y_{\ell,m}(U_{t+\lag,j})\Y_{\ell',m'}(U_{tk})  \right)^2  \\
-& \frac{(n-\lag)^2 r^4}{(4\pi)^4} \sum_{m,m'} \langle R_{\lag}- \bar{R}_{\lag;\pen}, \Y_{\ell,m} \otimes \Y_{\ell',m'} \rangle^2_\rangleLL\\
\le& B \frac{3(n-\lag)(r + 7r(r-1)  + 6r(r-1)(r-2))}{(4\pi)^2}(2\ell+1)(2\ell'+1)\|R_{\lag} - \bar{R}_{\lag;\pen}\|_\rangleHH^2  \\
+&  \frac{(n-\lag)^2 r (r-1)(r-2)(r-3)}{(4\pi)^4} \sum_{m,m'} \langle R_{\lag}- \bar{R}_{\lag;\pen}, \Y_{\ell,m} \otimes \Y_{\ell',m'} \rangle^2_\rangleLL\\
+&  \frac{(n-\lag)^2  (r + 7r(r-1) + 6r(r-1)(r-2))}{(4\pi)^4} \sum_{m,m'} \langle R_{\lag}- \bar{R}_{\lag;\pen}, \Y_{\ell,m} \otimes \Y_{\ell',m'} \rangle^2_\rangleLL\\
-&  \frac{(n-\lag)^2r^4}{(4\pi)^4} \sum_{m,m'} \langle R_{\lag}- \bar{R}_{\lag;\pen}, \Y_{\ell,m} \otimes \Y_{\ell',m'} \rangle^2_\rangleLL\\
\le&  (2\ell+1)(2\ell'+1) \,O(n r^3),
\end{align*}
where the last two inequalities are justified by Lemma \ref{lemma::sup}, for some $B>0$, and Equation \eqref{eq::Rbar-alphabound-time} with $\alpha=1$. 
For the second term, write
\begin{align*}
\sum_{t=1}^{n-\lag} \sumjk \W_{t+\lag,j}\W_{tk} \Y_{\ell,m}(U_{t+\lag,j})\Y_{\ell',m'}(U_{tk}) &= \sum_{t=1}^{n-\lag} \sumjk X_{t+\lag}(U_{t+\lag,j})X_t(U_{tk}) \Y_{\ell,m}(U_{t+\lag,j})\Y_{\ell',m'}(U_{tk}) \\
&+\sum_{t=1}^{n-\lag} \sumjk \epsilon_{t+\lag,j} X_t(U_{tk}) \Y_{\ell,m}(U_{t+\lag,j})\Y_{\ell',m'}(U_{tk}) \\
&+ \sum_{t=1}^{n-\lag} \sumjk X_{t+\lag}(U_{t+\lag,j})\epsilon_{tk}\Y_{\ell,m}(U_{t+\lag,j})\Y_{\ell',m'}(U_{tk})\\
&+ \sum_{t=1}^{n-\lag} \sumjk \epsilon_{t+\lag,j} \epsilon_{tk} \Y_{\ell,m}(U_{t+\lag,j})\Y_{\ell',m'}(U_{tk}).
\end{align*}
Denote the four terms in this last expression by $S_1,S_2,S_3$, and $S_4$ with indices corresponding to their location in the sum. Then,
\begin{align*}
&\mathbb{E}_\U\!\left[ \operatorname{Var}\left [S_1+S_2+S_3+S_4 | \U \right]\right] \\\le& 4 \left( \mathbb{E}_\U[ \operatorname{Var}[S_1 | \U]] + \mathbb{E}_\U[ \operatorname{Var}[S_2 | \U]] + \mathbb{E}_\U[ \operatorname{Var}[S_3 | \U]] + \mathbb{E}_\U[ \operatorname{Var}[S_4 | \U]] \right).
\end{align*}
We will illustrate how to derive the bound for $S_1$, since the other three are somewhat simpler to handle and of order at most $n r^3$.


We start by considering the quantity
$$
\mathbb{E}_\U [\Y_{\ell,m}(U_{t+\lag,j})\Y_{\ell',m'}(U_{tk})\Y_{\ell,m}(U_{t'+\lag,j'})\Y_{\ell',m'}(U_{t'k'}) \operatorname{Cov}[ X_{t+\lag}(U_{t+\lag,j}) X_t(U_{tk}), X_{t'+\lag}(U_{t'+\lag,j'}) X_{t'}(U_{t'k'}) | \U ]]
$$
and the forms it takes for different configurations of indices $t,t',j,j',k,k'$.

For 4.a, regardless of $t$ and $t'$, and for 1-2-3 when jointly $t\ne t'$, $t\ne t'+\lag$, $t'\ne t + \lag$, we have
\begin{align*}
&\mathbb{E}_\U \left[\Y_{\ell,m}(U_{t+\lag,j})\Y_{\ell',m'}(U_{tk})\Y_{\ell,m}(U_{t'+\lag,j'})\Y_{\ell',m'}(U_{t'k'}) \operatorname{Cov}\left[ X_{t+\lag}(U_{t+\lag,j})X_t (U_{tk}),X_{t'+\lag} (U_{t'+\lag,j'}) X_{t'}(U_{t'k'}) | \U \right] \right]\\
=&\operatorname{Cov}[\langle X_{t+\lag}, \Y_{\ell,m} \rangle_\rangleL \langle X_{t}, \Y_{\ell',m'} \rangle_\rangleL , \langle X_{t'+\lag}, \Y_{\ell,m} \rangle_\rangleL \langle X_{t'}, \Y_{\ell',m'} \rangle_\rangleL ].
\end{align*}
It is possible to show that for non-centered (real-valued) random variables $X_1,X_2,X_3,X_4$, with means $\mu_1,\mu_2,\mu_3,\mu_4$, the following holds
\begin{align*}
\operatorname{Cov}[X_1X_2, X_3X_4] &= \operatorname{Cum}[X_1,X_2,X_3,X_4] \\
&+ \mu_4 \operatorname{Cum}[X_3,X_1,X_2] + \mu_3 \operatorname{Cum}[X_4,X_1,X_2]\\
&+\mu_2 \operatorname{Cum}[X_1,X_3,X_4]
+\mu_1 \operatorname{Cum}[X_2,X_3,X_4]\\
&+ \operatorname{Cov}[X_1,X_3]\operatorname{Cov}[X_2,X_4]\\
&+ \operatorname{Cov}[X_1,X_4]\operatorname{Cov}[X_2,X_3]\\
&+ \mu_2 \mu_4 \operatorname{Cov}[X_1,X_3] + \mu_1\mu_3 \operatorname{Cov}[X_2,X_4]\\
&+ \mu_2\mu_3 \operatorname{Cov}[X_1,X_4] + \mu_1\mu_4 \operatorname{Cov}[X_2,X_3].
\end{align*}
Hence, we have
\begin{align*}
&\operatorname{Cov}[\langle X_{t+\lag}, \Y_{\ell,m} \rangle_\rangleL \langle X_{t}, \Y_{\ell',m'} \rangle_\rangleL , \langle X_{t'+\lag}, \Y_{\ell,m} \rangle_\rangleL \langle X_{t'}, \Y_{\ell',m'} \rangle_\rangleL ]\\
=&    \operatorname{Cum}[\langle X_{t+\lag}, \Y_{\ell,m} \rangle_\rangleL , \langle X_{t}, \Y_{\ell',m'} \rangle_\rangleL , \langle X_{t'+\lag}, \Y_{\ell,m} \rangle_\rangleL ,\langle X_{t'}, \Y_{\ell',m'} \rangle_\rangleL ] \\
+&  \langle \mu, \Y_{\ell',m'} \rangle_\rangleL  \operatorname{Cum}[\langle X_{t'+\lag}, \Y_{\ell,m} \rangle_\rangleL, \langle X_{t+\lag}, \Y_{\ell,m} \rangle_\rangleL , \langle X_{t}, \Y_{\ell',m'} \rangle_\rangleL  ] \\
+& \langle \mu, \Y_{\ell,m} \rangle_\rangleL  \operatorname{Cum}[\langle X_{t'}, \Y_{\ell',m'} \rangle_\rangleL, \langle X_{t+\lag}, \Y_{\ell,m} \rangle_\rangleL , \langle X_{t}, \Y_{\ell',m'} \rangle_\rangleL  ]\\
+&  \langle \mu, \Y_{\ell',m'} \rangle_\rangleL  \operatorname{Cum}[\langle X_{t+\lag}, \Y_{\ell,m} \rangle_\rangleL  , \langle X_{t'+\lag}, \Y_{\ell,m} \rangle_\rangleL ,\langle X_{t'}, \Y_{\ell',m'} \rangle_\rangleL ] \\
+& \langle \mu, \Y_{\ell,m} \rangle_\rangleL  \operatorname{Cum}[ \langle X_{t}, \Y_{\ell',m'} \rangle_\rangleL , \langle X_{t'+\lag}, \Y_{\ell,m}, \rangle_\rangleL ,\langle X_{t'}, \Y_{\ell',m'} \rangle_\rangleL  ]\\
+&  \operatorname{Cov}[\langle X_{t+\lag}, \Y_{\ell,m} \rangle_\rangleL , \langle X_{t'+\lag}, \Y_{\ell,m} \rangle_\rangleL]  \operatorname{Cov}[\langle X_{t}, \Y_{\ell',m'} \rangle_\rangleL, \langle X_{t'}, \Y_{\ell',m'} \rangle_\rangleL ] \\
+&  \operatorname{Cov}[\langle X_{t+\lag}, \Y_{\ell,m} \rangle_\rangleL , \langle X_{t'}, \Y_{\ell',m'} \rangle_\rangleL]  \operatorname{Cov}[\langle X_{t}, \Y_{\ell',m'} \rangle_\rangleL, \langle X_{t'+\lag}, \Y_{\ell,m} \rangle_\rangleL ] \\
+& \langle \mu, \Y_{\ell',m'} \rangle_\rangleL^2  \operatorname{Cov}[\langle X_{t+\lag}, \Y_{\ell,m} \rangle_\rangleL , \langle X_{t'+\lag}, \Y_{\ell,m} \rangle_\rangleL]    \\
+& \langle \mu, \Y_{\ell,m} \rangle_\rangleL^2  \operatorname{Cov}[\langle X_{t}, \Y_{\ell',m'} \rangle_\rangleL, \langle X_{t'}, \Y_{\ell',m'} \rangle_\rangleL ] \\
+& \langle \mu, \Y_{\ell,m} \rangle_\rangleL  \langle \mu, \Y_{\ell',m'} \rangle_\rangleL  \operatorname{Cov}[\langle X_{t+\lag}, \Y_{\ell,m} \rangle_\rangleL , \langle X_{t'}, \Y_{\ell',m'} \rangle_\rangleL] \\
+&\langle \mu, \Y_{\ell,m} \rangle_\rangleL  \langle \mu, \Y_{\ell',m'} \rangle_\rangleL  \operatorname{Cov}[\langle X_{t}, \Y_{\ell',m'} \rangle_\rangleL, \langle X_{t'+\lag}, \Y_{\ell,m} \rangle_\rangleL ].
\end{align*}
In the following, we show how to bound one of the terms where third-order cumulant appears. We can observe that, by stationarity,
\begin{align*}
    &\frac{1}{n-\lag} \sum_{t=1}^{n-\lag} \sum_{t'=1}^{n-\lag}| \langle \mu, \Y_{\ell',m'} \rangle_\rangleL  \operatorname{Cum}[\langle X_{t+\lag}, \Y_{\ell,m} \rangle_\rangleL  , \langle X_{t'+\lag}, \Y_{\ell,m} \rangle_\rangleL ,\langle X_{t'}, \Y_{\ell',m'} \rangle_\rangleL ]|\\
    =& |\langle \mu, \Y_{\ell',m'} \rangle_\rangleL | \sum_{|\xi| < n-\lag} \left (1-\frac{|\xi|}{n-\lag}\right)|
    \langle \CL_{\xi+\lag,\lag} \Y_{\ell,m} \otimes \Y_{\ell',m'}, \Y_{\ell,m} \rangle_\rangleL|
    \\ \le& |\langle \mu, \Y_{\ell',m'} \rangle_\rangleL|  \sum_{\xi \in \mathbb{Z}} | \langle \CL_{\xi+\lag,\lag} \Y_{\ell,m} \otimes \Y_{\ell',m'}, \Y_{\ell,m}  \rangle_\rangleL |.
\end{align*}
Then, for $\alpha \le p/q$ and using the property in Equation \eqref{eq::covops}, we obtain
\begin{align*}
   &\sum_{\xi \in \mathbb{Z}}  \sum_{\ell,m} \sum_{\ell',m'} \frac{D_\ell^{2\alpha} D_{\ell'}^{2\alpha}}{(1+ \pen D_\ell^2 D_{\ell'}^2)^2} |\langle \mu, \Y_{\ell',m'} \rangle_\rangleL|  | \langle \CL_{\xi+\lag,\lag} \Y_{\ell,m} \otimes \Y_{\ell',m'}, \Y_{\ell,m}  \rangle_\rangleL |\\
   \le& \sum_{\xi\in \mathbb{Z}}  \sum_{\ell,m} \sum_{\ell',m'} D_\ell^{2\alpha} D_{\ell'}^{2\alpha} |\langle \mu, \Y_{\ell',m'} \rangle_\rangleL|  | \langle\CL_{\xi+\lag,\lag} \Y_{\ell,m} \otimes \Y_{\ell',m'}, \Y_{\ell,m}  \rangle_\rangleL |\\ 
   \le& \|\mu\|_{\Hq} \sum_{\xi \in \mathbb{Z}} \|\CHq_{\xi+\lag,\lag}\|_{\operatorname{TR},\Hqprod}
\end{align*}
(possibly up to an arbitrary constant). The last inequality is justified by the fact that, if  $\CL_{h_1,h_2}$ is trace class (and hence compact), it has singular value decomposition
$$
\CL_{h_1,h_2}g = \sum_{j} \lambda_j \langle g, f_{1j} \rangle_\Hqprod  f_{2j}, \qquad g \in \Hqprod,
$$
see \cite[Theorem 4.3.1]{Hsing}. Then,
\begin{align*}
  &\sum_{\ell,m} \sum_{\ell',m'} |\langle \mu, Y_{\ell'm'} \rangle_\Hq||\langle \CL_{h_1,h_2} Y_{\ell,m} \otimes Y_{\ell',m'}, Y_{\ell,m} \rangle_\Hq| \\ 
    =& \sum_j \lambda_j \sum_{\ell,m} \sum_{\ell',m'} |\langle \mu, Y_{\ell'm'} \rangle_\Hq| |\langle Y_{\ell,m}\otimes Y_{\ell',m'}, f_{1j} \rangle_\Hqprod  | |\langle f_{2j}, Y_{\ell,m} \rangle_\Hq| \\
    \le& \sum_j \lambda_j \left ( \sum_{\ell,m} \sum_{\ell',m'} |\langle \mu, Y_{\ell'm'} \rangle_\Hq|^2 |\langle f_{2j}, Y_{\ell,m} \rangle_\Hq|^2 \right)^{1/2} \| f_{1j}\|_\Hqprod\\
    =& \|\mu\|_\Hq \sum_j \lambda_j \|f_{2j}\|_\Hq \| f_{1j}\|_\Hqprod.
 \end{align*}
Since $\{f_{1j}\}$ and $\{f_{2j}\}$ are orthonormal, we have 
\begin{align*}
  \sum_{\ell,m} \sum_{\ell',m'} |\langle \mu, Y_{\ell'm'} \rangle_\Hq||\langle \CL_{h_1,h_2} Y_{\ell,m} \otimes Y_{\ell',m'}, Y_{\ell,m} \rangle_\Hq| \le  \|\mu\|_\Hq \|\CL_{h_1,h_2}\|_{\operatorname{TR}, \Hqprod} .
 \end{align*}
Similarly it holds for the other terms in the sum.

Now, following Proof of Theorem \ref{th::R} in Section \ref{sec:proofs3}, we can bound the terms corresponding to cases 1-2-3 when $t= t'$, $t= t'+\lag$, $t'= t + \lag$. Indeed, for instance, when $t =t'+\lag$ for $h \ne 0$ and $j'=k$,
\begin{align*}
&\mathbb{E}_\U \left[\Y_{\ell,m}(U_{t+\lag,j})\Y_{\ell',m'}(U_{tk})\Y_{\ell,m}(U_{t'+\lag,j'})\Y_{\ell',m'}(U_{t'k'}) \operatorname{Cov}\left[ X_{t+\lag}(U_{t+\lag,j})X_t (U_{tk}),X_{t'+\lag} (U_{t'+\lag,j'}) X_{t'}(U_{t'k'}) | \U \right] \right]\\
=& \int_{\mathbb{S}^2} \int_{\mathbb{S}^2} \int_{\mathbb{S}^2} \Y_{\ell,m}(u) \Y_{\ell',m'}(w) \Y_{\ell,m}(w) \Y_{\ell',m'}(v) \operatorname{Cov}\left[ X_{t+\lag}(u)X_t(w), X_{t}(w)X_{t-\lag}(v) \right] du dv dw ,
\end{align*}
and
\begin{align*}
    |\operatorname{Cov}[X_{t+\lag}(u) X_t(w), X_t(w) X_{t-\lag}(v)]| &\le  \left ( \operatorname{Var} \left [ X_{t+\lag}(u)  X_t(w) \right] \right)^{1/2} \left( \operatorname{Var}\left[  X_t(w) X_{t-\lag}(v) \right] \right)^{1/2}\\
    &\le \left ( \mathbb{E}[ | X_{t+\lag}(u)|^2  |X_t(w)|^2 ] \right)^{1/2}   \left (  \mathbb{E}[|X_t(w)|^2 |X_{t-\lag}(u)|^2 ] \right)^{1/2} \\
    &\le \left (\mathbb{E}  |X_{t+\lag} (u)|^4 \right)^{1/4} \left (\mathbb{E}|X_t(w)|^4 \right) ^{1/2} \left ( \mathbb{E} |X_{t-\lag}(v)|^4 \right)^{1/4}.
\end{align*}
Thus,
\begin{align*}
&\left |\sum_{m,m'} \int_{\mathbb{S}^2} \int_{\mathbb{S}^2} \int_{\mathbb{S}^2} \operatorname{Cov}[X_{t+\lag}(u) X_t(w), X_t(w) X_{t-\lag}(v)] \Y_{\ell,m} (u) \Y_{\ell',m'} (w) \Y_{\ell,m} (w)  \Y_{\ell',m'} (v)  dudvdw \right|\\
\le&  \int_{\mathbb{S}^2} \int_{\mathbb{S}^2} \int_{\mathbb{S}^2} \left | \operatorname{Cov} \left [ X_{t+\lag}(u) \, X_t(w), X_t(w) X_{t-\lag}(v)  \right] \right | |\Y_{\ell,m} (u)| |\Y_{\ell',m'} (w)| |\Y_{\ell,m} (w)|  |\Y_{\ell',m'} (v)|  dudvdw \\
   \le& B' (4\pi) (2\ell+1)(2\ell'+1) \mathbb{E}\|X_0\|_\Hq^4 ,
\end{align*}
for all $t \in \mathbb{Z}$, again by applying Lemma \ref{lemma::sup}, for some $B'>0$.
In conclusion, for $\lag =0$ we have
\begin{align*}
    &\sum_{m,m'} \mathbb{E}_\U\!\left[ \operatorname{Var}\left [\sum_{t=1}^n \sum_{1 \le j \ne k \le r}  X_t(U_{tj})X_t(U_{tk}) \Y_{\ell,m}(U_{tj})\Y_{\ell',m'}(U_{tk})\Bigg| \U \right]\right]\\
   \le& \frac{nr^2(r-1)^2}{(4\pi)^4} \Bigg \{ \sum_{\xi_1,\xi_2,\xi_3 \in \mathbb{Z}} \|\CHq_{\xi_1,\xi_2,\xi_3}\|_{\operatorname{TR},\Hqprod} + 4\|\mu\|_\Hq \sum_{\xi_1,\xi_2 \in \mathbb{Z}} \|\CHq_{\xi_1,\xi_2}\|_{\operatorname{TR},\Hqprod } \\
   +&  2 \left (\sum_{\xi \in \mathbb{Z}} \|\CHq_\xi \|_{\operatorname{TR},\Hq}\right)^2  + 4 \|\mu\|^2_\Hq \sum_{\xi \in \mathbb{Z}} \|\CHq_\xi\|_{\operatorname{TR},\Hq} \Bigg\}\\
   +& B' \frac{n(2r(r-1) + 4r(r-1)(r-2))}{(4\pi)^2}(2\ell+1)(2\ell'+1) \mathbb{E}\|X_0\|^4_\Hq;
\end{align*}
while, for $\lag\ne0$,
\begin{align*}
&\sum_{m,m'} \mathbb{E}_\U\!\left[ \operatorname{Var}\left [\sum_{t=1}^{n-\lag} \sum_{j=1}^r \sum_{k=1}^r X_{t+\lag}(U_{t+\lag,j}) X_t(U_{tk}) \Y_{\ell,m}(U_{t+\lag,j})\Y_{\ell',m'}(U_{tk})\Bigg| \U \right]\right]\\
       \le&\frac{(n-\lag) r^4}{(4\pi)^4} \Bigg \{ \sum_{\xi_1,\xi_2,\xi_3 \in \mathbb{Z}} \|\CHq_{\xi_1,\xi_2,\xi_3}\|_{\operatorname{TR},\Hqprod} + 4\|\mu\|_\Hq \sum_{\xi_1,\xi_2 \in \mathbb{Z}} \|\CHq_{\xi_1,\xi_2}\|_{\operatorname{TR},\Hqprod } \\
   +&  2 \left (\sum_{\xi \in \mathbb{Z}} \|\CHq_\xi \|_{\operatorname{TR},\Hq}\right)^2  + 4 \|\mu\|^2_\Hq \sum_{\xi \in \mathbb{Z}} \|\CHq_\xi\|_{\operatorname{TR},\Hq} \Bigg\}\\
   +&B' \frac{3(n-\lag)(r + 7r(r-1) + 6r(r-1)(r-2))}{(4\pi)^2}(2\ell+1)(2\ell'+1) \mathbb{E}\|X_0\|^4_\Hq.
\end{align*}

Now, let us prove 3.
The first step is to obtain a useful analytic form for $R_{\lag;\pen} - \Tilde{R}_{\lag;\pen}$, by observing that
\begin{align*}
R_{\lag;\pen} - \Tilde{R}_{\lag;\pen} &= R_{\lag;\pen} - \bar{R}_{\lag;\pen} +  (\bar{F}''_{  \lag;\pen})^{-1}F'_{ \lag;\pen}(\bar{R}_{\lag;\pen}) \\
&= (\bar{F}''_{  \lag;\pen})^{-1} \left [\bar{F}''_{  \lag;\pen}(R_{\lag;\pen} - \bar{R}_{\lag;\pen}) + F'_{ \lag;\pen}(\bar{R}_{\lag;\pen})\right].
\end{align*}
Since $R_{\lag;\pen}$ minimizes $F'_{ \lag;\pen}$, it holds $F'_{ \lag;\pen}(R_{\lag;\pen}) = 0$ (see \cite[Theorem 3.6.3]{Hsing}).
Then, for any $g \in \Hprod$,
\begin{align*}
     \left [\bar{F}''_{  \lag;\pen}(R_{\lag;\pen} - \bar{R}_{\lag;\pen}) + F'_{ \lag;\pen}(\bar{R}_{\lag;\pen})\right] g&=  \left [\bar{F}''_{  \lag;\pen}(R_{\lag;\pen} - \bar{R}_{\lag;\pen}) + F'_{ \lag;\pen}(\bar{R}_{\lag;\pen}) - F'_{ \lag;\pen}(R_{\lag;\pen})\right]g\\
    &=\left[\bar{F}''_{  \lag; 0}(R_{\lag;\pen} - \bar{R}_{\lag;\pen}) - F''_{ \lag; 0}(R_{\lag;\pen} - \bar{R}_{\lag;\pen})  \right]g.
\end{align*}
where we used Lemma \ref{lemma::1-cov}; in other words,
$$
R_{\lag;\pen} - \Tilde{R}_{\lag;\pen}  = (\bar{F}''_{  \lag;\pen})^{-1}  \left[\bar{F}''_{  \lag; 0}(R_{\lag;\pen} - \bar{R}_{\lag;\pen}) - F''_{ \lag; 0}(R_{\lag;\pen} - \bar{R}_{\lag;\pen})  \right].
$$
Now, the same argument that leads to \eqref{eq::tilde-bar-cov-time2} gives us
$$
\| R_{\lag;\pen} - \Tilde{R}_{\lag;\pen} \|^2_\alpha = \frac14 \sum_{\ell,m} \sum_{\ell',m'} \frac{D_\ell^{2\alpha} D_{\ell'}^{2\alpha}}{(1+ \pen D_\ell^2D_{\ell'}^2)^2} \left ( \left[\bar{F}''_{  \lag; 0}(R_{\lag;\pen} - \bar{R}_{\lag;\pen}) - F''_{ \lag; 0}(R_{\lag;\pen} - \bar{R}_{\lag;\pen})  \right]\Y_{\ell,m} \otimes \Y_{\ell',m'}\right)^2,
$$
with
\begin{align*}
&\left[\bar{F}''_{  \lag; 0}(R_{\lag;\pen} - \bar{R}_{\lag;\pen}) - F''_{ \lag; 0}(R_{\lag;\pen} - \bar{R}_{\lag;\pen})  \right]\Y_{\ell,m} \otimes \Y_{\ell',m'} \\=& 2\langle R_{\lag;\pen} - \bar{R}_{\lag;\pen}, \Y_{\ell,m} \otimes \Y_{\ell',m'}\rangle_\rangleLL 
\\-& \frac{2(4\pi)^2}{\num} \sum_{t=1}^{n-\lag} \sumjk (R_{\lag;\pen}(U_{t+\lag,j}, U_{tk}) - \bar{R}_{\lag;\pen}(U_{t+\lag,j}, U_{tk})) \Y_{\ell,m}(U_{t+\lag,j}) \Y_{\ell',m'}(U_{tk}).
\end{align*}
Now, $R_{\lag;\pen} - \bar{R}_{\lag;\pen} = \sum_{\ell,m} \sum_{\ell',m'} h_{\ell,\ell',m,m'} \Y_{\ell,m} \Y_{\ell',m'}$. Then,
\begin{align*}
    \| R_{\lag;\pen} - \Tilde{R}_{\lag;\pen} \|^2_\alpha =  \sum_{\ell_1,m_1}  \sum_{\ell_2,m_2} \frac{D_{\ell_1}^{2\alpha}D_{\ell_2}^{2\alpha} }{(1+ \pen D_{\ell_1}^2 D_{\ell_2}^2)^2} \left ( \sum_{\ell_3,m_3} \sum_{\ell_4, m_4} h_{\ell_3,\ell_4,m_3,m_4} V_{\ell_{1:4},m_{1:4}} \right)^2,
\end{align*}
where $$
V_{\ell_{1:4},m_{1:4}} = \delta_{\ell_1}^{\ell_3} \delta_{m_1}^{m_3} \delta_{\ell_2}^{\ell_4} \delta_{m_2}^{m_4} - \frac{(4\pi)^2}{\num} \sum_{t=1}^{n-\lag} \sumjk \Y_{\ell_1,m_1} (U_{t+\lag,j} ) \Y_{\ell_2,m_2} (U_{tk} ) \Y_{\ell_3,m_3} (U_{t+\lag,j} ) \Y_{\ell_4,m_4} (U_{tk} ).
$$
By applying the Cauchy–Schwarz inequality for arbitrary $\theta \in (1/p, 1]$, we obtain
$$
\| R_{\lag;\pen} - \Tilde{R}_{\lag;\pen} \|^2_\alpha \le \| R_\pen - \bar{R}_\pen \|^2_\theta \sum_{\ell_1,m_1} \sum_{\ell_2,m_2} \frac{D_{\ell_1}^{2\alpha}D_{\ell_2}^{2\alpha} }{(1+ \pen D_{\ell_1}^2 D_{\ell_2}^2)^2} \sum_{\ell_3,m_3} \sum_{\ell_4,m_4} D_{\ell_3}^{-2\theta} D_{\ell_4}^{-2\theta} V^2_{\ell_{1:4},m_{1:4}}.
$$
It is readily seen that $\mathbb{E}[V_{\ell_{1:4},m_{1:4}}]=0$ and 
\begin{align*}
\mathbb{E}[V^2_{\ell_{1:4}, m_{1:4}}]  
&= \frac{(4\pi)^4}{\num^2} \mathbb{E}\left( \sum_{t=1}^{n-\lag} \sumjk \Y_{\ell_1,m_1} (U_{t+\lag,j} ) \Y_{\ell_2,m_2} (U_{tk} ) \Y_{\ell_3,m_3} (U_{t+\lag,j} ) \Y_{\ell_4,m_4} (U_{tk} ) \right)^2 \\&- \delta_{\ell_1}^{\ell_3} \delta_{m_1}^{m_3} \delta_{\ell_2}^{\ell_4} \delta_{m_2}^{m_4}.
\end{align*}
Thus, it is possible to show that
$$
\sum_{m_1,m_2,m_4,m_4}\mathbb{E}[V^2_{\ell_{1:4}, m_{1:4}}] \le D_{\ell_1}^{2\theta} D_{\ell_2}^{2\theta} (2\ell_1+1)(2\ell_2 + 1)(2\ell_3 + 1)(2\ell_4 + 1)\, O\!\left ( \frac{1}{nr}\right),
$$
and hence
$$
\sum_{\ell_1,m_1} \sum_{\ell_2,m_2} \frac{D_{\ell_1}^{2\alpha}D_{\ell_2}^{2\alpha} }{(1+ \pen D_{\ell_1}^2 D_{\ell_2}^2)^2} \sum_{\ell_3,m_3} \sum_{\ell_4,m_4} D_{\ell_3}^{-2\theta} D_{\ell_4}^{-2\theta} \mathbb{E}[V^2_{\ell_{1:4},m_{1:4}}] = O\!\left (  \frac{\log(1/\pen)}{nr \pen^{\alpha+\theta+1/p}}\right) .
$$
The rest of the proof follows exactly as in Proof of Theorem \ref{th::R} in Section \ref{sec:proofs3}.
\end{proof}

\bibliographystyle{imsart-number}
\bibliography{mybiblio}
\end{document}